\documentclass[11pt,a4paper
]{article}
\usepackage{amsmath,amssymb,amsthm}
\usepackage[margin=30mm]{geometry}
\usepackage[linktocpage=true]{hyperref}
\usepackage{cleveref}   
\usepackage{array}
\usepackage{tikz}
\usetikzlibrary{arrows.meta,decorations.markings}

\usepackage[pagewise, mathlines]{lineno} 
\usepackage{time}

\makeatletter
\def\@part[#1]#2{%
    \ifnum \c@secnumdepth >\m@ne
      \refstepcounter{part}%
      \addcontentsline{toc}{part}{\thepart\hspace{1em}#1}%
    \else
      \addcontentsline{toc}{part}{#1}%
    \fi
    {\parindent \z@ \raggedright
     \interlinepenalty \@M
     \normalfont
     \ifnum \c@secnumdepth >\m@ne
       \large\bfseries \partname\nobreakspace\thepart
       \par\nobreak
     \fi
     \LARGE \bfseries #2%
     \markboth{}{}\par}%
    \nobreak
    \vskip 3ex
    \@afterheading}
\makeatother

\theoremstyle{plain}
\newtheorem{thm}{Theorem}[section]
\newtheorem{prop}[thm]{Proposition}
\newtheorem{lem}[thm]{Lemma}
\newtheorem{cor}[thm]{Corollary}
\theoremstyle{definition}
\newtheorem{dfn}[thm]{Definition}
\newtheorem{rem}[thm]{Remark}
\newtheorem{eg}[thm]{Example}
\newtheorem{q}[thm]{Question}
\theoremstyle{plain}
\newtheorem{thmA}{Theorem}

\newtheorem{corA}[thmA]{Corollary}
\numberwithin{equation}{section}

\newcommand{\bb}[1]{\mathbb{#1}}
\newcommand{\mca}[1]{\mathcal{#1}}
\newcommand{\R}{\bb{R}}
\newcommand{\Z}{\bb{Z}}
\newcommand{\Q}{\bb{Q}}
\newcommand{\C}{\bb{C}}

\newcommand{\ol}{\overline}
\newcommand{\wt}{\widetilde}
\newcommand{\surj}{\twoheadrightarrow}
\newcommand{\spmx}[1]{{\small\begin{pmatrix}#1\end{pmatrix}}}
\DeclareMathOperator{\tr}{tr}
\newcommand{\internal}[1]{\par\medskip\noindent
\fbox{\parbox{0.95\linewidth}{\small\textbf{Internal note, not for the
published version.}\\ #1}}\par\medskip}
\renewcommand{\internal}[1]{}  

\DeclareMathOperator{\PSL}{PSL}

\DeclareMathOperator{\intr}{int}

\title{\bfseries Modular knots obey the Chebotarev law\\ for three inequivalent orderings}
\author{\Large Jun Ueki} 
\date{}

\begin{document}
\maketitle

\begin{abstract}
The modular knots are the closed orbits of the geodesic flow on 
the unit tangent bundle of the modular orbifold, the complement of a trefoil in $S^{3}$.
We compare three orderings of them: by the length of the word in $L$ and $R$, 
by a generic length function, and by the trace. 
The three orderings are pairwise inequivalent: one set of modular knots has density $1$ for the first ordering and for the generic orderings close to it, and density $\frac{2}{\pi}\arctan\frac{\pi}{3}$ for the third.
We prove that the modular knots with the trefoil obey the Chebotarev law in the sense of Mazur and McMullen for the generic and for the trace ordering, 
and we determine exactly when the law holds for the combinatorial ordering.
The proofs adapt the argument of McMullen to cusps by Markov sections with an escape locus, 
and replace it by the transfer operator of Mayer for the trace ordering.
The law passes to finite branched covers.
\end{abstract}

{\renewcommand{\thefootnote}{}\footnotetext{\textit{Date:} \today \ \ \now \ (GMT=JST$-9$). }}

{\small \setcounter{tocdepth}{2}\tableofcontents}

\part{Statements, and why the ordering matters}

\section{Introduction} \label{sec.intro}

\subsection{The Chebotarev law} \label{subsec.intro.law}

Let $M$ be a 3-manifold, by which we always mean an oriented connected closed 3-manifold or the complement of a finite link in one;
the main case is $M=S^{3}$.

\begin{dfn} \label{dfn.Chebotarev}
Let $(K_i)_{i\in \Z_{\geq 1}}$ be a sequence of mutually disjoint oriented knots in $M$.
For each $n\in \Z_{\geq 1}$, put $L_n=\bigcup_{i\leq n}K_i$, and for each $j>n$,
let $[K_j]$ denote the conjugacy class of $K_j$ in $\pi_1(M-L_n)$.
The sequence $(K_i)_i$ is said to \emph{obey the Chebotarev law} if the equality
\[\lim_{\nu \to \infty}
\frac{\#\{n<j\leq \nu\mid \rho([K_j])=C\}}{\nu}=\frac{\#C}{\#G}\]
holds for every $n\in \Z_{\geq 1}$, every finite group $G$,
every surjective homomorphism $\rho:\pi_1(M-L_n)\surj G$, and every conjugacy class $C\subset G$.
A countable link $\mca{K}$ in $M$ is said to be \emph{Chebotarev} if some ordering of its components obeys the Chebotarev law.
\end{dfn}

The law depends on the ordering, since the limit above is a density with respect to the enumeration;
\Cref{thmA.combinatorial} and \Cref{thmA.densities} below show that the dependence is essential.

Mazur \cite{Mazur2025Po} asked whether $S^{3}$ admits a Chebotarev arrangement of knots.
McMullen \cite{McMullen2013CM} proved that the closed orbits of a topologically mixing pseudo-Anosov flow on a \emph{closed} 3-manifold,
ordered by length, obey the Chebotarev law, and he derived from that theorem an example in $S^{3}$.
The unramified case of the law, in which the homomorphism is defined on the fundamental group of the manifold itself,
was treated by Parry and Pollicott \cite{ParryPollicott1986ETDS} for Axiom A flows and by Adachi and Sunada \cite[Proposition C]{AdachiSunada1987JFA} for a topologically mixing Anosov flow on a compact manifold;
see also Sunada \cite[Proposition II-2-12]{Sunada1984ASPM}.

Two natural families of examples arise in $S^{3}$.
One is the planetary link of a fibered hyperbolic link $L$,
that is, the set of the closed orbits of the suspension flow of the monodromy, together with $L$ itself.
The planetary link of the figure-eight knot obeys the Chebotarev law for the ordering by the lengths in a generic metric \cite[Corollary 1.3]{McMullen2013CM},
and it contains every knot type, because the corresponding template is universal, by a theorem of Ghrist \cite{Ghrist1997}.

The other is the family of the modular knots,
that is, the closed orbits of the geodesic flow on $T^{1}(\PSL_2\Z\backslash \bb{H}^{2})$,
which is the complement of a trefoil in $S^{3}$.
We call that trefoil the \emph{missing trefoil}.
Those knots reflect the arithmetic of $\PSL_2\Z$,
and the linking numbers with the missing trefoil are the values of the Rademacher symbol.
They also have a second history: by a theorem of Ghys \cite{Ghys2007ICM} they are exactly the Lorenz knots,
that is, the closed orbits of a certain topological model of the Lorenz equations, which is called the Lorenz template,
and the Lorenz knots have been studied as knotted closed orbits of a flow since the work of Birman and Williams \cite{BirmanWilliams1983T}.
The family therefore belongs at once to the arithmetic of quadratic irrationals,
to the topology of the trefoil complement, and to the theory of chaotic flows.
This article is concerned with the second family.
Two features place it outside the theorem of McMullen \cite{McMullen2013CM}.
First, the manifold is not closed, and the modular flow is not orbit equivalent to the restriction of a flow on a closed 3-manifold to the complement of a closed orbit,
whereas the flow on the complement of the figure-eight knot is such a restriction.
Indeed the closed orbits that go deep into the cusp wind around it in both directions;
see \Cref{prop.not.closed.orbit}.
Second, the closed geodesics have a preferred ordering, by $|\tr\gamma|$,
and the hyperbolic length is not attached to a H\"older roof function on the finite symbolic model that we construct,
because the return time of the section is not bounded away from zero; see \Cref{subsec.cusp}.

\subsection{The three orderings} \label{subsec.intro.results}

Let $K_{2,3}\subset S^{3}$ denote the missing trefoil,
and let $C_\gamma$ denote the modular knot attached to a primitive hyperbolic $\gamma\in \PSL_2\Z$.
Every such $\gamma$ is conjugate to a positive word in $L=\spmx{1&0\\1&1}$ and $R=\spmx{1&1\\0&1}$, unique up to a cyclic permutation,
and $\ell_{\rm c}(\gamma)$ denotes the length of that word,
which is called the \emph{combinatorial length}; see \Cref{subsec.bg.rademacher}.
We consider three orderings of the modular knots: by $\ell_{\rm c}$,
by a length function $\ell_h$ attached to a generic roof function on a symbolic model of the flow,
and by the hyperbolic length $\ell(\gamma)=2\log\xi_\gamma$,
where $\xi_\gamma>1$ is determined by $\xi_\gamma+\xi_\gamma^{-1}=|\tr\gamma|$.
The last ordering is the ordering by $|\tr\gamma|$.
We call them the \emph{combinatorial}, the \emph{generic}, and the \emph{trace ordering}.
The trace ordering differs from the ordering by the discriminant $D$ of the primitive binary quadratic form attached to $\gamma$,
which is the ordering in the equidistribution theorem of Duke \cite{Duke1988Invent}:
if $(t_D,u_D)$ is the least positive solution of $t^{2}-Du^{2}=4$, then $|\tr\gamma|=t_D$, and $u_D$ is unbounded.
The closed geodesics of a fixed discriminant have the same trace,
so that the law for the ordering by $D$ is a question of a different nature.

\Cref{table.orderings} summarizes the results of this subsection.

\begin{table}[t]
\centering
\renewcommand{\arraystretch}{1.45}
\begin{tabular}{@{}>{\raggedright\arraybackslash}p{33mm}|>{\raggedright\arraybackslash}p{41mm}>{\raggedright\arraybackslash}p{30mm}>{\raggedright\arraybackslash}p{30mm}@{}}
\hline
ordering & combinatorial & generic & trace\\
\hline
length function & $\ell_{\rm c}$ & $\ell_h$ & $\ell=2\log\xi_\gamma$\\
Chebotarev law with $K_{2,3}$ & the parity of $\Psi$ is the only obstruction (\Cref{thmA.combinatorial}) & holds (\Cref{thmA.generic}) & holds (\Cref{thmA.arith})\\
density of $\{|\Psi|\leq\ell\}$ & $1$ & $1$ & $\frac{2}{\pi}\arctan\frac{\pi}{3}$\\
density of $\{\Psi\ \text{is even}\}$ & no density; lower $\frac13$, upper $\frac23$ & $\frac12$ & $\frac12$\\
main tool & Perron--Frobenius theory of finite matrices (\Cref{sec.combinatorial}) & theorem of Parry and Pollicott (\Cref{thm.PP}) & transfer operator of Mayer (\Cref{sec.arith})\\
\hline
\end{tabular}
\caption{The three orderings of the modular knots.
The two sets are those of \Cref{thmA.densities};
the density $1$ for the generic ordering is proved under the additional assumption on $h$ stated after \Cref{thmA.generic}.
For the combinatorial ordering, the knots are counted by the thresholds $\ell_{\rm c}\leq N$.}
\label{table.orderings}
\end{table}

The dependence on the ordering appears already in the first nontrivial case.
Let $\rho:\pi_1(S^{3}-K_{2,3})\surj \Z/2\Z$ denote the composition of the abelianization with the reduction modulo $2$.
The linking number of $C_\gamma$ with $K_{2,3}$ is the value $\Psi(\gamma)$ of the Rademacher symbol,
by a theorem of Ghys recalled in \Cref{thm.ghys},
and $\Psi(\gamma)=\#\{R\}-\#\{L\}$ on the cyclic word of $\gamma$,
so that $\rho([C_\gamma])$ is the parity of $\Psi(\gamma)$, which is the parity of $\ell_{\rm c}(\gamma)$.
For the combinatorial ordering the class is therefore determined by the length,
and the proportion of each class oscillates instead of converging to $1/2$; see \Cref{rem.exact.length}.
The first theorem says that the parity is the only obstruction.

\begin{thmA}[{Precise form in \Cref{thm.combinatorial}}] \label{thmA.combinatorial}
Let $L$ be the union of finitely many modular knots, together with $K_{2,3}$ or not,
let $G$ be a finite group, and let $\rho:\pi_1(S^{3}-L)\surj G$ be a surjective homomorphism.
Let $\gamma$ run through the primitive hyperbolic conjugacy classes of $\PSL_2\Z$ with $C_\gamma\not\subset L$.
Then the following two conditions are equivalent.
\begin{enumerate}
\item For every conjugacy class $C\subset G$,
\[\lim_{N\to\infty}
\frac{\#\{\gamma\mid \ell_{\rm c}(\gamma)\leq N,\ \rho([C_\gamma])=C\}}
{\#\{\gamma\mid \ell_{\rm c}(\gamma)\leq N\}}=\frac{\#C}{\#G}.\]
\item There is no homomorphism $\chi:G\to\{\pm1\}$ such that
$\chi(\rho([C_\gamma]))=(-1)^{\operatorname{lk}(C_\gamma,K_{2,3})}$ holds for all but finitely many $\gamma$.
\end{enumerate}
Such a $\chi$ can exist only if $L$ contains $K_{2,3}$.
In particular, the modular knots without the missing trefoil satisfy condition \textup{(1)} for every $L$, $\rho$, and $C$.
\end{thmA}

The theorem counts the modular knots by the threshold $\ell_{\rm c}\leq N$ and not by the position in a sequence as in \Cref{dfn.Chebotarev},
because many knots share the same combinatorial length:
the modular knots of combinatorial length $N$ are about half of those of combinatorial length at most $N$;
see \Cref{rem.combinatorial.ties}.
For the other two orderings the knots of equal length are negligible in number,
and the order among them is arbitrary; a choice of that order is called a \emph{tie-breaking}.
If condition (2) fails, then the law fails in a precise way:
the proportion of each class among the knots with $\ell_{\rm c}\leq N$ oscillates between $\frac23$ and $\frac43$ of the expected value;
see \Cref{thm.combinatorial}.
To keep the missing trefoil in the link, we change the ordering.

\begin{thmA}[{Precise form in \Cref{thm.modular}}] \label{thmA.generic}
Let $h$ be a positive H\"older function on the space of the two-sided sequences in $L$ and $R$,
let $\ell_h(\gamma)$ denote the sum of $h$ over the periodic orbit attached to $\gamma$,
and assume that the ratio $\ell_h(\gamma_1)/\ell_h(\gamma_2)$ is irrational for some $\gamma_1$ and $\gamma_2$.
Then the modular knots ordered by $\ell_h$, with $K_{2,3}$ inserted at an arbitrary fixed position
and with an arbitrary order among the knots of equal length, obey the Chebotarev law in $S^{3}$.
In particular, for every irrational number $\varepsilon$ with $|\varepsilon|<1$,
the law holds for the ordering by $\ell_{\rm c}+\varepsilon\Psi$,
and the set of modular knots together with $K_{2,3}$ is a Chebotarev link.
\end{thmA}

Every such $\ell_h$ satisfies $c_1\ell_{\rm c}\leq \ell_h\leq c_2\ell_{\rm c}$ with $c_1=\min h$ and $c_2=\max h$.
The assumption on the ratio holds for a generic $h$ in the sense of Baire,
and these orderings are therefore called \emph{generic orderings},
in accordance with the generic metrics of McMullen \cite[Corollary 1.3]{McMullen2013CM};
see \Cref{subsec.PP}.
In \Cref{thmA.densities} and \Cref{corA.linking} below,
the symbol $\ell_h$ denotes a fixed choice of such a function; in \Cref{thmA.densities} we assume moreover that $c_1/c_2>0.9709$.
The function $\ell_h$ is attached to a roof function on the symbolic model of the modular flow,
and not to a metric on $S^{3}-K_{2,3}$; see \Cref{sec.ordering}.
By \Cref{thmA.combinatorial}, the assumption on the ratio cannot be dropped as long as $K_{2,3}$ belongs to the link:
the combinatorial length is $\ell_h$ for $h=1$.

\begin{thmA}[{\Cref{thm.arith.general}}] \label{thmA.arith}
The modular knots, again with $K_{2,3}$ inserted at an arbitrary fixed position,
obey the Chebotarev law when they are ordered by the hyperbolic length $\ell(\gamma)=2\log \xi_\gamma$, equivalently by $|\tr \gamma|$, with an arbitrary order among the knots of equal trace.
\end{thmA}

If no modular knot has been removed, then $\rho$ is a homomorphism on $\pi_1(S^{3}-K_{2,3})$, whose quotient by the center is $\PSL_2\Z$.
In that case the statement of \Cref{thmA.arith} could also be derived from the Selberg trace formula for a unitary representation and a multiplier system of real weight,
in the form given by Hejhal \cite[Chapter 9]{Hejhal1983book},
after a spectral analysis of the kind that Mozzochi \cite{Mozzochi2013Israel} carried out for a character.
The argument of \Cref{sec.arith} does not use the spectral theory of automorphic forms,
and it applies after the removal of finitely many modular knots,
where $\rho$ no longer factors through the Fuchsian group.

The three orderings are pairwise inequivalent,
in the sense that some set of modular knots has different densities for any two of them,
and none of the three statements follows from another by a comparison of densities.
The lengths $\ell_{\rm c}$ and $\ell_h$ are comparable, and neither of them is comparable to $\ell$,
since a closed geodesic that goes deep into the cusp has a hyperbolic length of the order of the logarithm of its combinatorial length;
see \Cref{prop.incommensurable}.
The inequivalence is visible on explicit sets.

\begin{thmA}[{\Cref{thm.two.densities}, \Cref{prop.density.ellh}, and
\Cref{cor.three.densities}}] \label{thmA.densities}
Let $\gamma$ run through the primitive hyperbolic conjugacy classes of $\PSL_2\Z$, and let $\Psi$ denote the Rademacher symbol.
The set $\{C_\gamma \mid |\Psi(\gamma)|\leq \ell(\gamma)\}$ has density $1$ for the combinatorial and for the generic ordering,
and density $\frac{2}{\pi}\arctan\frac{\pi}{3}=0.5147\ldots$ for the trace ordering.
The set $\{C_\gamma\mid \Psi(\gamma)\ \text{is even}\}$ has density $\frac12$ for the generic and for the trace ordering,
and it has no density for the combinatorial ordering,
where its lower and upper densities are $\frac13$ and $\frac23$.
\end{thmA}

The first set separates the trace ordering from the two others,
and the second set separates the combinatorial ordering from the two others.
The densities for the combinatorial ordering are taken along the thresholds $\ell_{\rm c}\leq N$, as in \Cref{thmA.combinatorial}.

Applying the three theorems to the cyclic quotients of $\pi_1(S^{3}-L)$,
we obtain equidistribution statements for linking numbers.
By the theorem of Ghys, the linking number of $C_\gamma$ with $K_{2,3}$ is $\Psi(\gamma)$,
and by a theorem of Simon \cite[Proposition 7.1]{Simon2025GT} the linking number of $C_\gamma$ with a modular knot $C_\sigma$ is the value on the cyclic word of $\gamma$ of a function $F_\sigma$ that counts occurrences of patterns attached to $\sigma$;
see \eqref{eq.simon}.

\begin{corA}[{\Cref{cor.rademacher} and \Cref{cor.linking}}] \label{corA.linking}
Let $K$ be either $K_{2,3}$ or a modular knot $C_\sigma$, and let $m\in \Z_{\geq 1}$.
As $\gamma$ runs through the primitive hyperbolic conjugacy classes of $\PSL_2\Z$ with $C_\gamma\neq K$,
the residues of $\operatorname{lk}(C_\gamma,K)$ modulo $m$ are equidistributed for the generic and for the trace ordering,
and for the combinatorial ordering as well unless $K=K_{2,3}$ and $m$ is even.
\end{corA}

For $K=K_{2,3}$ and the trace ordering, \Cref{corA.linking} is a special case of a theorem of Burrin and von Essen \cite[Theorem 1.5]{BurrinVonEssen2024IMRN},
proved through the Selberg trace formula; see also \Cref{rem.burrin.vonessen}.
For $K=K_{2,3}$ and the combinatorial ordering with $m$ odd,
it follows from the binomial computation of \Cref{subsec.ord.laws}.
The case of the generic ordering, and the case $K=C_\sigma$ for all three orderings, are new.
For $K=C_\sigma$ the homomorphism $\rho$ does not factor through $\PSL_2\Z$,
so that the proofs through the Selberg trace formula do not apply, at least not directly.
Ghys \cite{Ghys2007ICM} asked for an interpretation of the linking numbers of two modular knots; Duke, Imamo\u{g}lu,
and T\'oth \cite{DukeImamogluToth2017Duke} interpreted the linking numbers of the symmetrized knots $C_\gamma\cup C_{\gamma^{-1}}$ through modular cocycles,
and Simon \cite{Simon2025GT} gave combinatorial formulas.
\Cref{corA.linking} is a statement about the statistics of the same numbers,
and it holds jointly for $K_{2,3}$ and finitely many modular knots $C_{\sigma_1},\dots,C_{\sigma_n}$;
see \Cref{rem.linking.joint}.

\subsection{Branched covers} \label{subsec.intro.branched}

The idelic class field theory of 3-manifolds is available for a Chebotarev link,
since a Chebotarev link is stably generic \cite[Theorem 3]{Ueki7};
hence \Cref{thmA.generic} and \Cref{thmA.arith} make the union of the modular knots and of $K_{2,3}$ available for that theory;
see \Cref{subsec.rem.icft}.
The theorem of Neukirch--Uchida type of Gropper, Ueki, and Wang \cite{GropperUekiWang-NU} requires more:
the preimage of the link under every finite cover of $S^{3}$ branched along a finite sublink has to obey the Chebotarev law for the ordering induced from the given ordering \cite[Definition 2.11]{GropperUekiWang-NU}.
Here a component $K'$ of the preimage that has degree $k$ over $K_i$ receives the norm $p_i^{k}$, where $p_i$ denotes the $i$-th prime number,
in analogy with the norm of a prime ideal, and the components are counted by the threshold $p_i^{k}\leq X$.
That property is a formal consequence of the Chebotarev law, by \cite{UekiSC}:
the components of the preimage of a sequence obeying the Chebotarev law,
ordered by that norm or by the norm $\mathrm{N}(K')=i^{k}$, obey the Chebotarev law again,
and the same holds for every ordering in which the components of degree one appear in the order of their images, up to a negligible number of them,
and the components of higher degree have density zero; see \Cref{thm.branched.general}.
For the modular knots, the orderings defined by the lifted flows are of that kind.

\begin{thmA}[{\Cref{thm.branched}}] \label{thmA.branched}
Let $\varpi\colon M_\varpi\to S^{3}$ be a finite cover branched along a finite union of modular knots and of $K_{2,3}$.
Then the components of the preimage of the modular knots and of $K_{2,3}$ obey the Chebotarev law in $M_\varpi$ for the norm orderings attached to the trace ordering or to a generic ordering,
for the ordering by the periods of the lifted geodesic flow,
and for the ordering by the lifted $\ell_h$-periods.
\end{thmA}

Applied to the cover attached to a subgroup of finite index, that statement gives the following.

\begin{thmA}[{\Cref{thm.finiteindex}}] \label{thmA.finiteindex}
Let $\Gamma'\subset \PSL_2\Z$ be a subgroup of finite index.
Then the closed orbits of the geodesic flow on $T^{1}(\Gamma'\backslash \bb{H}^{2})$,
ordered by hyperbolic length, obey the Chebotarev law in that manifold.
\end{thmA}

\subsection{Compact Lie groups and triangle groups} \label{subsec.intro.compact}

For a generic ordering, the law holds also for a compact Lie group whose identity component is semisimple in place of a finite group;
see \Cref{thm.main.general} \textup{(2)}.
In the sequel \cite{UekiMKC2} we show that for the trace ordering the law holds for every compact Lie group,
because the cusp prevents any character of the group from being a function of the hyperbolic length,
and that for the other two orderings the law fails exactly for the characters that are functions of the length;
for $\mathrm{SU}(2)$ the traces of the Frobenius classes obey the Sato--Tate law.
The $(a,b,\infty)$ triangle groups, for which the unit tangent bundle is the complement of a knot in a lens space \cite{MatsusakaUeki2023RMS}, are treated in a separate article.

\subsection{Method} \label{subsec.intro.method}

The two features noted at the end of \Cref{subsec.intro.law} are what the proofs have to overcome, and Part II is organized around them.

For the first one we introduce in \Cref{dfn.markov} the notion of a \emph{Markov section with an escape locus},
in which the rectangles are not assumed to be compact,
and the first return map is allowed to be undefined on a finite union of stable arcs.
The Farey tessellation, that is, the orbit under $\PSL_2\Z$ of the geodesic from $0$ to $\infty$,
provides such a section for the geodesic flow of the modular orbifold.
As a cross-section it is classical (\Cref{rem.modular.sources}):
it has two rectangles, and the expanding part of its return map is a double cover of the Farey map.
With that notion the argument of McMullen \cite[Sections 2--5]{McMullen2013CM} goes through, and \Cref{thmA.generic} follows.
The combinatorial length is the length attached to the constant roof function,
whose suspension has a circle factor (\Cref{subsec.mixing}), so that this argument does not apply to it;
\Cref{thmA.combinatorial} is proved on the same section by the Perron--Frobenius theory of finite matrices.

For the second one the symbolic model has to reproduce the hyperbolic length,
so the roof function has to be the return time.
\Cref{prop.no.finite.markov} shows that the flow admits no Markov section with compact rectangles,
and on the section from the Farey tessellation the return time is not bounded away from zero.
We therefore induce on a subrectangle, which produces the continued fraction coding,
and we run the theory of Artin type $L$-functions on the resulting countable alphabet,
as Pollicott \cite{Pollicott1986BSMF} did for the trivial representation.
A point of the section is determined by the pair $(\xi,\eta)$ of endpoints of the corresponding geodesic, with $\eta$ the expanding coordinate.
Write $f$ for the first return map, $r$ for the return time,
and $F$ for the action of $f$ on the coordinate $\eta$.
The argument is made possible by an explicit cohomology relation between $r$ and $\log|F'|$, of a classical kind (\Cref{prop.induced} \textup{(3)}):
\[r=\log|F'(\eta)|+\Theta\circ f-\Theta,\qquad
\Theta(\xi,\eta)=\tfrac12\log\frac{-\xi}{\eta}-\log|\eta-\xi|,\]
an identity that holds for any Fuchsian group and any section that consists of vectors crossing a translate of the geodesic from $0$ to $\infty$.
The $L$-function is then a quotient of Fredholm determinants of the transfer operator of Mayer twisted by $\rho$,
and the analytic input reduces to two standard theorems,
namely the trace formula for nuclear operators of order zero and the Tauberian theorem of Wiener and Ikehara.
In particular we use neither the Selberg zeta function nor the thermodynamic formalism of countable Markov shifts;
the required spectral facts are proved by explicit estimates in \Cref{lem.nuclear} and \Cref{lem.non-vanishing}.

\subsection{Organization} \label{subsec.intro.plan}

Part I states the results and explains why the ordering matters.
\Cref{sec.background} recalls the modular knots,
the Rademacher symbol, the precise relationship with the Lorenz knots,
and the distribution theorem of Sarnak and Mozzochi.
\Cref{sec.ordering} compares the three length functions and proves \Cref{thmA.densities},
with \Cref{thmA.combinatorial}, \Cref{thmA.generic}, and \Cref{thmA.arith} taken for granted.

Part II proves \Cref{thmA.combinatorial}, \Cref{thmA.generic}, and \Cref{thmA.arith}.
\Cref{sec.symbolic} and \Cref{sec.flat} contain the symbolic part of the argument,
which applies the theorem of Parry and Pollicott.
\Cref{sec.markov} and \Cref{sec.spine} contain the geometric part,
which deduces the statement on the manifold from the symbolic one.
\Cref{sec.escape} combines the two parts into \Cref{thm.main.general}, whose assumptions do not require the manifold to be compact,
and \Cref{sec.modular} verifies those assumptions for the modular flow,
on the classical cross-section from the Farey tessellation.
In \Cref{subsec.induced} we induce the first return map on a subrectangle,
which gives the coding by continued fractions that \Cref{sec.arith} uses,
and \Cref{subsec.refine} treats the removal of finitely many modular knots.
\Cref{subsec.modular.generic} proves \Cref{thmA.generic},
and \Cref{subsec.cusp} explains in which sense the cusp is an obstruction.
The three orderings are then treated by three different tools:
the generic orderings by the theorem of Parry and Pollicott,
the combinatorial ordering by the Perron--Frobenius theory of finite matrices in \Cref{sec.combinatorial},
and the trace ordering by the transfer operator in \Cref{sec.arith}.
\Cref{sec.linking} applies \Cref{thmA.combinatorial}, \Cref{thmA.generic}, and \Cref{thmA.arith} to the linking numbers and proves \Cref{corA.linking}.

Part III treats the preimages under branched covers and further topics.
\Cref{sec.stably.chebotarev} recalls the stability of the Chebotarev law under branched covers from \cite{UekiSC} and proves \Cref{thmA.branched},
and \Cref{sec.finiteindex} deduces the law for the subgroups of finite index.
\Cref{sec.remarks} discusses three further topics, namely the idelic class field theory of 3-manifolds,
an analogue of the supersingular primes, and the reconstruction of a manifold from its fundamental group,
and it ends with three questions.

\Cref{thm.main.general} is stated for any flow that satisfies the assumptions listed there.
The statements of \Crefrange{sec.symbolic}{sec.escape} are made for a compact Lie group with semisimple identity component,
as in McMullen \cite{McMullen2013CM}; from \Cref{sec.modular} on, the group is finite.

\Crefrange{sec.symbolic}{sec.spine} follow McMullen \cite[Sections 2--5]{McMullen2013CM},
where the argument is presented concisely, and we write out the intermediate steps in the form that the present setting requires.
We follow the notation of \cite{McMullen2013CM} where possible.
In particular, $h$ denotes a roof function and not an entropy.
We supply the arguments for the following three items.
\begin{enumerate}
\item The flat cocycle depends on a choice of base paths, which we make
explicit.
\item The complement of the union of the flow boxes is not a locally finite
complex, so that the deformation argument has to be carried out with respect to the finite family of the faces of the boxes.
\item The correspondence between closed orbits and periodic points needs a
realization assumption, once the rectangles are allowed to be noncompact.
\end{enumerate}
A Markov section with an escape locus (\Cref{dfn.markov}) is what makes the noncompact case accessible.

\section{Modular knots, Lorenz knots, and the Rademacher symbol}
\label{sec.background}

\subsection{Modular knots} \label{subsec.bg.modular}

Let $\bb{H}^{2}$ denote the upper half plane and $\Gamma=\PSL_2\Z$.
The unit tangent bundle $M=T^{1}(\Gamma\backslash \bb{H}^{2})$ is homeomorphic to the complement of the trefoil $K_{2,3}$ in $S^{3}$.
The geodesic flow on $M$ is called the \emph{modular flow}, following Ghys.
Its closed orbits are called \emph{modular knots},
and they correspond to the conjugacy classes of the primitive hyperbolic elements $\gamma\in \Gamma$;
we write $C_\gamma$ for the modular knot attached to $\gamma$.
Its length is $\ell(\gamma)=2\log \xi_\gamma$, where $\xi_\gamma>1$ is determined by $\xi_\gamma+\xi_\gamma^{-1}=|\tr\gamma|$,
so that ordering the modular knots by length is the same as ordering them by $|\tr \gamma|$.

We identify $M$ with the homogeneous space $\Gamma\backslash \PSL_2\R$ as follows.
Let $v_0$ denote the upward unit tangent vector at $\sqrt{-1}\in \bb{H}^{2}$.
The action of $\PSL_2\R$ on $T^{1}\bb{H}^{2}$ by the differentials of the M\"obius transformations is simply transitive,
so that the map $\PSL_2\R\to T^{1}\bb{H}^{2}$, $g\mapsto g\,v_0$, is a diffeomorphism, and it is equivariant for the left actions of $\Gamma$.
Hence it induces a diffeomorphism $\Gamma\backslash \PSL_2\R\to M$.
The geodesic through $v_0$ is the imaginary axis, parametrized by $t\mapsto e^{t}\sqrt{-1}$,
and it is the orbit of $v_0$ under the one-parameter subgroup $a_t=\spmx{e^{t/2}&0\\0&e^{-t/2}}$.
Since the geodesic flow commutes with the isometries, under this identification the modular flow is the right multiplication $\Gamma g\mapsto \Gamma g\,a_t$.

\subsection{The Rademacher symbol} \label{subsec.bg.rademacher}

Let $\Delta(z)=q\prod_{n\geq 1}(1-q^{n})^{24}$ with $q=e^{2\pi i z}$ be the discriminant function, a modular form of weight $12$.
For $\gamma=\spmx{a&b\\c&d}\in \Gamma$, the \emph{Dedekind symbol} $\Phi:\Gamma\to \Z$ is determined by
\[\log \Delta(\gamma z)-\log \Delta(z)=
\begin{cases}
6\log\bigl(-(cz+d)^{2}\bigr)+2\pi i\,\Phi(\gamma) & \text{if } c\neq 0,\\
2\pi i\,\Phi(\gamma) & \text{if } c=0,
\end{cases}\]
where $\log\Delta(z)=2\pi i z+24\sum_{n\geq 1}\log(1-q^{n})$ and the branch of the logarithm is chosen so that $-\pi\leq \operatorname{Im}\log z<\pi$.
The \emph{Rademacher symbol} $\Psi:\Gamma\to \Z$ is then given by
\[\Psi(\gamma)=\Phi(\gamma)-3\operatorname{sgn}\bigl(c(a+d)\bigr).\]
Unlike $\Phi$, the function $\Psi$ is constant on conjugacy classes.
The conventions differ from one source to another, and we follow Matsusaka \cite{Matsusaka2024MathAnn}.

The symbol $\Psi$ is remarkably ubiquitous.
Atiyah \cite{Atiyah1987MA} proved that seven definitions of $\Psi$ arising in unrelated settings all agree,
and Barge and Ghys \cite{BargeGhys1992} showed that the coboundary of $\Psi$ is, up to a constant factor that depends on the normalization,
the Euler cocycle of $\PSL_2\R$ acting on the circle $\partial\bb{H}^{2}$, restricted to $\Gamma$;
the coboundary determines $\Psi$, because $\Gamma$ admits no nonzero homomorphism to $\Z$.
A version for the triangle groups is \cite[Theorem 4.24]{MatsusakaUeki2023RMS}.
The characterization that concerns us is the one of Ghys.

\begin{thm}[{\cite[Sections 3.3--3.5]{Ghys2007ICM}}] \label{thm.ghys}
For every primitive hyperbolic $\gamma\in \Gamma$,
\[{\rm lk}\bigl(C_\gamma,K_{2,3}\bigr)=\Psi(\gamma).\]
\end{thm}

There is also a combinatorial description, which is the form in which $\Psi$ appears in this article.
Every primitive hyperbolic conjugacy class contains an element that is a positive word in
\[L=\spmx{1&0\\1&1},\qquad R=\spmx{1&1\\0&1},\]
in which both letters occur, and that word is unique up to a cyclic permutation;
the notation $L$ and $R$ is the classical one.
We reprove that assertion in \Cref{prop.modular} \textup{(4)} below.
In terms of that word,
\[\Psi(\gamma)=\#\{R\}-\#\{L\},\]
by Rademacher and Grosswald \cite[(70)]{RademacherGrosswald1972} and Ghys \cite[Section 3.5]{Ghys2007ICM}.
The sign in that formula depends on the conventions,
both for $\Psi$ and for the orientations used in the definition of the linking number.
Every statement of this article about $\Psi$ involves either $|\Psi|$,
or the full set of residues of $\Psi$ modulo an integer,
or a law that is symmetric under $\Psi\mapsto -\Psi$, so that no statement is affected by that choice.

\subsection{Lorenz knots} \label{subsec.bg.lorenz}

The Lorenz equations \cite{Lorenz1963JAS} are
\[\frac{dx}{dt}=10(y-x),\qquad \frac{dy}{dt}=28x-y-xz,\qquad
\frac{dz}{dt}=xy-\frac83 z,\]
whose solutions Lorenz found to depend sensitively on the initial conditions.
The relationship between the Lorenz equations and the Lorenz knots needs care,
because a Lorenz knot is not defined as a closed orbit of the equations.
\begin{enumerate}
\item A \emph{geometric Lorenz model} is a flow built so as to imitate the
equations, and it has an attractor that contains a singularity \cite{Guckenheimer1976, AfraimovichBykovShilnikov1977}.
That attractor is structurally stable of codimension two:
any flow near a geometric model is topologically equivalent (conjugate, in the terminology of the authors),
on a neighborhood of the attractor, to a member of an explicit two-parameter family \cite{GuckenheimerWilliams1979IHES}.
The closed orbits of a geometric model form a countable link in $S^{3}$, which depends on the model.
\item Birman and Williams \cite{BirmanWilliams1983T} introduced a certain topological model, which is called the \emph{Lorenz template}.
It is a branched surface in $S^{3}$ with an action of $\R_{\geq 0}$,
and the closed orbits of every geometric model correspond to closed orbits carried by it,
in such a way that corresponding finite sublinks are isotopic in $S^{3}$.
The closed orbits carried by the Lorenz template, other than two unknotted ones, are called the \emph{Lorenz knots}.
They deduced that the Lorenz knots are fibered and are closures of positive braids,
and Williams \cite{Williams1984ETDS} proved that they are prime.
\item Tucker \cite{Tucker1999CRASP} proved that the equations have the
attractor of a geometric model.
Hence the periodic orbits inside that attractor are Lorenz knots.
Not every Lorenz knot occurs there, whereas for some other parameter values
every Lorenz knot is a closed orbit of the equations \cite[Corollary 1.3]{Pinsky2023PNAS}.
We make no statement about closed orbits of the equations outside the attractor.
\end{enumerate}

Ghys \cite[Section 3.5]{Ghys2007ICM} proved that the family of the modular knots and the family of the Lorenz knots coincide.
Both are indexed by the primitive cyclic words in two letters in which both letters occur, and the two indexings agree.
Ghys proved moreover that corresponding finite sublinks are isotopic in $S^{3}$;
the collapse of Birman and Williams \cite{BirmanWilliams1983T} preserves the isotopy class of a finite union of orbits, and the two templates are isotopic.

Bonatti and Pinsky \cite{BonattiPinsky2021} constructed a family $\{X_r\}_{r\in [-1,1]}$ of flows on $S^{3}$ with an invariant trefoil knot formed by invariant manifolds of the singular points.
For $r<0$ the chain recurrent set of $X_r$ contains a Lorenz attractor.
For $r>0$ the flow $X_r$ is topologically equivalent to a flow on $S^{3}$ that satisfies Axiom A
and strong transversality and that extends the geodesic flow of a funnel deformation of the modular orbifold;
it is obtained in \cite[Theorem 1.1]{BonattiPinsky2021} by turning the cusp into a funnel,
and the modular flow is the limit of those geodesic flows and not one of them.
The identification of the two families of knots is the theorem of Ghys recalled above.
Topological equivalence is an equivalence of orbits and does not preserve lengths;
see \Cref{subsec.cusp} for the periods of that completion.

\subsection{The distribution of the Rademacher symbol}
\label{subsec.bg.distribution}

Sarnak \cite{Sarnak2010CMA} outlined the theorem below and indicated that the tools of the spectral theory of hyperbolic surfaces,
and the Selberg trace formula in particular, should give it.
Mozzochi \cite{Mozzochi2013Israel} supplied the proofs,
by way of variants of the prime geodesic theorem for forms of real weight.

\begin{thm}[{\cite{Sarnak2010CMA}, \cite[Theorem 1.5]{Mozzochi2013Israel}}] \label{thm.cauchy}
Let $\gamma$ run through the conjugacy classes of the primitive hyperbolic elements of $\Gamma$.
For all $-\infty\leq a\leq b\leq \infty$,
\[\lim_{\nu\to\infty}
\frac{\#\{\gamma\mid \ell(\gamma)<\nu,\
a\leq \Psi(\gamma)/\ell(\gamma)\leq b\}}
{\#\{\gamma \mid \ell(\gamma)<\nu\}}
=\frac{\arctan(\pi b/3)-\arctan(\pi a/3)}{\pi}.\]
\end{thm}

Thus $\Psi/\ell$ obeys the Cauchy law of scale $3/\pi$ for the ordering by length.
For the combinatorial ordering the same symbol obeys a Gaussian law,
and \Cref{sec.ordering} explains why the ordering matters.

\section{On the ordering} \label{sec.ordering}

\subsection{The three length functions} \label{subsec.ord.roof}

Every primitive hyperbolic conjugacy class $\gamma$ of $\PSL_2\Z$ is represented by a cyclic word in $L$ and $R$,
and the symbolic model of the modular flow constructed in \Cref{sec.modular} is the full shift $\Sigma$ on these two letters.
A modular knot corresponds to a periodic orbit of the shift.
A \emph{roof function} is a positive H\"older function $h\colon \Sigma\to \R_{>0}$,
and the length $\ell_h(\gamma)$ is the sum of $h$ over the periodic orbit attached to $\gamma$;
see \Cref{subsec.suspension}.
The constant function $h=1$ gives $\ell_h=\ell_{\rm c}$.
If $h$ takes the value $\alpha$ on the sequences $x\in \Sigma$ with $x_0=L$ and the value $\beta$ on those with $x_0=R$, then
\[\ell_h(\gamma)=\alpha\, n_L(\gamma)+\beta\, n_R(\gamma)
=\frac{\alpha+\beta}{2}\,\ell_{\rm c}(\gamma)
+\frac{\beta-\alpha}{2}\,\Psi(\gamma),\]
where $n_L(\gamma)$ and $n_R(\gamma)$ denote the numbers of the letters $L$ and $R$ in the cyclic word of $\gamma$.

The roof function is not required to be the return time of the flow,
and in \Cref{thm.main.general} it is not the return time.
Indeed, the only roles of the roof function in the proof are to define the ordering and to rule out a circle factor,
while the correspondence of closed orbits and the computation of the Frobenius classes depend only on the Markov section.
The hyperbolic length $\ell$ is the length attached to the return time.
The modular flow admits no Markov section whose rectangles are compact,
and the return time of the section of \Cref{sec.modular} is not bounded away from zero;
see \Cref{subsec.cusp}.
The trace ordering is therefore treated separately,
in \Cref{sec.arith}, on a symbolic model with a countable alphabet.

\subsection{Comparing the two lengths} \label{subsec.ord.compare}

Recall that $\ell_{\rm c}(\gamma)$ denotes the combinatorial length,
which is comparable to the length attached to any roof function that is bounded above and bounded away from zero.

\begin{prop} \label{prop.incommensurable}
Let $\kappa$ denote the golden ratio $(1+\sqrt5)/2$.
Then the following two assertions hold.
\begin{enumerate}
\item $\ell(\gamma)\leq 2\log\kappa\cdot \ell_{\rm c}(\gamma)$ for every
hyperbolic $\gamma\in \PSL_2\Z$, with equality exactly for the conjugates of the powers of $LR$.
\item For every $T\geq 2\log 3$,
\[\max\{\ell_{\rm c}(\gamma)\mid \ell(\gamma)\leq T\}\geq e^{T/2}-2,\]
and the words $L^{q}R$ realize that bound.
\end{enumerate}
\end{prop}

\begin{proof}
(1) Let $\|A\|=\max_{v\neq 0}|Av|/|v|$ denote the operator norm of a real matrix $A$ with respect to the Euclidean norm;
its square is the larger eigenvalue of $A^{\top}A$.
Since $L^{\top}L=\spmx{2&1\\1&1}$ has the eigenvalues $\kappa^{2}$ and $\kappa^{-2}$, we have $\|L\|=\kappa$,
and $\|R\|=\kappa$ because $R$ is the transpose of $L$.
The operator norm is submultiplicative, so that a word of length $n$ has operator norm at most $\kappa^{n}$.
Since $\xi_\gamma\leq \|\gamma\|$, we get $\ell(\gamma)=2\log\xi_\gamma\leq 2n\log\kappa$.
For $\gamma=(LR)^{m}$ we have $\xi_\gamma=\kappa^{2m}$ and $n=2m$, which is the case of equality.
Conversely, assume that $\ell(\gamma)=2\log\kappa\cdot n$ holds for a hyperbolic $\gamma\in \PSL_2\Z$ whose cyclic word has length $n$, so that $\xi_\gamma=\kappa^{n}$.
If two equal letters are adjacent in the cyclic word, then, after a cyclic permutation,
which replaces $\gamma$ by a conjugate and does not change $\xi_\gamma$,
the linear word contains the block $L^{2}$ or $R^{2}$.
The matrix $(L^{2})^{\top}L^{2}=\spmx{5&2\\2&1}$ has the eigenvalues $(1\pm\sqrt2)^{2}$,
so that $\|L^{2}\|=1+\sqrt2$, and the same value holds for $R^{2}$, which is the transpose of $L^{2}$.
Hence $\xi_\gamma\leq \|\gamma\|\leq (1+\sqrt2)\,\kappa^{n-2}<\kappa^{n}$, which contradicts $\xi_\gamma=\kappa^{n}$.
Therefore the letters alternate along the cyclic word, the length $n$ is even,
and the word is the cyclic word of $(LR)^{n/2}$.

(2) The word $L^{q}R=\spmx{1&1\\q&q+1}$ has trace $q+2$ and combinatorial length $q+1$.
Since $\xi_\gamma<\tr\gamma$ for a hyperbolic $\gamma$, we get $\ell(\gamma)<2\log(q+2)$,
which is at most $T$ provided that $q\leq e^{T/2}-2$.
Put $q=\lfloor e^{T/2}-2\rfloor$, which satisfies $q\geq 1$ because $T\geq 2\log 3$.
Then the resulting word has combinatorial length $\lfloor e^{T/2}-2\rfloor+1>e^{T/2}-2$.

\end{proof}

So the two length functions are comparable from one side only:
$\ell\leq 2\log\kappa\cdot\ell_{\rm c}$ holds by (1),
whereas by (2) the initial segment $\{\ell\leq T\}$ contains geodesics whose combinatorial length is exponential in $T$.
In the ordering by $\ell$, the class of $L^{q}R$ is among the first $q^{2}$ classes,
whereas in the ordering by $\ell_{\rm c}$ it is preceded by a number of classes of the order of $2^{q}/q$.
This comparison is only an indication; the proof that the orderings are inequivalent is \Cref{thm.two.densities}.

\begin{rem} \label{rem.crossings}
The combinatorial length $\ell_{\rm c}(\gamma)$ is the number of edges of the Farey tessellation
that a lift of the closed geodesic of $\gamma$ crosses in one period; see \Cref{sec.modular}.
The same count is defined for a geodesic segment.
If the segment has length $T$ and is chosen at random, then the count is of the order of $T\log T$, in the following sense.
Let $v$ be a unit vector chosen at random with respect to the Liouville probability measure,
and let $\wt{v}$ be a lift of $v$ to $T^{1}\bb{H}^{2}$.
Then the number of edges of the Farey tessellation that the geodesic segment $\varphi_{[0,T]}(\wt v)$ crosses,
divided by $T\log T$, converges in probability to $6/\pi^{2}$ as $T\to \infty$.
Indeed, that number is the sum of the continued fraction digits met by the segment,
the number of those digits is asymptotic to $T\cdot 6\log 2/\pi^{2}$ by the ergodic theorem
and the value $\pi^{2}/(6\log 2)$ of the Lyapunov exponent of the Gauss map \cite{Series1985JLMS},
and the sum of $n$ digits divided by $n\log n$ converges in probability to $1/\log 2$ by the weak law of Khinchin \cite{Khintchine1935CM}.
\end{rem}

\subsection{Two limit laws and three densities} \label{subsec.ord.laws}

A density with respect to one ordering says nothing about a density with respect to another.
A concrete instance is the Rademacher symbol $\Psi(\gamma)=\#R-\#L$ evaluated on the cyclic word of \Cref{subsec.bg.rademacher},
for which the combinatorial and the trace ordering can be compared by an explicit computation.

For the combinatorial ordering, $\Psi$ is the sum of the values of the function that takes the value $1$ on the letter $R$ and the value $-1$ on the letter $L$, so that
\[\sum_{|w|=n}e^{2\pi i\theta \Psi(w)}
=\bigl(e^{2\pi i\theta}+e^{-2\pi i\theta}\bigr)^{n}
=\bigl(2\cos 2\pi\theta\bigr)^{n},\]
where the sum ranges over the $2^{n}$ words of length $n$.
The function $\Psi$ is constant on each class of cyclic permutations,
and the proportion of classes that contain fewer than $n$ words tends to zero,
so the same computation describes the distribution of $\Psi$ over the modular knots of combinatorial length $n$.
Dividing by $2^{n}$ and putting $\theta=u/(2\pi\sqrt n)$ gives the characteristic function $\cos^{n}(u/\sqrt n)$ of $\Psi/\sqrt n$,
which converges to $e^{-u^{2}/2}$ as $n\to\infty$,
so that $\Psi/\sqrt{n}$ converges in law to a standard Gaussian.
Equivalently, the twisted count decays like $\cos^{n}(2\pi\theta)=\exp(-2\pi^{2}\theta^{2}n+O(\theta^{4}n))$, that is, \emph{quadratically} in $\theta$.

For the trace ordering, $\Psi(\gamma)/\ell(\gamma)$ obeys a Cauchy law with scale $3/\pi$ by Sarnak \cite{Sarnak2010CMA} and Mozzochi \cite{Mozzochi2013Israel},
which corresponds to a twisted count that decays like $e^{-6|\theta| T}$, that is, \emph{linearly} in $|\theta|$.
Indeed the twisted sum over the classes with $\ell\leq T$ is of the order of $e^{(1-6|\theta|)T}$ for $|\theta|\leq 1/24$,
by the estimate of Sarnak \cite{Sarnak2010CMA} and Mozzochi \cite{Mozzochi2013Israel} recalled in \Cref{subsubsec.q.bias};
the transfer operator reproduces the exponent $1-6\theta$ numerically, as \Cref{rem.numerics} shows.
The same function therefore has different limit laws for the two orderings,
as Sarnak \cite{Sarnak2010CMA} pointed out: the combinatorial ordering does not make the cusp singular,
and for that ordering the fluctuations of a quasimorphism are Gaussian,
by the central limit theorem of Calegari and Fujiwara \cite{CalegariFujiwara2010ETDS}, who state it for the Rademacher function and any word metric,
while the cusp produces the Cauchy law for the ordering by length.

Guivarc'h and Le Jan \cite{GuivarchLeJan1993ASENS} established the corresponding dichotomy for the winding of the geodesic flow itself,
with respect to the Liouville measure, by the Farey coding and a perturbation of the transfer operator;
see \cite{BabillotPeigne2006BSMF} for other manifolds with cusps.
Epstein \cite[Theorem 8.2]{Epstein1987Duke} determined the asymptotic number of the closed geodesics of length at most $T$
in a fixed homology class of a non-compact hyperbolic surface of finite area and of genus at least $2$.
In the denominator of that number, a generator of the homology that comes from a cusp contributes the factor $T$,
whereas each of the other generators contributes the factor $T^{1/2}$, as in the compact case;
see \cite[Th\'eor\`eme A]{BabillotPeigne2000ASENS} for groups of Schottky type with parabolic elements.
Dubno \cite{Dubno2026GD} showed that the Cauchy law disappears when the excursions into the cusp are bounded.
Consider the closed geodesics whose continued fraction digits are at most a fixed $A$, and order them by the number of digits.
Then the ratio $\Psi/\sqrt{\ell}$ has a Gaussian limit law, and so does $\Psi/\sqrt{\ell_{\rm c}}$.

The difference appears not only in limit laws: a single set of modular knots has two different densities.

\begin{thm}[{\Cref{thmA.densities}}] \label{thm.two.densities}
Let
\[\mca{A}=\{C_\gamma \mid |\Psi(\gamma)|\leq \ell(\gamma)\},\]
where $\gamma$ runs through the primitive hyperbolic conjugacy classes of $\PSL_2\Z$, so that $\mca{A}$ is a set of modular knots.
Then $\mca{A}$ has density $1$ for the ordering by combinatorial length, for any tie-breaking, and density
\[\frac{2}{\pi}\arctan\frac{\pi}{3}=0.5147\ldots<1\]
for the ordering by hyperbolic length.
\end{thm}

\begin{proof}
The second assertion is \Cref{thm.cauchy} with $a=-1$ and $b=1$.

For the first one, write $\gamma=L^{k_1}R^{l_1}\cdots L^{k_j}R^{l_j}$ with $k_i,l_i\geq 1$, so that $j$ is the number of blocks.
Each factor $L^{k}R^{l}=\spmx{1&l\\k&kl+1}$ has all its entries at least one,
and if a matrix $M$ with non-negative entries has all its entries at least $m$,
then the entries of $M\,L^{k}R^{l}$ are at least $2m$, because $k,l\geq 1$.
Hence all the entries of $\gamma$ are at least $2^{j-1}$, so that $\tr\gamma\geq 2^{j}$ and
\[\ell(\gamma)=2\log \xi_\gamma\geq 2\log\frac{\tr\gamma}{2}
\geq 2(j-1)\log 2 .\]
Now order the modular knots by the combinatorial length $n$.
Consider the uniform distribution on the $2^{n}$ words of length $n$.
The number of blocks is $j=\frac12\#\{i\mid w_i\neq w_{i+1}\}$, where $i$ runs cyclically,
and $\Psi$ is the sum of $n$ independent variables taking the values $\pm 1$.
The indicators of $\{w_i\neq w_{i+1}\}$ with $1\leq i\leq n-1$ are independent Bernoulli variables of parameter $1/2$,
and the cyclic term changes the count by at most one.
The proportion of words of length $n$ with $j\geq n/5$ and $|\Psi|\leq n^{2/3}$ tends to one,
by the large deviation estimate for the first sum and by the central limit theorem for the second.
For such a word, $\ell\geq 2(n/5-1)\log 2>n^{2/3}\geq |\Psi|$ provided that $n$ is large enough, so that $C_\gamma\in\mca{A}$.
Since the number of non-primitive words of length $n$ is $O(2^{n/2})$,
the proportion of primitive words of length $n$ that lie in $\mca{A}$ tends to one,
and since every primitive class of length $n$ contains exactly $n$ words,
the same holds for the proportion of primitive classes.
The number of primitive classes of length $n$ is $2^{n}/n+O(2^{n/2})$,
so for each fixed $m\in \Z_{\geq 1}$ the classes of length at most $N-m$ form a proportion $O(2^{-m})$ of the classes of length at most $N$,
and hence the density of $\mca{A}$ for the ordering by the combinatorial length is one.
The tie-breaking among the classes of length $N$ does not matter,
because the classes of length $N$ outside $\mca{A}$ are $o(2^{N}/N)$ in number,
which is negligible compared with the number of classes of length less than $N$.
\end{proof}

\begin{rem} \label{rem.densities.numerics}
The two densities can be observed numerically.
Among the words of combinatorial length $n$, the proportion lying in $\mca{A}$ is $0.961$, $0.987$, and $0.995$ for $n=12$, $16$, and $20$.
Among the conjugacy classes with $\tr\gamma\leq N$, it is $0.525$, $0.521$, and $0.517$ for $N=60$, $200$,
and $600$, decreasing towards $0.5147\ldots$; the numbers of those classes are $488$, $4186$, and $30674$.
\end{rem}

The first assertion of \Cref{thm.two.densities} also holds for any ordering that is close enough to the combinatorial one, in the following sense.

\begin{prop} \label{prop.density.ellh}
Put
\[a=\frac25\log\frac45+\frac35\log\frac65=0.0201\ldots,\]
and let $\ell'$ be a function on the primitive hyperbolic conjugacy classes of $\PSL_2\Z$ that satisfies $c_1\ell_{\rm c}\leq \ell'\leq c_2\ell_{\rm c}$ for some constants $c_2\geq c_1>0$.
If $c_1/c_2>1-a/\log 2=0.9709\ldots$, then the set $\mca{A}$ of \Cref{thm.two.densities} has density $1$ for the ordering by $\ell'$.
\end{prop}

\begin{proof}
A word of length $n$ is said to be \emph{exceptional} if the corresponding class $\gamma$ satisfies $|\Psi(\gamma)|>\ell(\gamma)$.
Let $j$ denote the number of blocks, as in the proof of \Cref{thm.two.densities}.
That proof shows that if $j\geq n/5$, then $\ell(\gamma)\geq \frac25 n\log 2-2\log 2$,
so an exceptional word satisfies $j<n/5$ or $|\Psi|>\frac25 n\log 2-2\log 2$.
The indicators considered there have sum at most $2j$,
and the Chernoff bound gives at most $2^{n}e^{-an+o(n)}$ words of length $n$ with $j<n/5$,
the rate being $D(\tfrac25\,\|\,\tfrac12)=a$.
The inequality of Hoeffding bounds the number of words of length $n$ with $|\Psi|>\frac25 n\log 2-2\log 2$ by $2^{n}e^{-bn+o(n)}$ with $b=\frac{2}{25}(\log 2)^{2}=0.0384\ldots>a$.
Hence the number of exceptional words of length $n$ is at most $2^{n}e^{-an+o(n)}$.
The functions $j$, $\Psi$, and $\ell$ are invariant under cyclic permutations,
so the same bound holds for the number of primitive classes of combinatorial length $n$ outside $\mca{A}$,
and summing over $n\leq X$ bounds the number of primitive classes with $\ell_{\rm c}\leq X$ outside $\mca{A}$ by $2^{(1-a/\log 2)X+o(X)}$.

Now order the classes by $\ell'$, with any tie-breaking,
and for $\nu\in \Z_{\geq 1}$ let $T_\nu$ denote the $\nu$-th value of $\ell'$, so that $T_\nu\to \infty$ as $\nu\to\infty$.
Every class among the first $\nu$ satisfies $\ell'\leq T_\nu$, hence $\ell_{\rm c}\leq T_\nu/c_1$,
and every class with $\ell_{\rm c}<T_\nu/c_2$ satisfies $\ell'<T_\nu$ and is therefore among the first $\nu$.
Since the number of classes of combinatorial length $n$ is $2^{n}/n+O(2^{n/2})$,
the latter inclusion gives $\nu\geq 2^{T_\nu/c_2+o(T_\nu)}$.
Hence the proportion of the first $\nu$ classes that lie outside $\mca{A}$ is at most
\[2^{\,(1-a/\log 2)\,T_\nu/c_1-T_\nu/c_2+o(T_\nu)},\]
and the exponent tends to $-\infty$, because $(1-a/\log 2)/c_1<1/c_2$.
\end{proof}

The roof function $h$ of \Cref{thm.modular} can be chosen with $h_{\min}/h_{\max}>1-a/\log 2$,
for instance as in \Cref{eg.explicit.ordering} with $|\varepsilon|$ small;
for such an $h$ the set $\mca{A}$ has density $1$ for the ordering by $\ell_h$ as well.
We do not pursue here the question for which roof functions $h$ that holds.

\begin{rem} \label{rem.exact.length}
The genericity of the roof function in \Cref{thm.modular} cannot be removed.
Indeed, consider the ordering by the combinatorial length itself,
and let $\rho:\pi_1(S^{3}-K_{2,3})\surj \Z/2\Z$ be the composition of the abelianization with the reduction modulo $2$.
By \Cref{thm.ghys}, the class $\rho([C_\gamma])$ is the residue of $\Psi(\gamma)$,
and $\Psi(\gamma)=\#\{R\}-\#\{L\}\equiv \ell_{\rm c}(\gamma) \pmod 2$,
so that the class is determined by the parity of $\ell_{\rm c}(\gamma)$.
The number of the modular knots of combinatorial length $n$ is $2^{n}/n$ up to a factor $1+o(1)$.
The proportion of each class therefore oscillates:
the upper and lower densities of the even class are $2/3$ and $1/3$.
Hence the modular knots ordered by $\ell_{\rm c}$,
with $K_{2,3}$ at any fixed position and for any tie-breaking, do not obey the Chebotarev law.
\end{rem}

The second set of \Cref{thmA.densities} is the set of \Cref{rem.exact.length}.

\begin{cor}[{Second assertion of \Cref{thmA.densities}}]
\label{cor.three.densities}
Let $\mca{B}=\{C_\gamma\mid \Psi(\gamma)\ \text{is even}\}$,
where $\gamma$ runs through the primitive hyperbolic conjugacy classes of $\PSL_2\Z$.
Then $\mca{B}$ has density $\frac12$ for the ordering by $\ell_h$ of \Cref{thm.modular} and for the ordering by hyperbolic length,
and for the ordering by combinatorial length, along the thresholds $\ell_{\rm c}\leq N$,
its lower and upper densities are $\frac13$ and $\frac23$.
\end{cor}

\begin{proof}
Let $\rho:\pi_1(S^{3}-K_{2,3})\surj \Z/2\Z$ be as in \Cref{rem.exact.length},
so that $\rho([C_\gamma])$ is the parity of $\Psi(\gamma)$.
The first assertion is the case $G=\Z/2\Z$ of \Cref{thm.modular},
the second is the case $G=\Z/2\Z$ of \Cref{thm.arith}, and the third is \Cref{rem.exact.length}.
\end{proof}

\Cref{thm.two.densities}, together with \Cref{prop.density.ellh} and \Cref{cor.three.densities},
shows that none of \Cref{thm.combinatorial}, \Cref{thm.modular},
and \Cref{thm.arith.general} follows from another by a comparison of densities.
The set that separates the combinatorial ordering from the other two is defined by a Frobenius class,
whereas the set $\mca{A}$ that separates the trace ordering from the generic one is not:
the latter two orderings both obey the Chebotarev law,
and hence they give the same density to every set defined by a Frobenius class.

\part{Symbolic dynamics and transfer operators}

\section{Symbolic dynamics} \label{sec.symbolic}

\subsection{Graphs and subshifts} \label{subsec.graph}

Let $\mca{G}$ be a nonempty finite directed graph with vertex set $V(\mca{G})$ and edge set $E(\mca{G})$.
We assume that an edge $e=(v,w)$ is determined by its initial vertex $v$ and its terminal vertex $w$,
and that every vertex has at least one incoming and at least one outgoing edge.
The associated subshift of finite type is
\[\Sigma(\mca{G})=\{x:\Z\to V(\mca{G})\mid (x_i,x_{i+1})\in E(\mca{G})
\text{ for all } i\in \Z\},\]
with the shift $\sigma(x)_i=x_{i+1}$ and the metric $d(x,y)=\sup_{i\in \Z}2^{-|i|}d_0(x_i,y_i)$,
where $d_0(v,w)=1$ if $v\neq w$ and $d_0(v,v)=0$.
The space $\Sigma(\mca{G})$ is compact, and $\sigma$ is a homeomorphism.

\begin{lem} \label{lem.irreducible}
The following three conditions are equivalent.
\begin{itemize}
\item[\textup{(i)}] Any two vertices of $\mca{G}$ are joined by a directed
path.
\item[\textup{(ii)}] The graph $\mca{G}$ is connected, and every edge lies on
a directed loop.
\item[\textup{(iii)}] For any nonempty open subsets $U$ and $V$ of
$\Sigma(\mca{G})$, there is an integer $n\geq 1$ with $U\cap \sigma^{-n}(V)\neq \emptyset$.
\end{itemize}
Moreover, if these conditions hold, then there is a point $x\in \Sigma(\mca{G})$ whose forward orbit $\{\sigma^{n}(x)\mid n\in \Z_{\geq 0}\}$ is dense in $\Sigma(\mca{G})$.
\end{lem}

\begin{proof}
Since every vertex has an incoming and an outgoing edge,
every directed path extends on both sides to a point of $\Sigma(\mca{G})$.
In particular, for a directed path $u=(u_0,\dots,u_k)$ and $j\in \Z$,
the cylinder $[u]_j=\{y\in \Sigma(\mca{G})\mid y_{j+i}=u_i\ \text{for}\ 0\leq i\leq k\}$ is nonempty,
and these cylinders form a basis of the topology of $\Sigma(\mca{G})$.

Assume (i).
Then $\mca{G}$ is connected, and for an edge $e=(v,w)\in E(\mca{G})$ a directed path from $w$ to $v$ closes $e$ into a directed loop, which gives (ii).
Assume (ii), and let $v$ and $w$ be two vertices.
Since $\mca{G}$ is connected, there is a path from $v$ to $w$ in which some steps may traverse an edge backwards.
Consider such a step, which goes from $b$ to $a$ along an edge $e=(a,b)$.
By (ii) the edge $e$ lies on a directed loop $\mu$, which we may read as $a\to b\to \cdots \to a$;
the portion of $\mu$ from $b$ to $a$ is then directed, and we substitute it for the backward step.
That substitution decreases the number of backward steps by one,
so after finitely many substitutions we obtain a directed path from $v$ to $w$, which gives (i).

Assume (i), and let $U$ and $V$ be nonempty open subsets of $\Sigma(\mca{G})$.
Choose cylinders $[u]_j\subset U$ and $[u']_{j'}\subset V$ for directed paths $u=(u_0,\dots,u_k)$ and $u'$.
By (i) there is a directed path from the terminal vertex of $u$ to the initial vertex of $u'$,
and since the terminal vertex of $u$ lies on a directed loop by (ii),
which follows from (i), the connecting path can be taken of length $\ell_0+mc$ for a fixed loop length $c\geq 1$ and every $m\in \Z_{\geq0}$.
Choose $m$ so that $n=j+k+\ell_0+mc-j'\geq 1$ holds.
Let $y\in \Sigma(\mca{G})$ contain the path $u$ at the position $j$ and the connecting path at the position $j+k$,
so that $y$ contains the path $u'$ at the position $j+k+\ell_0+mc=j'+n$.
Then $y\in [u]_j$ and $\sigma^{n}(y)\in [u']_{j'}$ hold,
and hence $U\cap \sigma^{-n}(V)\neq \emptyset$, which gives (iii).

Assume (iii), and let $v$ and $w$ be two vertices.
The sets $U=\{y\in\Sigma(\mca{G})\mid y_0=v\}$ and $V=\{y\in\Sigma(\mca{G})\mid y_0=w\}$ are nonempty and open,
so that there are $n\geq 1$ and $y\in U\cap \sigma^{-n}(V)$,
and then $(y_0,y_1,\dots,y_n)$ is a directed path from $v$ to $w$, which gives (i).

For the last assertion, assume (i), and let $p_1,p_2,\dots$ be an enumeration of the directed paths, which form a countable set.
Joining them one after another by connecting directed paths produces a ray $x_0x_1x_2\cdots$,
and extending the ray to the left along incoming edges produces a point $x\in \Sigma(\mca{G})$.
Every directed path $u$ is an initial segment of directed paths of arbitrarily large lengths, each of which occurs in the enumeration,
so that $u$ occurs in $(x_n)_{n\geq 0}$ at arbitrarily large positions,
and hence the forward orbit of $x$ meets every nonempty cylinder $[u]_j$.
The forward orbit of $x$ is therefore dense in $\Sigma(\mca{G})$.
\end{proof}

A graph satisfying these conditions is said to be \emph{irreducible}.
The shift $\sigma$ is said to be \emph{topologically transitive} if condition (iii) holds.
We assume from now on that $\mca{G}$ is irreducible.

\begin{rem} \label{rem.dense.orbit}
In McMullen \cite[Section 2]{McMullen2013CM} the irreducibility of $\mca{G}$ is stated to be equivalent to the existence of a dense orbit of the two-sided shift $\sigma$,
and \Cref{lem.irreducible} replaces that condition by the transitivity (iii).
The two conditions are not equivalent.
Let $\mca{G}$ denote the graph with two vertices $a$ and $b$ and with the three edges $(a,a)$, $(a,b)$, and $(b,b)$.
Then $\Sigma(\mca{G})$ consists of the two constant sequences and of the sequences with one transition,
the orbit of a sequence with one transition is dense, and the graph is not irreducible.
The converse implication does hold: if $\mca{G}$ is irreducible,
then $\sigma$ has a dense forward orbit, which is the last assertion of \Cref{lem.irreducible}.
The results of McMullen \cite{McMullen2013CM} are not affected,
because the irreducibility is an assumption there.
\end{rem}

\subsection{Suspension flows} \label{subsec.suspension}

Let $h:\Sigma(\mca{G})\to (0,\infty)$ be a H\"older continuous function, which is called a \emph{roof function}.
The suspension is
\[\Sigma(\mca{G},h)=\bigl(\Sigma(\mca{G})\times \R\bigr)/\!\sim,
\qquad (\sigma(x),t)\sim (x,t+h(x)),\]
with the flow $s\cdot [x,t]=[x,s+t]$.
Since $\Sigma(\mca{G})$ is compact and $h$ is continuous and positive,
there are constants $0<h_{\min}\leq h_{\max}<\infty$ with $h_{\min}\leq h\leq h_{\max}$, and hence $\Sigma(\mca{G},h)$ is compact.

If $x\in \Sigma(\mca{G})$ satisfies $\sigma^n(x)=x$ for some $n\geq 1$,
then the point $[x,0]$ lies on a closed orbit $\tau$ of the suspension.
If $n$ is the least such integer, then the length of $\tau$ is
\[\ell_h(\tau)=h^{(n)}(x)=\sum_{i=0}^{n-1}h(\sigma^{i}x).\]
Conversely every closed orbit arises in this way,
and two periodic points give the same closed orbit if and only if they lie on the same $\sigma$-orbit.
We write $\tau_1,\tau_2,\dots$ for the closed orbits, ordered by length.

\subsection{Mixing and circle factors} \label{subsec.mixing}

A flow on a space $X$ is said to be \emph{topologically transitive} if for all nonempty open subsets $U$ and $V$ of $X$,
there is $t>0$ such that $(t\cdot U)\cap V\neq \emptyset$.
The flow is said to be \emph{topologically mixing} if for all nonempty open subsets $U$ and $V$ of $X$,
there is $t_0\in \R$ such that $(t\cdot U)\cap V\neq \emptyset$ holds for all $t\geq t_0$.
A topologically mixing flow is topologically transitive.
The flow is said to \emph{have a circle factor} if there are $a>0$ and a continuous map $p:X\to S^1$ such that
\begin{equation} \label{eq.circle}
p(t\cdot x)=e^{iat}p(x)
\end{equation}
holds for all $t\in \R$ and all $x\in X$.
A map $p$ as in \eqref{eq.circle} is precisely a continuous eigenfunction of absolute value one for the flow.
If the flow has a dense orbit, then the absolute value of a continuous eigenfunction is constant,
and a nonzero eigenfunction may be normalized to have absolute value one,
so that the absence of a circle factor is exactly the topological weak mixing property that appears in Parry and Pollicott \cite[Chapter 8]{ParryPollicott1990}.

\subsection{Group extensions and Frobenius classes} \label{subsec.frobenius}

Let $G$ be a compact Lie group, possibly finite, and let $\alpha:\Sigma(\mca{G})\to G$ be H\"older continuous.
The associated principal $G$-bundle over $\Sigma(\mca{G},h)$ is
\[\Sigma(\mca{G},h,\alpha)=\bigl(\Sigma(\mca{G})\times \R\times G\bigr)/\!\sim,
\qquad (\sigma(x),t,g\alpha(x))\sim (x,t+h(x),g),\]
with the flow $s\cdot [x,t,g]=[x,s+t,g]$ lifting the flow on the base and with the commuting left action of $G$ given by $g'\cdot [x,t,g]=[x,t,g'g]$.

Let $\tau$ be a closed orbit of $\Sigma(\mca{G},h)$ of length $\ell_h(\tau)$,
and let $P\in \Sigma(\mca{G},h,\alpha)$ be a point above a point $p\in \tau$.
Since the projection intertwines the two flows, the point $\ell_h(\tau)\cdot P$ again lies above $p$,
and hence there is a unique $g\in G$ with $\ell_h(\tau)\cdot P=g\cdot P$.
If we replace $P$ by $g'\cdot P$, then $\ell_h(\tau)\cdot (g'\cdot P)=g'\cdot (g\cdot P)=(g'gg'^{-1})\cdot (g'\cdot P)$, so that $g$ is replaced by $g'gg'^{-1}$.
Moving $P$ along the flow leaves $g$ unchanged, since the $G$-action commutes with the flow.
Therefore the conjugacy class of $g$ depends only on $\tau$.

\begin{dfn} \label{dfn.frobenius}
The conjugacy class $[\tau]\subset G$ obtained above is called \emph{the Frobenius class} of $\tau$.
\end{dfn}

If $x\in \Sigma(\mca{G})$ satisfies $\sigma^n(x)=x$ and $\tau$ is the corresponding closed orbit,
then undoing the identifications gives the formula
\[[\tau]=\bigl[\alpha(x)\,\alpha(\sigma x)\cdots \alpha(\sigma^{n-1}x)\bigr],\]
in which the factors appear in the order $x,\sigma x,\dots,\sigma^{n-1}x$.

\begin{rem} \label{rem.PP.convention}
In Parry and Pollicott \cite[Chapter 8]{ParryPollicott1990} the extension is defined by $\theta(x,g)=(\sigma x,\alpha(x)g)$ instead,
and the class of a closed orbit is that of the reversed product.
The map $g\mapsto g^{-1}$ identifies the construction above for the cocycle $\alpha$ with the construction of \cite{ParryPollicott1990} for the cocycle $\alpha^{-1}$,
and it preserves the Haar measure and the class functions,
so that the results of \cite[Chapter 8]{ParryPollicott1990} apply to the classes considered here.
\end{rem}

\subsection{The theorem of Parry and Pollicott} \label{subsec.PP}

A function $F:G\to \C$ is called a \emph{class function} if $F(ghg^{-1})=F(h)$ holds for all $g,h\in G$.

\begin{thm}[{\cite[Theorem 8.5]{ParryPollicott1990}}] \label{thm.PP}
Let $h$ be a roof function on $\Sigma(\mca{G})$, and let $\alpha:\Sigma(\mca{G})\to G$ be a map such that $\alpha(x)$ depends only on the coordinates $(x_0,x_1)$ of $x\in \Sigma(\mca{G})$.
Assume that the flow on $\Sigma(\mca{G},h,\alpha)$ is topologically transitive and has no circle factor.
Then for every continuous class function $F\in C(G,\C)$,
\[\lim_{T\to \infty}\frac{1}{\pi(T)}\sum_{\ell_h(\tau)\leq T}F([\tau])
=\int_G F(g)\,dg,\]
where the sum is over the closed orbits $\tau$ with $\ell_h(\tau)\leq T$,
the number of those closed orbits is denoted by $\pi(T)$, and $dg$ is the Haar probability measure.
\end{thm}

The conclusion counts the closed orbits by the threshold $\ell_h(\tau)\leq T$.
It is equivalent to the statement for a sequence, by the following lemma.

\begin{lem} \label{lem.ties}
Assume that $\pi(T)\sim e^{\delta T}/(\delta T)$ holds as $T\to \infty$ for some $\delta>0$.
Then $\#\{\tau\mid \ell_h(\tau)=T\}=o(\pi(T))$ holds as $T\to \infty$.
Consequently the conclusion of \Cref{thm.PP} is equivalent to the equality $\lim_{N\to \infty}\frac{1}{N}\sum_{i=1}^{N}F([\tau_i])=\int_G F(g)\,dg$
for every enumeration $\tau_1,\tau_2,\dots$ of the closed orbits by increasing length,
with an arbitrary order among the closed orbits of equal length.
\end{lem}

\begin{proof}
For every $\varepsilon>0$ we have $\#\{\tau\mid \ell_h(\tau)=T\}\leq \pi(T)-\pi(T-\varepsilon) =(1-e^{-\delta\varepsilon}+o(1))\,\pi(T)$, and $\varepsilon$ is arbitrary.
If the position $N$ lies in the block of the closed orbits of length $T$,
then $N/\pi(T)\to 1$ holds by the first assertion,
and the partial sum up to $N$ differs from the sum over $\ell_h(\tau)\leq T$ by at most $\|F\|_\infty\cdot\#\{\tau\mid \ell_h(\tau)=T\}=o(\pi(T))$.
\end{proof}

The asymptotic $\pi(T)\sim e^{\delta T}/(\delta T)$, with $\delta$ the topological entropy,
is the prime orbit theorem of Parry and Pollicott \cite[Theorem 6.9]{ParryPollicott1990},
which holds when $\Sigma(\mca{G},h)$ is topologically weak-mixing.
Let $G$ be a finite group and let $F$ be the indicator function of a conjugacy class $C\subset G$.
Then the latter limit becomes
\[\lim_{N\to \infty}\frac{\#\{i\leq N\mid [\tau_i]=C\}}{N}=\frac{\#C}{\#G}.\]

Parry and Pollicott state the theorem for every H\"older continuous $\alpha$.
We state it for the maps $\alpha$ that depend only on $(x_0,x_1)$, which include the flat cocycles of \Cref{sec.flat}.
If $\alpha$ is such a map, then the properties of the $L$-functions recalled below are proved by Adachi and Sunada \cite[Theorems 4-6 and 5-5]{AdachiSunada1987JFA}.
Their $L$-functions are attached to a unitary representation of the fundamental group of an irreducible graph
and to a positive Lipschitz length function on the one-sided shift.
A map $\alpha$ as above, composed with a unitary representation of $G$, defines such a representation by multiplication along the loops.
A positive H\"older function on $\Sigma(\mca{G})$ is cohomologous to a positive Lipschitz function
that depends only on the coordinates $x_i$ with $i\geq 0$ \cite[p.~583]{ParryPollicott1983AnnMath},
and cohomologous functions define the same $L$-functions.

The argument of Parry and Pollicott \cite{ParryPollicott1990} is distributed over Chapters 6 and 8, and we recall its steps.
For a finite dimensional unitary representation $\varrho$ of $G$ with character $\chi$, the Artin type $L$-function is
\[L(s,\chi)=\prod_{\tau}\det\bigl(I-\varrho([\tau])e^{-s\ell_h(\tau)}\bigr)^{-1},\]
the product being over the closed orbits, and it converges for $\mathrm{Re}(s)>\delta$,
where $\delta>0$ denotes the topological entropy of the suspension flow.
The normalization of \cite{ParryPollicott1990} corresponds to $\delta=1$ after rescaling $s$.
For the trivial character it reduces to the zeta function of the suspension flow,
which has a simple pole at $s=\delta$ and no other singularity on the line $\mathrm{Re}(s)=\delta$ by \cite[Theorem 6.3]{ParryPollicott1990},
whose weak mixing assumption holds for the base flow because transitivity and the absence of a circle factor pass from the extension to the base.
For a nontrivial irreducible character, \cite[Theorem 8.4]{ParryPollicott1990} states that $L(s,\chi)$ is analytic and nowhere vanishing on $\mathrm{Re}(s)\geq \delta$,
and the proof of that statement is exactly where the two assumptions are used,
namely transitivity of the extension and the absence of a circle factor.

The logarithmic derivative $-L'(s,\chi)/L(s,\chi)$ is the Laplace--Stieltjes transform of the function
$\vartheta_\chi(T)=\sum_{k\ell_h(\tau)\leq T}\ell_h(\tau)\,\chi([\tau]^{k})$,
where the sum is over the closed orbits $\tau$ and the integers $k\geq 1$.
Let $d$ denote the dimension of $\varrho$, and let $\mathbf 1$ denote the trivial character.
If $u\in \C$ satisfies $|u|=1$, then the function $d\,\vartheta_{\mathbf 1}+\mathrm{Re}(u\,\vartheta_\chi)$ is non-decreasing,
because $d+\mathrm{Re}(u\,\chi(g))\geq 0$ holds for every $g\in G$.
The Wiener--Ikehara theorem for non-decreasing functions \cite[Appendix I, (I.8)]{ParryPollicott1990},
applied to $\vartheta_{\mathbf 1}$ and to this function with $u=1$ and with $u=i$,
yields $\vartheta_{\mathbf 1}(T)\sim e^{\delta T}/\delta$ and $\vartheta_\chi(T)=o(e^{\delta T})$ for $\chi\neq \mathbf 1$.
The terms with $k\geq 2$ contribute $O(Te^{\delta T/2})$ to $\vartheta_\chi(T)$, and a partial summation yields
\[\sum_{\ell_h(\tau)\leq T}\chi([\tau])=\Bigl(\int_G\chi\,dg\Bigr)\pi(T)+o(\pi(T)).\]
Hence the equality of \Cref{thm.PP} holds for every finite linear combination of irreducible characters.
It holds for every continuous class function $F$ by the Peter--Weyl theorem,
because the average on the left hand side and the integral on the right hand side
change by at most $\|F-F'\|_\infty$ if $F$ is replaced by $F'$.

\begin{cor} \label{cor.nocircle}
If $\ell_h(\tau_i)/\ell_h(\tau_j)\notin \Q$ holds for some $i$ and $j$,
then $\Sigma(\mca{G},h)$ has no circle factor.
\end{cor}

\begin{proof}
Suppose that $p$ is a circle factor as in \eqref{eq.circle}.
Evaluating \eqref{eq.circle} along a closed orbit $\tau$ gives $e^{ia\ell_h(\tau)}=1$,
that is, $\ell_h(\tau)\in (2\pi/a)\Z$, and hence all the ratios $\ell_h(\tau_i)/\ell_h(\tau_j)$ are rational.
\end{proof}

We say that the roof function $h$ is \emph{non-arithmetic} if $\ell_h(\tau_i)/\ell_h(\tau_j)\notin \Q$ holds for some $i$ and $j$.
That is the assumption under which McMullen obtains the topological mixing \cite[Corollary 2.2]{McMullen2013CM}, and it can be verified on two closed orbits.
If the roof function depends only on the first letter $x_0$, with values $h(v)$ for $v\in V(\mca{G})$,
then the length of a closed orbit is the sum of the values $h(v)$ over the vertices of the corresponding loop.
For the full shift on two letters with values $\alpha$ and $\beta$,
such a roof function is non-arithmetic if and only if $\alpha/\beta\notin \Q$; see \cite[Chapter 10, Example 2]{ParryPollicott1990}.

A generic choice of the roof function $h$, with $\mca{G}$ fixed, is non-arithmetic.
This is the condition behind the phrase \emph{generic metric} in the statements of McMullen \cite{McMullen2013CM}; see the paragraph before \cite[Corollary 2.2]{McMullen2013CM}.
Indeed, fix an exponent of H\"older continuity.
For two distinct closed orbits $\tau_1$ and $\tau_2$, the occupation measures are distinct,
so for each $c\in \Q$ the map $h\mapsto \ell_h(\tau_1)-c\,\ell_h(\tau_2)$ is a nonzero continuous linear functional on the Banach space of H\"older functions, and its kernel is a closed proper subspace.
Hence the set of the positive H\"older functions that are non-arithmetic is residual, by the Baire category theorem.
The ordering by $\ell_h$ with $h$ non-arithmetic is called a \emph{generic ordering}.

The next lemma is used in the proof of \Cref{thm.main.general}, in which the closed orbits of a refined section correspond to those of the given section up to finitely many exceptions.

\begin{lem} \label{lem.finitely.many}
Assume that $\Sigma(\mca{G},h)$ has more than one closed orbit.
Let $c\colon \Sigma(\mca{G})\to S^{1}$ be a H\"older continuous function,
and for a closed orbit $\tau$ of $\Sigma(\mca{G},h)$ let $c(\tau)$ denote the product of the values of $c$ over the corresponding periodic orbit of $\sigma$.
If $c(\tau)=1$ holds for all but finitely many closed orbits $\tau$, then $c(\tau)=1$ holds for every closed orbit $\tau$.
In particular, if $a\in\R$ and $e^{ia\ell_h(\tau)}=1$ holds for all but finitely many closed orbits $\tau$, then $e^{ia\ell_h(\tau)}=1$ holds for every closed orbit $\tau$.
\end{lem}

\begin{proof}
For a directed loop $w$ of $\mca{G}$, let $c(w)$ denote the product of the values of $c$ along the periodic sequence of $w$, over one period $|w|$.
If $w=z^{j}$, then $c(w)=c(z)^{j}$, and $c(w)=c(\tau)$ if $w$ is the primitive directed loop of $\tau$.
By the assumption, there are finitely many primitive directed loops $z_1,\dots,z_r$ such that $c(w)=1$ holds unless the primitive root of $w$ is one of them.
Here the cyclic rotations of a directed loop are listed as distinct directed loops,
so that the primitive root of a directed loop based at a vertex $v$ is a directed loop based at $v$.

Let $\tau$ be a closed orbit with primitive directed loop $z$ based at a vertex $v$,
and let $u$ be a directed loop based at $v$ whose primitive root is not $z$.
Such a loop exists, because $\mca{G}$ is irreducible and $\Sigma(\mca{G},h)$ has more than one closed orbit,
so that $\mca{G}$ is not a cycle and some directed loop based at $v$ is not a power of $z$.
Assume that the primitive root of $z^{k}u$ is among $z_1,\dots,z_r$ for infinitely many $k$.
Then there are $k<k'$, $j<j'$, and $i$ with $z^{k}u=z_i^{\,j}$ and $z^{k'}u=z_i^{\,j'}$.
We obtain $z^{k'-k}z_i^{\,j}=z_i^{\,j'}$ and hence $z^{k'-k}=z_i^{\,j'-j}$.
Since $z$ and $z_i$ are primitive, we have $z_i=z$, and the equality $z^{k}u=z^{j}$ implies $u=z^{j-k}$, which contradicts the choice of $u$.
Hence $c(z^{k}u)=1$ holds for every sufficiently large $k$.

By the H\"older continuity of $c$, the ratio $c(z^{k+1}u)/c(z^{k}u)$ converges to $c(z)$ as $k\to\infty$.
Indeed, the periodic sequence of $z^{k+1}u$ is obtained from that of $z^{k}u$ by inserting one copy of $z$ in the middle of the block $z^{k}$.
Each position of $z^{k+1}u$ outside the inserted copy corresponds to a position of $z^{k}u$, each position of the inserted copy corresponds to a position of the periodic sequence of $z$,
and in both cases the two sequences agree on a block of length at least $k|z|/2-O(1)$ around the corresponding positions.
Hence each of the $(k+1)|z|+|u|$ factors of $c(z^{k+1}u)$ differs from the corresponding factor of $c(z^{k}u)$ or of $c(z)$ by $O(\vartheta^{k|z|/2})$ for some $\vartheta<1$,
and the sum of these errors tends to $0$ as $k\to\infty$.
Since $c(z^{k+1}u)=c(z^{k}u)=1$ holds for every sufficiently large $k$, we obtain $c(z)=1$, that is, $c(\tau)=1$.
The last assertion is the case $c=e^{iah}$.
\end{proof}

\section{Flat bundles over the graph} \label{sec.flat}

The cocycle $\alpha$ that arises in the geometric situation is not an arbitrary H\"older map but comes from a homomorphism on the fundamental group of the graph.
We state what that extra structure gives.

There is a continuous map
\[C:\Sigma(\mca{G},h)\to \mca{G}\]
that sends $\{x\}\times [0,h(x)]$ affinely onto the edge $(x_0,x_1)$.
A closed orbit $\tau$ is sent by $C$ to a directed loop in $\mca{G}$,
and hence to a conjugacy class in $\pi_1(\mca{G})$.
Given a homomorphism $\rho:\pi_1(\mca{G},v_0)\to G$,
we define a cocycle $\tilde{\rho}$ that models the pullback along $C$ of the associated flat bundle.
Choose for every vertex $v$ of $\mca{G}$ a path $\theta_v$ from $v_0$ to $v$, with $\theta_{v_0}$ constant, and put
\[\tilde\rho(x)=\rho\bigl(\theta_{x_0}e_{x_0x_1}\theta_{x_1}^{-1}\bigr).\]
A different choice of the paths $\theta_v$ replaces $\tilde\rho$ by a cocycle that differs from it by a coboundary with values in $G$.
For a point $x\in \Sigma(\mca{G})$ with $\sigma^{n}(x)=x$, the products along the closed path telescope as
\[\tilde\rho(x)\,\tilde\rho(\sigma x)\cdots \tilde\rho(\sigma^{n-1}x)
=\rho\bigl(\theta_{x_0}\,e_{x_0x_1}e_{x_1x_2}\cdots e_{x_{n-1}x_0}\,
\theta_{x_0}^{-1}\bigr),\]
so that the Frobenius class of a closed orbit $\tau$ equals $\rho(C(\tau))$ and does not depend on the choice.

\begin{lem} \label{lem.dense}
Let $\rho:\pi_1(\mca{G})\to G$ be a homomorphism with dense image.
\begin{enumerate}
\item For every vertex $v$ of $\mca{G}$, the elements $\rho(\theta_v\,\mu\,\theta_v^{-1})$, where $\mu$ runs through the directed loops based at $v$, form a dense subset of $G$.
\item The flow on $\Sigma(\mca{G},h,\tilde\rho)$ is topologically transitive and has a dense orbit.
\end{enumerate}
\end{lem}

\begin{proof}
Fix a vertex $v$ of $\mca{G}$, and let $S\subset \pi_1(\mca{G},v)$ denote the subsemigroup consisting of the classes of the directed loops based at $v$.
For a path $p$ in $\mca{G}$ from a vertex $a$ to a vertex $b$,
put $\tilde\rho(p)=\rho(\theta_{a}\,p\,\theta_{b}^{-1})\in G$.
The map $\tilde\rho$ is multiplicative on composable paths,
and for a directed path $p$ the element $\tilde\rho(p)$ is the product of the values of the cocycle $\tilde\rho$ along $p$.
Moreover $\tilde\rho(\pi_1(\mca{G},v))=\rho(\pi_1(\mca{G},v_0))$,
so that $\tilde\rho(\pi_1(\mca{G},v))$ is dense in $G$.
Assertion (1) says that $\tilde\rho(S)$ is dense in $G$.

Let $g\in \tilde\rho(\pi_1(\mca{G},v))$, and choose a loop $c=(w_0,w_1,\dots,w_n)$ with $w_0=w_n=v$ and $\tilde\rho(c)=g$,
where consecutive vertices are joined by an edge in one direction or the other.
Let $m(c)$ denote the number of steps that are traversed backwards.
If $m(c)=0$, then $c$ is directed and $g\in \tilde\rho(S)$.
Suppose $m(c)\geq 1$, and let $e=(w_{i+1},w_i)\in E(\mca{G})$ be an edge traversed backwards.
Since $\mca{G}$ is irreducible, \Cref{lem.irreducible} provides a directed loop $\mu$ based at $w_{i+1}$ that begins with $e$.
For $k\geq 1$, insert $\mu^{k}$ after the backward step,
and cancel that step against the initial edge $e$ of $\mu$.
The resulting loop $c_k$ satisfies $m(c_k)=m(c)-1$ and $\tilde\rho(c_k)=g_1\tilde\rho(\mu)^{k}g_2$,
where $g_1$ is the value of $\tilde\rho$ on the part of $c$ up to and including the backward step
and $g_2$ is the value on the rest, so that $g_1g_2=\tilde\rho(c)=g$.
Since $G$ is compact, the closure of $\{\tilde\rho(\mu)^k\mid k\geq 1\}$ is a compact subgroup and contains the identity,
so that there are arbitrarily large $k$ with $\tilde\rho(\mu)^k$ arbitrarily close to the identity.
For such $k$ we obtain $\tilde\rho(c_k)$ arbitrarily close to $g$.
Repeating the procedure $m(c)$ times, we conclude that $g$ lies in the closure of $\tilde\rho(S)$.
This proves (1).

We prove (2), first for a finite group $G$.
Consider the graph $\mca{G}\times G$ with vertex set $V(\mca{G})\times G$
and with an edge from $(v,g)$ to $(w,g\tilde\rho(v,w))$ whenever $(v,w)\in E(\mca{G})$;
its subshift is the extension $\Sigma(\mca{G})\times G$,
and its suspension by the roof $h$, pulled back from the base, is $\Sigma(\mca{G},h,\tilde\rho)$.
That graph is irreducible.
Indeed, given $(v,g)$ and $(w,g')$, choose by irreducibility of $\mca{G}$ a directed path from $v$ to $w$, with holonomy $g_0$ say,
and then a directed loop $\mu$ at $w$ with $\tilde\rho(\mu)=g_0^{-1}g^{-1}g'$,
which exists by (1) at the vertex $w$, because a dense subset of the finite group $G$ is $G$ itself.
The concatenation joins $(v,g)$ to $(w,g')$.
By \Cref{lem.irreducible} the extension shift $\sigma_G$ is topologically transitive and has a dense forward orbit.
The sets $\{[y,t]\mid y\in \tilde C,\ |t-t_0|<\epsilon\}$,
for cylinders $\tilde C$ of the extension shift,
form a basis of the topology of $\Sigma(\mca{G},h,\tilde\rho)$.
Given two such sets, the proof of \Cref{lem.irreducible} provides $n\geq 1$ with $n\,h_{\min}>t_0-t_0'$ and a point $y\in \tilde C\cap \sigma_G^{-n}(\tilde C')$,
and the flow for the time $h^{(n)}(y)+t_0'-t_0>0$ maps $[y,t_0]$ into the second set,
so that the flow is topologically transitive.
Likewise the forward orbit of $[y,0]$, for a point $y$ with dense forward $\sigma_G$-orbit, meets every basis set, and hence is dense.

For a compact group $G$ we approximate the required holonomies, as in McMullen \cite[Lemma 3.2]{McMullen2013CM}.
For a directed path $u=(u_{-k},\dots,u_{k})$, let $C_u$ denote the cylinder of the points $y\in\Sigma(\mca{G})$ with $(y_{-k},\dots,y_{k})=u$.
The sets $\{[y,t,g']\mid y\in C_u,\ |t-t_0|<\epsilon,\ g'\in U\}$, for such cylinders $C_u$ and for nonempty open subsets $U$ of $G$,
form a basis of the topology of $\Sigma(\mca{G},h,\tilde\rho)$.

Let two such sets be given, the first with $u=(u_{-k},\dots,u_{k})$, $t_0$, and $U$, and the second with $u'=(u'_{-k'},\dots,u'_{k'})$, $t_0'$, and $U'$,
and let $g\in U$.
Choose a directed path $q$ from $u_{k}$ to $u'_{-k'}$,
let $g_1$ denote the value of $\tilde\rho$ on the directed path $(u_0,\dots,u_k)$ followed by $q$,
and let $g_2$ denote the value of $\tilde\rho$ on the directed path $(u'_{-k'},\dots,u'_{0})$.
By (1), there is a directed loop $\mu$ based at $u'_{-k'}$ with $g\,g_1\,\tilde\rho(\mu)\,g_2\in U'$.
The loop $\mu$ may be taken arbitrarily long, because it may be followed by a power $\mu_1^{j}$ of a directed loop $\mu_1$ based at $u'_{-k'}$,
and $\tilde\rho(\mu_1)^{j}$ is arbitrarily close to the identity for arbitrarily large $j$, as in the proof of (1).
Let $w$ denote the directed path obtained by concatenating $(u_0,\dots,u_k)$, $q$, $\mu$, and $(u'_{-k'},\dots,u'_{0})$, let $n$ denote its length, and assume $n\,h_{\min}>t_0-t_0'$.
Let $y\in\Sigma(\mca{G})$ be a point that contains $u$ at the positions $-k,\dots,k$, the path $w$ at the positions $0,\dots,n$, and $u'$ at the positions $n-k',\dots,n+k'$.
Then $y\in C_u$ and $\sigma^{n}(y)\in C_{u'}$,
and the product $\tilde\rho(y)\tilde\rho(\sigma y)\cdots\tilde\rho(\sigma^{n-1}y)$ equals $g_1\tilde\rho(\mu)g_2$.
Hence the flow for the time $h^{(n)}(y)+t_0'-t_0>0$ maps the point $[y,t_0,g]$ of the first set to the point $[\sigma^{n}(y),t_0',g\,g_1\tilde\rho(\mu)g_2]$ of the second set,
so that the flow is topologically transitive.

The space $\Sigma(\mca{G},h,\tilde\rho)$ is compact and metrizable.
For a countable basis $(V_i)_i$ of its topology, the sets $\bigcup_{t\in\R}t\cdot V_i$ are open, and they are dense by the transitivity.
By the Baire category theorem their intersection is nonempty, and the orbit of each of its points is dense.
\end{proof}

\begin{thm} \label{thm.flat}
Let $G$ be a compact Lie group whose identity component $G^{0}$ is semisimple,
and let $\rho:\pi_1(\mca{G})\to G$ be a homomorphism with dense image.
If the flow on $\Sigma(\mca{G},h)$ has no circle factor,
then the conclusion of \Cref{thm.PP} holds for $\Sigma(\mca{G},h,\tilde\rho)$.
\end{thm}

\begin{proof}
By \Cref{lem.dense} the flow on $\Sigma(\mca{G},h,\tilde\rho)$ is topologically transitive,
so by \Cref{thm.PP} it suffices to exclude a circle factor.
Suppose that $p$ is a circle factor as in \eqref{eq.circle}.
The actions of $G$ and of $\R$ on $\Sigma(\mca{G},h,\tilde\rho)$ commute,
and hence for each $g\in G$ the function $x\mapsto p(g\cdot x)/p(x)$ is constant along the flow lines.
Since some flow line is dense and the function is continuous,
that function is constant, say equal to $\chi(g)$.
The map $\chi:G\to S^1$ is a continuous homomorphism.
As $G^{0}$ is semisimple, it has no nontrivial characters,
so that $\chi$ factors through the finite group $G/G^{0}$ and its image is finite of some order $n$.
Then $p^{n}$ is $G$-invariant, so it descends to a circle factor of $\Sigma(\mca{G},h)$,
which contradicts the assumption on the base flow.
\end{proof}

For a finite group $G$, the assumption on $G^{0}$ is vacuous, and dense image means that $\rho$ is surjective.

\begin{rem}[{\cite[Proposition C]{AdachiSunada1987JFA}, \cite[Theorem 3.1]{McMullen2013CM}}] \label{rem.flat.finite}
If $G$ is finite, then the assumption of \Cref{thm.flat} on the base flow is implied by the topological mixing
assumed in \cite[Proposition C]{AdachiSunada1987JFA} and in \cite[Theorem 3.1]{McMullen2013CM}.
The statements of \cite[Theorems 8.6 and 9.5 (b)]{ParryPollicott1990} assume only the topological transitivity.
They are derived there from \cite[Theorem 8.5]{ParryPollicott1990} by the weak mixing of the base flow,
and the remark after \cite[Theorem 9.5]{ParryPollicott1990} refers to the weak mixing of the base flow as the condition of part (b).
The following example shows that this assumption is necessary.
Let $\mca{G}$ be the graph with two vertices and all four edges, let $h\equiv 1$, let $G=\Z/2\Z$,
and let $\rho$ send the class of a loop with $n$ edges to $n$ modulo $2$.
Then the flow on $\Sigma(\mca{G},h)$ has a circle factor,
the flow on $\Sigma(\mca{G},h,\tilde\rho)$ is topologically transitive,
the Frobenius class of a closed orbit $\tau$ is trivial if and only if $\ell_h(\tau)$ is even,
and the proportion of the closed orbits $\tau$ with $\ell_h(\tau)\leq T$ and trivial Frobenius class
tends to $2/3$ as $T\to \infty$ through even integers and to $1/3$ through odd integers.
The base flows with a circle factor are treated in \cite[Theorem 3 (ii)]{ParryPollicott1986ETDS},
where the smallest values of $a$ in \eqref{eq.circle} for the extension and for the base flow are denoted by $a$ and $Na$.
The example, in which $N=2$, shows that an assumption on $N$ is needed there.
\end{rem}

\section{Markov sections} \label{sec.markov}

We now pass to the geometric side.
In this section $M$ is a 3-manifold and $\varphi=(\varphi_t)_{t\in\R}$ is a flow on $M$ that is Anosov,
or more generally pseudo-Anosov, so that stable and unstable laminations $W^{\rm s}$ and $W^{\rm u}$ are defined.
We write $\varphi_I(x)=\{\varphi_t(x)\mid t\in I\}$ for an interval $I$.

\subsection{Rectangles} \label{subsec.rectangle}

A subset $\mca{R}\subset M$ is called a \emph{rectangle} if $\mca{R}$ is an embedded surface transverse to $\varphi$ that is equipped with a homeomorphism $\mca{R}\cong A\times B$ onto a product of two intervals
such that $A\times\{b\}$ lies in a leaf of $W^{\rm s}$ and $\{a\}\times B$ lies in a leaf of $W^{\rm u}$ for every $a\in A$ and every $b\in B$.
We call $A\times\{b\}$ a \emph{stable arc} and $\{a\}\times B$ an \emph{unstable arc}.
An \emph{edge} of $\mca{R}$ is an arc $A\times\{b\}$ with $b$ an endpoint of $B$ belonging to $B$,
or an arc $\{a\}\times B$ with $a$ an endpoint of $A$ belonging to $A$.
We write $\partial \mca{R}$ for the union of the edges and $\intr \mca{R}=\mca{R}-\partial \mca{R}$;
a compact rectangle has four edges, and an interval open at an end contributes no edge there.
In McMullen \cite[Section 4]{McMullen2013CM} the intervals $A$ and $B$ are compact;
we do not assume that here, because the modular flow requires the noncompact case.

Flowing forward contracts the $A$ direction and expands the $B$ direction.
Consider the geodesic flow on the unit tangent bundle of a hyperbolic surface.
A rectangle is then described by the pair of endpoints of the corresponding geodesic,
the stable arcs being the sets on which the forward endpoint is constant.

\subsection{The Markov condition} \label{subsec.markov.cond}

Let $(\mca{R}_i)_{i=1}^{n}$ be a finite family of rectangles $\mca{R}_i\cong A_i\times B_i$ with mutually disjoint interiors, and put $\mca{R}=\bigcup_i \mca{R}_i$.
Two rectangles are allowed to share an edge, since only their interiors are used below.

\begin{dfn} \label{dfn.markov}
Let $\ell^{+}\subset \mca{R}$ be a union of finitely many stable arcs,
and let $\ell^{-}\subset \mca{R}$ be a union of finitely many unstable arcs.
The family $(\mca{R}_i)_{i=1}^{n}$ is called a \emph{Markov section with escape locus $(\ell^{+},\ell^{-})$} if the following four conditions hold.
\begin{itemize}
\item[\textup{(i)}] For $x\in \mca{R}$, there is $t>0$ with $\varphi_t(x)\in
\mca{R}$ if and only if $x\notin \ell^{+}$. For $x\in \mca{R}-\ell^{+}$, the least such $t$ is denoted by $r(x)$, and the first return map is $f(x)=\varphi_{r(x)}(x)$.
\item[\textup{(ii)}] For $x\in \mca{R}$, there is $t<0$ with $\varphi_t(x)\in
\mca{R}$ if and only if $x\notin \ell^{-}$.
\item[\textup{(iii)}] For all $i$ and $j$, put
$\mca{R}_{ij}=\{x\in \intr \mca{R}_i\mid f(x)\ \text{is defined and}\ f(x)\in \intr \mca{R}_j\}$.
Then there are possibly empty open intervals $A_{ij}\subset A_j$
and $B_{ij}\subset B_i$ with $\mca{R}_{ij}=\intr(A_i)\times B_{ij}$ and $f(\mca{R}_{ij})=A_{ij}\times \intr(B_j)$,
where $\intr$ denotes the interior of an interval and the equalities are read in the product coordinates of $\mca{R}_i$ and of $\mca{R}_j$ respectively.
\item[\textup{(iv)}] The return time $r$ is continuous on each $\mca{R}_{ij}$.
\end{itemize}
If $\ell^{+}=\ell^{-}=\emptyset$ and every $\mca{R}_i$ is compact,
then the family is called a \emph{Markov section}.
\end{dfn}

The set $\mca{R}_{ij}$ is called the \emph{transition set} from $\mca{R}_i$ to $\mca{R}_j$.
Condition (iii) says that $\mca{R}_{ij}$ is full in the contracting direction and that its image is full in the expanding direction.
In the definition of Iakovoglou \cite{Iakovoglou2025CRM},
the closure of each connected component of $f(\intr \mca{R}_i)\cap \intr \mca{R}_j$ is required to be a vertical subrectangle of $\mca{R}_j$, and several components are allowed.
If several components occur, then the graph has several edges with the same pair of endpoints,
and the convention of \Cref{subsec.graph} is restored by passing to the edge shift.
For the section constructed in \Cref{sec.modular} each transition is connected,
so that no two edges there have the same pair of endpoints.

The same structure is called a Markov partition in \cite{Ratner1973Markov, Brunella1995JMSJ, Iakovoglou2025CRM} and a Markov section in \cite{McMullen2013CM}.
We follow the second usage, under which the rectangles lie in a surface transverse to the flow rather than partition the manifold.
The family is not a cross-section either, in the usual sense of that word,
since some orbits of $\varphi$ never meet $\mca{R}$ and the forward orbit of a point of $\ell^{+}$ does not meet $\mca{R}$ again.

The existence of Markov sections was established in the following sequence of works.
Ratner \cite{Ratner1973Markov} treated topologically transitive Anosov flows on compact manifolds of every dimension.
Brunella \cite[Theorem 2]{Brunella1995JMSJ} treated topologically transitive flows that are expansive in the sense of \cite[Definition 1.1]{InabaMatsumoto1990JJM},
by first producing a surface of section \cite[Theorem 1]{Brunella1995JMSJ} whose first return map is conjugate to a pseudo-Anosov map,
and then applying the Markov partitions of Fathi and Shub \cite[Expos\'e 10]{FathiShub1979}.
Iakovoglou \cite{Iakovoglou2025CRM} removed the transitivity assumption.
He obtained in addition a Markov partition whose stable and unstable boundaries are contained in $\bigcup_{\gamma\in \Gamma}F^{\rm s}(\gamma)$ and in $\bigcup_{\gamma\in \Gamma}F^{\rm u}(\gamma)$,
where $\Gamma$ is any given finite family of periodic orbits that contains every singular periodic orbit and whose stable and unstable saturations are dense.
The definition of a \emph{pseudo-Anosov flow} varies slightly in the literature,
in whether a rate of contraction is required.
The definition of Fenley \cite[p.~6]{Fenley2008ASPM} is topological,
while the definition of Calegari \cite[Definition 6.41]{DCalegari2007book} requires the flow to contract and to expand at exponential rates.
The results of this article are stated for a flow admitting a Markov section with the properties collected in \Cref{thm.main.general}, so that they do not depend on that choice.
In dimension three, an expansive flow is pseudo-Anosov in the topological sense,
by Oka \cite{Oka1990Osaka} and by Inaba and Matsumoto \cite[Theorem 1.5]{InabaMatsumoto1990JJM}.
The prongs then satisfy $p\geq 3$, so that flows with 1-pronged singular orbits are treated neither by Brunella \cite{Brunella1995JMSJ} nor by Iakovoglou \cite{Iakovoglou2025CRM}.

\begin{rem} \label{rem.bounded}
In \cite[Section 3]{Brunella1995JMSJ} and in \cite{Iakovoglou2025CRM},
a Markov partition is required to meet every orbit within a uniformly bounded time,
that is, $\varphi_{[0,\alpha]}(\mca{R})=M$ for some $\alpha>0$.
That requirement forces $M$ to be compact, and the section of \Cref{sec.modular} fails it in two ways:
the return time is unbounded above, and the escape locus is not empty.
\end{rem}

\subsection{The symbolic model} \label{subsec.coding}

Let $(\mca{R}_i)_{i=1}^{n}$ be a Markov section with escape locus $(\ell^{+},\ell^{-})$,
and let $\mca{G}$ denote the graph with $V(\mca{G})=\{v_i\}_{i=1}^n$ and $E(\mca{G})=\{e_{ij}\mid \mca{R}_{ij}\neq \emptyset\}$.
Let $y\in \mca{R}$ be a point whose orbit meets $\mca{R}$ at arbitrarily large positive and negative times.
An \emph{itinerary} of $y$ is a sequence $x\in \Sigma(\mca{G})$ with $f^{i}(y)\in \mca{R}_{x_i}$ for every $i\in \Z$,
and it is unique unless the orbit of $y$ meets $\partial \mca{R}$.
The itinerary of a point determines a map
\[p:\Sigma(\mca{G})\to \mca{R},\]
characterized by $p(x)\in \mca{R}_{x_0}$ and $p(\sigma x)=f(p(x))$.
If there is no escape locus and every $\mca{R}_i$ is compact,
then $p$ is surjective, there is a unique continuous roof function $h$ on $\Sigma(\mca{G})$ with $h(x)=r(p(x))$ whenever $p(x)\in \bigcup_{ij}\mca{R}_{ij}$,
and taking suspensions gives a continuous map
\[\pi:\Sigma(\mca{G},h)\to M\]
that takes the symbolic flow to $\varphi$ and sends a closed orbit $\tau$ to a closed orbit of $\varphi$ of period $\ell_h(\tau)$.
If moreover $\varphi_{[0,\alpha]}(\mca{R})=M$ holds for some $\alpha>0$, then $\pi$ is surjective.
If in addition the return map contracts the local stable arcs and expands the local unstable arcs at exponential rates,
as it does for an Anosov flow, then $p$ and $h$ are H\"older continuous;
see Bowen \cite[p.~176]{Bowen1972JDE}, \cite[Lemma 2.2 and p.~436]{Bowen1973AJM},
and McMullen \cite[Section 4]{McMullen2013CM}.
Expansiveness alone does not give exponential rates.

In the presence of an escape locus the situation is different,
and $p$ is not necessarily defined on all of $\Sigma(\mca{G})$,
since the itinerary of a point of $\mca{R}$ terminates in forward time if its orbit meets $\ell^{+}$,
and terminates in backward time if its orbit meets $\ell^{-}$.
For the modular flow the sequences that occur as itineraries are exactly the ones without a constant tail (\Cref{lem.tails}).
We therefore do not use $p$ globally.
The two facts that we do use concern closed orbits only,
namely the correspondence of \Cref{lem.orbits.bijection} and the identification of monodromies of \Cref{lem.homotopic},
and the roof function is chosen rather than deduced from the return time.

A periodic point $x$ of $\Sigma(\mca{G})$ is said to be \emph{realized} if some closed orbit of $\varphi$ has $x$ as its itinerary.
If all the rectangles are compact and there is no escape locus, then every periodic point is realized:
in each of the two coordinates the composition along one period,
taken in its contracting direction, maps a compact interval into itself and therefore has a fixed point,
and the pair of fixed coordinates is a point of $\mca{R}$ fixed by the return map.
Without those assumptions the realization of the periodic points is not a consequence of the definition, and \Cref{thm.main.general} assumes it;
for the section of \Cref{sec.modular} exactly two periodic points are not realized, by \Cref{prop.modular} (4).

\begin{lem} \label{lem.arcs}
The first return map $f$ maps the intersection of a stable arc with $\mca{R}_{ij}$ into a single stable arc of $\mca{R}_j$,
and the intersection of an unstable arc with $\mca{R}_{ij}$ into a single unstable arc of $\mca{R}_j$.
Moreover, distinct stable arcs meeting $\mca{R}_{ij}$ are mapped into distinct stable arcs of $\mca{R}_j$, and likewise for unstable arcs.
\end{lem}

\begin{proof}
The return time is continuous on $\mca{R}_{ij}$ by condition (iv) of \Cref{dfn.markov},
and hence $f$ is continuous on $\mca{R}_{ij}$.
The leaves of $W^{\rm s}$ and of $W^{\rm u}$ are invariant under the flow,
so $f$ maps a subset of a leaf into the same leaf.
The intersection of a leaf of $W^{\rm s}$ with $\mca{R}_j$ is a union of complete stable arcs,
and it contains no product $A_j\times I$ with $I$ an interval,
since a leaf meets the transverse surface $\mca{R}_j$ in a set with empty interior.
A connected subset of such a union therefore lies in a single stable arc,
and the image of the intersection of a stable arc with $\mca{R}_{ij}$ is connected.
The unstable case is symmetric.
For the last assertion, suppose that two stable arcs meeting $\mca{R}_{ij}$ are mapped into one stable arc of $\mca{R}_j$,
and pick two points of the two arcs on one unstable arc.
Their images lie on one unstable arc by the first part and on one stable arc by assumption,
hence coincide, contradicting the injectivity of the first return map.
\end{proof}

\begin{lem} \label{lem.orbits.bijection}
Assume that every compact orbit segment meets $\mca{R}$ in finitely many points.
Then at most one closed orbit of $\varphi$ meets a given edge of a given rectangle.
Consequently the assignment that sends a closed orbit of $\varphi$ meeting $\intr \mca{R}$ to its itinerary is a bijection onto the set of realized periodic points of $\Sigma(\mca{G})$ modulo the shift,
once finitely many closed orbits are removed from each side.
\end{lem}

\begin{proof}
The itinerary of a point of $\mca{R}$ is unique unless the orbit meets $\partial \mca{R}$.
A closed orbit meets $\mca{R}$ in finitely many points by the assumption, and $f$ permutes them cyclically,
so that the itinerary of each of them is a periodic point.
Let $J$ be an unstable arc of some $\mca{R}_i$.
All points of $J$ lie in one leaf of $W^{\rm u}$, so their backward orbits are asymptotic to each other,
and two distinct closed orbits of a flow are never backward asymptotic,
being disjoint compact sets at a positive distance;
therefore at most one closed orbit meets $J$.
The same argument in forward time applies to a stable arc.
Since there are finitely many edges, only finitely many closed orbits of $\varphi$ meet $\partial \mca{R}$,
and hence only finitely many have a non-unique itinerary.

We show next that two distinct closed orbits avoiding $\partial \mca{R}$ have distinct itineraries.
Suppose that closed orbits $\tau$ and $\tau'$ avoid $\partial \mca{R}$ and share an itinerary $x$,
let $n\in\Z_{\geq 1}$ be a common multiple of the numbers of points of $\tau\cap \mca{R}$ and of $\tau'\cap \mca{R}$, so that $\sigma^{n}(x)=x$,
and let $p=(a_p,b_p)\in \tau\cap \mca{R}_{x_0}$ and $q=(a_q,b_q)\in \tau'\cap \mca{R}_{x_0}$ be points with itinerary $x$, so that $f^{n}(p)=p$ and $f^{n}(q)=q$,
written in the product coordinates of $\mca{R}_{x_0}$.
If $b_p=b_q$, then $p$ and $q$ lie on one stable arc, the two orbits are forward asymptotic,
and the positive distance between distinct closed orbits gives $\tau=\tau'$;
if $a_p=a_q$, then the same argument in backward time applies.
Assume therefore that $a_p\neq a_q$ and $b_p\neq b_q$ hold, and put $z=(a_q,b_p)\in \mca{R}_{x_0}$.
By condition (iii) of \Cref{dfn.markov}, membership in $\mca{R}_{ij}=\intr(A_i)\times B_{ij}$ amounts to one condition on each coordinate,
so $q\in \mca{R}_{x_0x_1}$ gives $a_q\in \intr(A_{x_0})$, $p\in \mca{R}_{x_0x_1}$ gives $b_p\in B_{x_0x_1}$, and hence $z\in \mca{R}_{x_0x_1}$.
By \Cref{lem.arcs}, the point $f(z)$ lies both on the stable arc through $f(p)$ and on the unstable arc through $f(q)$,
and is therefore the intersection point of those two arcs.
The same membership argument applies at every step,
so by induction $f^{i}(z)$ is defined and is the intersection of the stable arc through $f^{i}(p)$ with the unstable arc through $f^{i}(q)$, for every $i\geq 0$.
For $i=n$ we obtain $f^{n}(z)=(a_q,b_p)=z$, so $z$ lies on a closed orbit of $\varphi$.
That closed orbit is forward asymptotic to $\tau$,
since $z$ and $p$ lie on one stable arc, and backward asymptotic to $\tau'$,
since $z$ and $q$ lie on one unstable arc; hence it equals both of them, and $\tau=\tau'$, a contradiction.

Finally, a realized periodic point is by definition the itinerary of a closed orbit.
A closed orbit meeting $\partial \mca{R}$ has only finitely many itineraries,
since at most finitely many symbols are possible at each of its finitely many crossings.
Hence, after removing the finitely many closed orbits that meet $\partial \mca{R}$
and the finitely many periodic points that occur among their itineraries, the assignment is a bijection as stated.
\end{proof}

\subsection{Transitivity and irreducibility} \label{subsec.irreducible}

The graphs of the sections of this article are shown to be irreducible by the following criterion.

\begin{lem} \label{lem.flow.irreducible}
Assume that the flow $\varphi$ is topologically transitive and that every compact orbit segment meets $\mca{R}$ in finitely many points.
Then the graph $\mca{G}$ is irreducible.
\end{lem}

\begin{proof}
Let $v_i$ and $v_j$ be two vertices of $\mca{G}$, and fix $\epsilon>0$.
The sets $U=\varphi_{(-\epsilon,\epsilon)}(\intr \mca{R}_i)$ and $V=\varphi_{(-\epsilon,\epsilon)}(\intr \mca{R}_j)$ are open and nonempty,
since $\intr \mca{R}_i$ and $\intr \mca{R}_j$ are nonempty surfaces transverse to the flow.
We apply the transitivity to $U$ and $\varphi_{-2\epsilon}(V)$, add $2\epsilon$ to the time,
and obtain $t>2\epsilon$ such that the set of points $w\in U$ with $\varphi_{t}(w)\in V$ is nonempty; that set is open.
The set $\partial \mca{R}\cup \ell^{+}\cup \ell^{-}$ is a countable union of compact arcs contained in $\mca{R}$.
Flowing one of those arcs over a compact interval of times yields a compact set,
and the flow map is locally injective on the product of the arc and the interval,
since the arc is transverse to the flow, so the image has topological dimension at most two and has empty interior in the 3-manifold $M$.
Therefore the set of points $w$ such that $\varphi_{[-\epsilon,\,t+\epsilon]}(w)$ meets $\partial \mca{R}\cup \ell^{+}\cup \ell^{-}$ is a countable union of closed sets with empty interior,
and since $M$ is a Baire space, we may choose a point $w\in U$ with $\varphi_{t}(w)\in V$ outside that union.
Write $w=\varphi_{s}(y)$ with $y\in \intr \mca{R}_i$ and $|s|<\epsilon$,
and $\varphi_{t}(w)=\varphi_{s'}(z)$ with $z\in \intr \mca{R}_j$ and $|s'|<\epsilon$,
so that $z=\varphi_{T}(y)$ holds with $T=t+s-s'>0$.
The segment $\varphi_{[0,T]}(y)$ is contained in $\varphi_{[-\epsilon,\,t+\epsilon]}(w)$,
so it meets $\mca{R}$ in finitely many points $y=y_0,y_1,\dots,y_m=z$, listed by increasing time,
and all of those points lie in $\intr \mca{R}-(\ell^{+}\cup \ell^{-})$.
Let $k<m$, and let $u_k$ denote the vertex with $y_k\in \intr \mca{R}_{u_k}$.
The first return $f(y_k)$ is defined by condition (i) of \Cref{dfn.markov}, since $y_k\notin \ell^{+}$,
and equals $y_{k+1}$, since the segment meets $\mca{R}$ at no intermediate time.
Hence $y_k\in \mca{R}_{u_{k}u_{k+1}}$ holds, so that $e_{u_ku_{k+1}}$ is an edge of $\mca{G}$,
and $u_0u_1\cdots u_m$ is a directed path from $v_i$ to $v_j$.
\end{proof}

For the section of \Cref{sec.modular} the second assumption holds,
because a crossing of $\mca{R}$ is a crossing of an edge of the Farey tessellation
and a geodesic segment of finite length crosses finitely many edges of a locally finite tessellation.
The refined sections of \Cref{subsec.refine} satisfy the second assumption as well,
since the refinement does not change the surface $\mca{R}$.
The first assumption holds for the geodesic flow of a hyperbolic orbifold of finite area,
because the Liouville probability measure is ergodic,
by the theorem of Hopf \cite{Hopf1939Leipzig}, and has full support:
for nonempty open subsets $U$ and $V$ of $M$,
the set of points whose forward orbit meets $V$ is invariant up to null sets and has positive measure,
hence has full measure, and therefore meets $U$.

\section{The spine and the surjectivity lemma} \label{sec.spine}

\subsection{Flow boxes and the spine}
\label{subsec.boxes}

Keep the notation of \Cref{sec.markov}, and allow an escape locus.
The sets
\[D^{+}=\{w\in M\mid \varphi_t(w)\notin \mca{R}\ \text{for all}\ t>0\},
\quad
D^{-}=\{w\in M\mid \varphi_t(w)\notin \mca{R}\ \text{for all}\ t<0\}\]
are attached to the section $\mca{R}$ itself and not to a particular decomposition of it into rectangles,
which matters because the refinements below change the decomposition and not the section.
We have $D^{+}\cap \mca{R}=\ell^{+}$ and $D^{-}\cap \mca{R}=\ell^{-}$,
by the equivalences of \Cref{dfn.markov} (i) and (ii).

Following McMullen \cite[(5.1)]{McMullen2013CM}, put
\[E_{ij}=\bigcup_{x\in \mca{R}_{ij}}\varphi_{[0,r(x)]}(x),
\qquad
\mca{U}=\Bigl(\bigcup_{i}\intr
\mca{R}_i\Bigr)\cup\Bigl(\bigcup_{i,j}E_{ij}\Bigr).\]
For $i\neq j$ the set $E_{ij}$ is homeomorphic to $\mca{R}_{ij}\times[0,1]$ by the map $(x,t)\mapsto \varphi_{t\,r(x)}(x)$ and is contractible.
For $i=j$ that map identifies $(x,1)$ with $(f(x),0)$ whenever $f(x)\in \mca{R}_{ii}$,
and $E_{ii}$ is not necessarily contractible.
For each pair $(i,j)$ with $\mca{R}_{ij}\neq \emptyset$, choose a point $y_{ij}\in \mca{R}_{ij}$,
and let $s_{ij}=\varphi_{[0,r(y_{ij})]}(y_{ij})$ be the first return segment through it.
Since $\mca{R}_{ij}$ is connected, any two first return segments in $E_{ij}$ are homotopic through first return segments.

\begin{rem} \label{rem.contractible}
In McMullen \cite[Section 5]{McMullen2013CM} the sets $E_{ij}$ are stated to be contractible without the restriction $i\neq j$,
and the uniqueness of the homotopy class of the embedding of the graph is deduced from that.
For the section of \Cref{sec.modular}, both sets $E_{11}$ and $E_{22}$ have the homotopy type of a circle.
The homotopies through first return segments repair the argument,
and the results of \cite{McMullen2013CM} are not affected.
\end{rem}

Choose $x_i\in \mca{R}_{ij}\cap f(\mca{R}_{ki})$ for some $j$ and $k$,
which is possible by \Cref{dfn.markov} (iii).
Define a map
\[\iota:\mca{G}\to \mca{U}\]
by sending the vertex $v_i$ to $x_i$ and the edge $e_{ij}$ to the concatenation of a path in $\intr \mca{R}_i$ from $x_i$ to $y_{ij}$, the segment $s_{ij}$,
and a path in $\intr \mca{R}_j$ from $f(y_{ij})$ to $x_j$.
The rectangles are contractible and the segments are homotopic through first return segments,
so the homotopy class of $\iota$ does not depend on the choices.
If $L$ is a union of closed orbits with $L\cap \mca{R}\subset \partial \mca{R}$, then $L\cap \mca{U}=\emptyset$,
because a point of $\mca{U}$ lies on an orbit that meets $\bigcup_i \intr \mca{R}_i$,
whereas an orbit contained in $L$ meets $\mca{R}$ only in $\partial \mca{R}$.
If the section has no escape locus and every $\mca{R}_i$ is compact,
then $\mca{U}$ is the set denoted by the same letter in McMullen \cite[(5.1)]{McMullen2013CM}.

\begin{lem} \label{lem.VU}
The set $\mca{U}-\mca{R}$ is the disjoint union of the open flow boxes
\[\intr E_{ij}=\bigcup_{x\in \mca{R}_{ij}}\varphi_{(0,r(x))}(x),\]
taken over the pairs $(i,j)$ with $\mca{R}_{ij}\neq \emptyset$.
\end{lem}

\begin{proof}
A point of $\mca{U}-\mca{R}$ lies in the interior of exactly one first return segment,
and by the definition of $\mca{U}$ the base point of that segment lies in some $\mca{R}_{ij}$, which gives the partition.
Each box is open in $M$, because $\mca{R}_{ij}$ is an open subset of the surface $\mca{R}$ transverse to the flow
and $r$ is continuous on $\mca{R}_{ij}$ by transversality at $f(x)\in \intr \mca{R}_j$.
\end{proof}

\begin{lem} \label{lem.homotopic}
Let $\tau$ be a closed orbit of $\Sigma(\mca{G},h)$ whose itinerary is realized by a closed orbit $\gamma$ of $\varphi$ with $\gamma\cap \mca{R}\subset \bigcup_{i,j}\mca{R}_{ij}$.
Then $\gamma$ is contained in $\mca{U}$, and the loops $\gamma$ and $\iota(C(\tau))$ are freely homotopic in $\mca{U}$.
\end{lem}

\begin{proof}
Let $(x_0,x_1,\dots,x_{n-1})$ denote the itinerary of $\tau$, read cyclically,
and let $y_m$ denote the $m$-th crossing of $\gamma$ with $\mca{R}$, so that $y_m\in \mca{R}_{x_mx_{m+1}}$.
Both loops are concatenations of $n$ arcs, the $m$-th arc of the first loop being the orbit segment of $\gamma$ from $y_m$ to $y_{m+1}$,
which lies inside $E_{x_mx_{m+1}}$; in particular $\gamma$ is contained in $\mca{U}$.
The $m$-th arc of the second loop is $\iota(e_{x_mx_{m+1}})$,
which also lies in $\intr \mca{R}_{x_m}\cup E_{x_mx_{m+1}}\cup \intr \mca{R}_{x_{m+1}}$.
For each $m$, the two arcs have endpoints in the contractible sets $\intr \mca{R}_{x_m}$ and $\intr \mca{R}_{x_{m+1}}$,
and both lie in $X=\intr \mca{R}_{x_m}\cup E_{x_mx_{m+1}}\cup \intr \mca{R}_{x_{m+1}}$, which is a subset of $\mca{U}$.
The arc of $\gamma$ is a first return segment in $E_{x_mx_{m+1}}$,
and the middle part of $\iota(e_{x_mx_{m+1}})$ is the first return segment $s_{x_mx_{m+1}}$.
After composing with paths inside the two contractible sets,
the two arcs differ by a loop that is a concatenation of the two first return segments,
one of them reversed, and of paths inside those sets.
Homotoping one segment to the other through first return segments and contracting the paths shows that the loop is null-homotopic in $X$.
Hence the two arcs are homotopic in $X$ rel the two rectangle interiors.
Concatenating these homotopies over $m$ gives a free homotopy between the two loops inside $\mca{U}$.
\end{proof}

Since $L\cap \mca{U}=\emptyset$, a free homotopy in $\mca{U}$ is all that is needed in order to compute a conjugacy class in $\pi_1(M-L)$.

In particular, for a homomorphism $\rho:\pi_1(M-L)\to G$ we have $\rho([\gamma])=\rho(\iota_*(C(\tau)))$,
so that the Frobenius class of $\tau$ computed with the flat cocycle $\widetilde{\rho\circ\iota_*}$ coincides with the class $\rho([\gamma])$ that occurs in \Cref{dfn.Chebotarev}.
Here the symbolic and the topological monodromies are matched.

\subsection{Surjectivity} \label{subsec.surj}

The complement of $\mca{U}$ is not a locally finite 2-complex:
the flow saturation of an edge of a rectangle can be dense in $M$,
and for the geodesic flows considered here it is dense by condition (E4$'$) below,
as it is in the compact case of McMullen \cite{McMullen2013CM}.
What is finite is the family of the \emph{faces} defined below,
and that is what the deformation argument uses.
An arc of $\mca{R}$ is said to be \emph{singular} if it is an edge of some $\mca{R}_i$,
a stable arc $A_i\times\{b\}$ with $b\in B_i$ an endpoint of some $B_{ij}$,
or one of the finitely many arcs of $\ell^{+}$ and of $\ell^{-}$.
The singular arcs form a finite family.
For an edge, for a stable arc as above, and for an arc of $\ell^{+}$, the corresponding \emph{face} is
\[\mca{W}(J)=\bigcup_{x\in J}\varphi_{[0,r(x)]}(x),\]
with the convention that $\varphi_{[0,r(x)]}(x)=\varphi_{[0,\infty)}(x)$ for $x\in \ell^{+}$;
for an arc $J$ of $\ell^{-}$ the corresponding face is $\varphi_{(-\infty,0]}(J)$.
Let $\mca{W}$ denote the union of the finitely many faces,
and let $N$ denote the union of the orbits that never meet $\mca{R}$.

We consider the following conditions.
\begin{itemize}
\item[\textup{(E0)}] Every compact orbit segment meets $\mca{R}$ in finitely
many points.
\item[\textup{(E1)}] The closure of every face is a properly embedded
surface in $M$.
\item[\textup{(E2)}] The set $N$ is a union of finitely many orbits.
\item[\textup{(E3)}] $M-\mca{U}\subset \mca{W}\cup N$.
\item[\textup{(E4$'$)}] For every unstable arc $J$ of every $\mca{R}_i$, almost
every point of $J$ has a dense forward orbit in $M$,
and for every stable arc $J$ of every $\mca{R}_i$,
almost every point of $J$ has a dense backward orbit in $M$.
\end{itemize}
Condition (E0) already occurred as the second assumption of \Cref{lem.flow.irreducible}.
Condition (E3) is verified for the modular sections in \Cref{sec.modular},
and it passes to the refinements, as shown in the proof of \Cref{thm.main.general}.
Condition (E4$'$) is stated on arcs and not on $M$, because that is what the deformation below requires;
a condition of full measure in $M$ says nothing about a subset of dimension one.
Condition (E4$'$) holds for the geodesic flow of a complete hyperbolic surface or orbifold of finite area, and compactness is not needed for that.
Indeed, let $\wt v$ be a lift of $v$ to the unit tangent bundle of the hyperbolic plane, and let $\wt v^{+}$ denote the forward endpoint of the geodesic through $\wt v$.
The closure of the forward orbit of $v$ depends only on $\wt v^{+}$,
because two vectors with the same forward endpoint are forward asymptotic,
and by ergodicity almost every value of $\wt v^{+}$ gives a dense forward orbit.
Along a lift of an unstable arc the coordinate $\wt v^{+}$ varies while the backward endpoint is constant,
so that the condition on $\wt v^{+}$ becomes a condition of full measure on the arc.
The statement for stable arcs and backward orbits is symmetric.

If the section has no escape locus and $M$ is closed, then all the faces are compact and $N=\emptyset$,
and the next lemma is Lemma 5.1 of McMullen \cite{McMullen2013CM}.
We use the escape locus only as the family of the arcs on which the return time is infinite.

\begin{figure}[t]
\centering
\begin{tikzpicture}[scale=1.4,>=Stealth]
  \fill[blue!7] (-1.4,-1.5) rectangle (1,1.5);
  \node[blue!60!black] at (-0.2,1.25) {\small $\intr E_{mi}$};
  \foreach \x in {1,3,5.5} { \draw[dashed,gray!70!black] (\x,-1.5) -- (\x,1.5); }
  \foreach \x in {1,3,5.5} { \node[gray!70!black,above] at (\x,1.5) {$\mca{R}$}; }
  \draw[thin,green!45!black] (0.3,0) -- (1,0);
  \draw[very thick,green!45!black] (1,0) -- (6.3,0);
  \draw[very thick,green!45!black,->] (5.5,0) -- (6.3,0);
  \node[green!45!black] at (2,-0.75) {\small $\mca{W}(J')$};
  \node[green!45!black] at (5.0,0.32) {\small $\mca{W}(J)$};
  \fill (1,0) circle (0.045);  \node at (1.15,-0.24) {\small $q'$};
  \fill (3,0) circle (0.045);  \node at (3.13,-0.24) {\small $q$};
  \fill (4.2,0) circle (0.045); \node[anchor=west] at (4.27,-0.24) {\footnotesize $p=\varphi_{s_0}(q)$};
  \fill (0.3,0) circle (0.045);
  \draw[thin,gray!60!black] (0.28,-0.05) -- (-0.15,-0.62);
  \node[anchor=north east] at (-0.1,-0.6) {\small $\varphi_{-T}(p)$};
  \node[gray!60!black] at (1.2,1.2) {\small $J'$};
  \node[gray!60!black] at (3.2,1.2) {\small $J$};
  \draw[red,thick,decoration={markings,mark=at position 0.8 with {\arrow{>}}},postaction={decorate}]
     (3.9,-1.4) .. controls (4.05,-0.6) and (4.1,-0.3) .. (4.2,0)
     .. controls (4.3,0.3) and (4.45,0.9) .. (4.6,1.4);
  \node[red] at (4.72,0.9) {$\alpha$};
  \draw[blue!70!black,thick,decoration={markings,mark=at position 0.55 with {\arrow{>}}},postaction={decorate}]
     (3.9,-1.4) .. controls (4.0,-0.9) .. (4.05,-0.4)
     -- (0.55,-0.4) arc (270:90:0.4) -- (4.05,0.4)
     .. controls (4.3,0.9) .. (4.6,1.4);
  \node[blue!70!black] at (2.5,0.62) {$\alpha'$};
  \draw[->,thick] (5.3,-1.75) -- (6.3,-1.75) node[right] {$t$};
\end{tikzpicture}
\caption{The deformation in the proof of \Cref{lem.surj}, drawn in a slice
transverse to the singular arcs, for a stable arc $J$. The flow runs to the
right. The loop $\alpha$ crosses the face $\mca{W}(J)$ at $p$. The faces
$\mca{W}(J)$ and $\mca{W}(J')$ containing the backward orbit segment of $p$
form a surface, which ends at $J'$, and the orbit segment continues (thin
line) into the open flow box $\intr E_{mi}$. The loop $\alpha'$ runs along
the two sides of that surface, crosses $\mca{R}$ twice, which is not a face,
and passes from one side to the other near $\varphi_{-T}(p)$, where no face
is present.}
\label{fig.deformation}
\end{figure}

\begin{lem} \label{lem.surj}
Assume \textup{(E1)}, \textup{(E2)}, \textup{(E3)}, and \textup{(E4$'$)},
and let $L\subset M$ be a union of finitely many closed orbits with $L\cap \mca{R}\subset \partial \mca{R}$.
Then $\iota$ induces a surjection $\iota_*:\pi_1(\mca{G})\surj \pi_1(M-L)$.
\end{lem}

\begin{proof}
We first check that $\iota_{*}:\pi_1(\mca{G})\to \pi_1(\mca{U})$ is surjective.
Let $\beta$ be a loop in $\mca{U}$ based at $\iota(v_{i_0})$.
The set $\mca{R}$ is a surface transverse to the flow and is closed in $\mca{U}$,
and $\beta$ is compact, so after a small homotopy in $\mca{U}$ the loop $\beta$ meets $\mca{R}$ transversely in finitely many points.
By \Cref{lem.VU}, the complement $\mca{U}-\mca{R}$ is the disjoint union of the open flow boxes,
and each arc of $\beta$ between two consecutive intersection points lies in a single box.
Let $y$ be an intersection point, lying in the rectangle $\mca{R}_k$.
Since $\mca{R}$ is an embedded surface, a small neighborhood of $y$ in $M$ meets $\mca{R}$ in a neighborhood of $y$ inside $\mca{R}_k$.
The arc of $\beta$ adjacent to $y$ on the forward side of $\mca{R}$ lies in a single box $\intr E_{kj}$,
and its points lie on first return segments whose base points lie in that neighborhood and converge to $y$;
hence $y$ lies in the closure of $\mca{R}_{kj}$.
Likewise the arc adjacent to $y$ on the backward side lies in a single box $\intr E_{mk}$,
and $y$ lies in the closure of $f(\mca{R}_{mk})$.
Therefore $y$ lies in the closure of the open set $f(\mca{R}_{mk})\cap \mca{R}_{kj}=A_{mk}\times B_{kj}$,
and a further small homotopy, supported near the intersection points and inside the adjacent boxes,
puts every intersection point in some $\mca{R}_{ij}$ and in some $f(\mca{R}_{mk})$.
The open box $\intr E_{ij}$ is homeomorphic to the product $\mca{R}_{ij}\times (0,1)$,
so each arc of $\beta$ is homotopic rel its endpoints, up to paths inside the two rectangle interiors,
to $\iota(e_{ij})$, to the reverse of $\iota(e_{ij})$,
or to a constant, according to the combination of the ends met by its endpoints.
By concatenation, $\beta$ is homotopic in $\mca{U}$ to $\iota(c)$ for a closed edge path $c$ in $\mca{G}$,
in which some edges may be traversed backwards,
and the class of $c$ lies in $\pi_1(\mca{G})$, which proves the surjectivity.

It thus suffices to prove that every loop $\alpha:S^1\to M-L$ based at a point of $\mca{U}$ can be deformed inside $M-L$ into $\mca{U}$, since $\mca{U}\subset M-L$.

By (E1) there are finitely many faces, the closure of each of them a properly embedded surface, and by (E2) the set $N$ is of dimension one.
The image of $\alpha$ is compact, so we may put $\alpha$ in general position:
transverse to the closure of every face, disjoint from $N$,
and meeting the union of the closures of the faces in finitely many points.

Let $p=\alpha(s)$ be one of those points, lying on the closure of a face over a singular arc $J$,
and write $p=\varphi_{s_0}(q)$ with $q\in \ol{J}$,
where $s_0\geq 0$ for a face $\mca{W}(J)$ and $s_0\leq 0$ for a backward face.
Assume first that $J$ is stable.
By (E4$'$) we may perturb $\alpha$ near $s$ so that the backward orbit of $q$ is dense in $M$.
In particular the orbit of $p$ is not closed,
and the backward orbit of $q$ enters the nonempty open set $\bigcup_{i,j}\intr E_{ij}$,
so that $\varphi_{-T}(p)$ lies in the interior of a flow box for some $T>s_0$.
The segment $\varphi_{[-T,0]}(p)$ is disjoint from $L$,
because the orbit of $p$ is not closed and $L$ is a union of finitely many closed orbits.

We deform $\alpha$ near $s$ as in McMullen \cite[Figure 2]{McMullen2013CM}, and as shown in \Cref{fig.deformation}.
The loop first approaches $p$, then follows a path close to $\varphi_{[-T,0]}(p)$,
then passes from one side of $\mca{W}(J)$ to the other near $\varphi_{-T}(p)$,
where a whole neighborhood lies in the interior of a flow box,
and finally returns along a nearly identical path.
The deformation is supported in a small neighborhood of $\varphi_{[-T,0]}(p)$, so it does not cross $L$.
It creates no new intersection with a face, because every face is a union of orbit segments between consecutive crossings of $\mca{R}$,
so that the faces met by $\varphi_{[-T,0]}(p)$ form,
for a generic $q$ whose backward orbit avoids the finitely many corners of $\mca{R}$, a surface containing that segment,
and the two paths run along the two sides of that surface.
The number of intersection points therefore decreases by one.

If $J$ is unstable, then we argue in the same way in positive time.
After finitely many steps we obtain a loop with image in $\mca{U}$, by (E3), which proves the assertion.
\end{proof}

\begin{rem} \label{rem.dichotomy}
In McMullen \cite[Lemma 5.1]{McMullen2013CM} the two types of edges of a rectangle are treated in positive and in negative time respectively.
In the present setting the faces over the escape loci are treated in exactly the same way,
since they too are swept out by a stable or by an unstable arc.
Compactness of $M$ is used in \cite{McMullen2013CM} only through the density of a generic orbit, that is, through (E4$'$).
\end{rem}

\section{The Chebotarev law for a section with an escape locus}
\label{sec.escape}

\begin{thm} \label{thm.main.general}
Let $\varphi$ be a flow on a 3-manifold $M$ admitting a Markov section with escape locus that satisfies \textup{(E0)}, \textup{(E1)}, \textup{(E2)}, \textup{(E3)}, and \textup{(E4$'$)},
and let $h$ be a positive H\"older roof function on $\Sigma(\mca{G})$ that is non-arithmetic,
that is, such that the ratio of the lengths of some two closed orbits of $\Sigma(\mca{G},h)$ that are realized by closed orbits of $\varphi$ is irrational.
For a finite union $L$ of closed orbits meeting $\mca{R}$,
a Markov section with the same escape locus is called a \emph{refinement along $L$} if it is obtained by cutting the rectangles along finitely many stable and unstable arcs, called the \emph{cutting arcs},
among them the local stable and unstable arcs through the points of $L\cap \mca{R}$,
each cutting arc being placed in one of the two adjacent rectangles as an edge.
Assume moreover that all but finitely many periodic points of the shift are realized, each by a unique point of $\mca{R}$,
and that every such union $L$ admits a refinement along $L$ that has the same property and satisfies \textup{(E1)}.
Let $(K_i)_i$ denote the closed orbits of $\varphi$ that meet $\mca{R}$,
ordered by the length $\ell_h$ attached to $h$,
with an arbitrary order among the closed orbits of equal length.
Then the following two assertions hold.
\begin{enumerate}
\item The sequence $(K_i)_i$ obeys the Chebotarev law.
\item Let $n\in \Z_{\geq 1}$, put $L_n=\bigcup_{i\leq n}K_i$, let $G$ be a
compact Lie group whose identity component is semisimple,
let $\rho:\pi_1(M-L_n)\to G$ be a homomorphism with dense image,
and let $F\in C(G,\C)$ be a continuous class function.
Then
\[\lim_{N\to \infty}\frac{1}{N}\sum_{i=n+1}^{N}F\bigl(\rho([K_i])\bigr)
=\int_G F(g)\,dg.\]
\end{enumerate}
\end{thm}

Assertion \textup{(2)} is the counterpart, for a section with an escape locus, of the theorem of McMullen \cite[Theorem 1.4]{McMullen2013CM}.
In \cite{UekiMKC2} the assumption on the identity component of $G$ is removed, and the exceptions are described.

The roof function $h$ is not required to be the return time of $\varphi$.
The assumptions therefore involve no rate of contraction,
and $h$ is any positive H\"older function that is non-arithmetic;
for instance, a function of the first letter whose values give two closed orbits with an irrational ratio of lengths.
The ordering is then the one attached to $h$,
and not the one by the periods of $\varphi$; see \Cref{sec.ordering}.

\begin{proof}
We prove (2).
Assertion (1) is the case of a finite group $G$ and of the indicator function of a conjugacy class $C\subset G$, by the remark following \Cref{thm.PP}.
Let $n\in \Z_{\geq 1}$, let $G$ and $\rho$ be as in (2), and let $F\in C(G,\C)$.
Let the refined section be a refinement along $L_n$ given by the assumption, so that the points of $L_n\cap \mca{R}$ lie on edges of the refined rectangles.
Conditions (E0), (E2), and (E4$'$) hold for the refined section,
because the surface $\mca{R}$ and the set $N$ are unchanged and the arcs of the refined rectangles are subarcs of the arcs of the original rectangles,
and condition (E1) holds by the assumption on the refinements.

Condition (E3) also holds for the refined section.
We first prove the following claim.
Let $x$ be a point of the interior of a refined rectangle $\mca{R}'_k=A'_k\times B'_k$ that lies in no transition set of the refined section and in the closure of some $\mca{R}_{ij}$.
Then the stable arc of $\mca{R}'_k$ through $x$ is singular for the refined section.
Indeed, every neighborhood of $x$ meets $\mca{R}_{ij}\cap \intr \mca{R}'_k$ in a nonempty open set.
By the last assertion of \Cref{lem.arcs}, the preimage in $\mca{R}_{ij}$ of each cutting arc is contained in a single stable or unstable arc,
so this open set contains a point whose image lies in the interior of a refined rectangle.
Since there are finitely many refined rectangles, $x$ lies in the closure of the transition set $\mca{R}'_{kl}$ for some refined rectangle $\mca{R}'_l$.
By condition (iii) of \Cref{dfn.markov} for the refined section, we have $\mca{R}'_{kl}=\intr(A'_k)\times B'_{kl}$ for an open interval $B'_{kl}\subset B'_k$.
The point $x$ does not lie in $\mca{R}'_{kl}$ and its first coordinate lies in $\intr(A'_k)$,
so its second coordinate $b$ is an endpoint of $B'_{kl}$, and the stable arc $A'_k\times\{b\}$ is singular for the refined section.
This proves the claim.

We have $\mca{W}\subset \mca{W}'$.
Indeed, the arcs of $\ell^{+}$ and of $\ell^{-}$ are singular for both sections, and every edge of a given rectangle is a union of edges of refined rectangles.
Let $x$ be a point of a stable arc $A_i\times\{b\}$ such that $b\in B_i$ is an endpoint of some $B_{ij}$, and assume that $x$ lies on no edge of a refined rectangle.
Then $x$ lies in the interior of a refined rectangle and in the closure of $\mca{R}_{ij}$.
The intervals $B_{ij'}$ with $i$ fixed are open and mutually disjoint, so that $b$ lies in none of them.
Hence $x$ lies in no transition set of the given section, and therefore in no transition set of the refined section.
By the claim, $x$ lies on a stable arc that is singular for the refined section.
Since the return time is unchanged, every face of the given section is contained in $\mca{W}'$.

We have $M-\mca{U}'=(M-\mca{U})\cup(\mca{U}-\mca{U}')$, and the first set lies in $\mca{W}'\cup N$.
A point of $\mca{U}-\mca{U}'$ either lies in some $\intr \mca{R}_i$ on a cutting arc,
or lies in a box $E_{ij}$ over a base point $x\in \mca{R}_{ij}$ such that $x$ or $f(x)$ lies on a cutting arc.
If such a base point $x$ lies on no cutting arc, then $x$ lies in the interior of a refined rectangle and in no transition set of the refined section,
so that the claim applies to $x$; the dashed arc in \Cref{fig.refinement} consists of such points.
In each case the point lies on the face of the refined section over a singular arc, since the cutting arcs are edges of the refined rectangles.

Write $\mca{G}'$ for the graph of the refined section.
Sending a refined rectangle to the rectangle containing it defines a graph map $\mca{G}'\to \mca{G}$
and a shift equivariant projection $\Sigma(\mca{G}')\to \Sigma(\mca{G})$,
and we equip $\Sigma(\mca{G}')$ with the roof $h'$ pulled back along that projection.
The graph $\mca{G}'$ is irreducible by \Cref{lem.flow.irreducible}:
its second assumption is (E0), and (E4$'$) provides a point with a dense forward orbit,
whose closure and hence whose $\omega$-limit set is $M$,
so that every forward tail is dense and $\varphi$ is topologically transitive.
By \Cref{lem.surj}, applied to the refined section, the map $\iota_*$ is surjective,
so that $\rho\circ \iota_*:\pi_1(\mca{G}')\to G$ has dense image.

By \Cref{lem.orbits.bijection}, applied to the refined section,
and by the realization assumption for the refined section,
the closed orbits of $\varphi$ meeting $\mca{R}$ and the closed orbits of $\Sigma(\mca{G}',h')$ correspond bijectively up to finitely many exceptions on each side.
The length $\ell_{h'}$ of a closed orbit of $\Sigma(\mca{G}',h')$ arising from a closed orbit $\gamma$ of $\varphi$ equals the length $\ell_h$ of the corresponding closed orbit of $\Sigma(\mca{G},h)$.
Indeed, let $y\in \gamma\cap \mca{R}$ denote the realizing point and let $p$ be a period of either itinerary;
then $f^{p}(y)$ realizes the same periodic point as $y$, hence $f^{p}(y)=y$ by the uniqueness,
and $p$ is a multiple of the number of points of $\gamma\cap \mca{R}$.
The least periods of both itineraries therefore equal that number, and the roofs agree.

We show that, if $a\in\R$ and $e^{ia\ell_{h'}(\tau')}=1$ holds for all but finitely many closed orbits $\tau'$ of $\Sigma(\mca{G}',h')$, then $a=0$.
By \Cref{lem.orbits.bijection} and the realization assumption, applied to the given section and to the refined section,
all but finitely many closed orbits $\tau$ of $\Sigma(\mca{G},h)$ correspond, through a closed orbit of $\varphi$, to a closed orbit $\tau'$ of $\Sigma(\mca{G}',h')$ with $\ell_{h'}(\tau')=\ell_h(\tau)$.
Hence $e^{ia\ell_h(\tau)}=1$ holds for all but finitely many closed orbits $\tau$ of $\Sigma(\mca{G},h)$,
and therefore for every closed orbit, by \Cref{lem.finitely.many}.
This lemma applies: the graph $\mca{G}$ is irreducible by \Cref{lem.flow.irreducible}, for the same reason as $\mca{G}'$, and $\Sigma(\mca{G},h)$ has more than one closed orbit, since $h$ is non-arithmetic.
If $a\neq 0$, then all the ratios of the lengths of the closed orbits of $\Sigma(\mca{G},h)$ are rational, which contradicts the assumption that $h$ is non-arithmetic.
In particular, $\Sigma(\mca{G}',h')$ has no circle factor, as in the proof of \Cref{cor.nocircle}.

By \Cref{thm.flat}, applied to $\Sigma(\mca{G}',h')$ and to the flat cocycle attached to $\rho\circ\iota_*$,
the conclusion of \Cref{thm.PP} holds for the closed orbits of $\Sigma(\mca{G}',h')$.
By \Cref{lem.ties}, the conclusion holds for every enumeration by increasing $\ell_{h'}$ as well,
with an arbitrary order among the closed orbits of equal length.
The bijection with the closed orbits of $\varphi$ preserves the enumeration up to finitely many exceptions,
which do not affect a density, because $\ell_{h'}$ agrees with $\ell_h$ under the bijection.

If a closed orbit of $\varphi$ has a crossing outside $\bigcup_{i,j}\mca{R}'_{ij}$,
then that crossing or its image under $f$ lies on an edge of a refined rectangle,
because a point lying in the interior of its refined rectangle,
and whose image also lies in such an interior,
belongs to some $\mca{R}'_{ij}$ by the definition of the sets $\mca{R}'_{ij}$.
By the first assertion of \Cref{lem.orbits.bijection},
each of the finitely many edges meets at most one closed orbit,
so after removing finitely many further closed orbits the assumption of \Cref{lem.homotopic} holds,
and the Frobenius class of a closed orbit $\tau$ of $\Sigma(\mca{G}',h')$ agrees with $\rho([\gamma])$ for the corresponding closed orbit $\gamma$ of $\varphi$.
Assertion (2) follows.
\end{proof}

\begin{lem} \label{lem.refinement.markov}
Let $(\mca{R}_i)_{i=1}^{n}$ be a Markov section with escape locus $(\ell^{+},\ell^{-})$ that satisfies \textup{(E0)},
and let $L$ be a finite union of closed orbits meeting $\mca{R}$ with $L\cap \partial \mca{R}=\emptyset$.
Let $(\mca{R}'_k)_k$ denote the family obtained by cutting each rectangle $\mca{R}_i\cong A_i\times B_i$ along the arcs $A_i\times\{b\}$ and $\{a\}\times B_i$ through the points $(a,b)$ of $L\cap \mca{R}_i$,
each cut value being assigned to one of the two adjacent intervals.
Then $(\mca{R}'_k)_k$ is a Markov section with the same escape locus.
If every edge of every rectangle is contained in $\ell^{+}\cup \ell^{-}$,
then the condition $L\cap \partial \mca{R}=\emptyset$ holds for every finite union $L$ of closed orbits.
\end{lem}

\begin{figure}[t]
\centering
\begin{tikzpicture}[scale=1.15,>=Stealth]
\begin{scope}
  \def\W{3.2}\def\H{3.2}
  \fill[blue!12] (0,1.0) rectangle (\W,2.2);
  \draw[thick] (0,0) rectangle (\W,\H);
  \draw[red!70!black,dashed,thick] (0,1.3) -- (\W,1.3);
  \draw[red!70!black,thick] (0,1.8) -- (\W,1.8);
  \draw[red!70!black,thick] (1.2,0) -- (1.2,\H);
  \fill (1.2,1.8) circle (1.6pt);
  \node[above right] at (1.2,1.8) {\small $p$};
  \node[blue!70!black] at (2.5,2.0) {\small $\mca{R}_{ij}$};
  \node[above] at (2.8,\H) {\small $\mca{R}_i$};
  \draw[|-|] (0,-0.25) -- (\W,-0.25) node[midway,below] {\small $A_i$};
  \draw[|-|] (-0.25,0) -- (-0.25,\H);
  \node[left] at (-0.25,0.5) {\small $B_i$};
  \draw[|-|] (-0.6,1.0) -- (-0.6,2.2) node[midway,left] {\small $B_{ij}$};
\end{scope}
\draw[->,thick] (4.0,1.6) -- (4.9,1.6) node[midway,above] {$f$};
\begin{scope}[xshift=6.0cm]
  \def\W{3.2}\def\H{3.2}
  \fill[blue!12] (0.9,0) rectangle (1.9,\H);
  \draw[thick] (0,0) rectangle (\W,\H);
  \draw[red!70!black,thick] (0,2.133) -- (\W,2.133);
  \draw[red!70!black,thick] (1.275,0) -- (1.275,\H);
  \fill (1.275,2.133) circle (1.6pt);
  \node[above right] at (1.275,2.133) {\small $f(p)$};
  \draw[red!70!black,thick] (0,0.8) -- (\W,0.8);
  \draw[red!70!black,thick] (2.6,0) -- (2.6,\H);
  \fill (2.6,0.8) circle (1.6pt);
  \node[above right] at (2.6,0.8) {\small $q$};
  \node[above,blue!70!black] at (1.4,\H) {\small $f(\mca{R}_{ij})$};
  \node[above] at (2.8,\H) {\small $\mca{R}_j$};
  \draw[|-|] (0,-0.25) -- (\W,-0.25);
  \node[below] at (2.7,-0.25) {\small $A_j$};
  \draw[|-|] (0.9,-0.6) -- (1.9,-0.6) node[midway,below] {\small $A_{ij}$};
  \draw[|-|] (\W+0.25,0) -- (\W+0.25,\H) node[midway,right] {\small $B_j$};
\end{scope}
\end{tikzpicture}
\caption{A rectangle $\mca{R}_i\cong A_i\times B_i$ with the transition set $\mca{R}_{ij}=\intr(A_i)\times B_{ij}$ (left),
and the rectangle $\mca{R}_j$ with the image $f(\mca{R}_{ij})=A_{ij}\times \intr(B_j)$ (right).
The solid lines are the cutting arcs through the points $p$, $f(p)$, and $q$ of $L$.
The cutting arcs through $p$ are mapped into the cutting arcs through $f(p)$,
and no other unstable cutting arc of $\mca{R}_j$ meets $f(\mca{R}_{ij})$;
these are the two inclusions in the proof of \Cref{lem.refinement.markov}.
The dashed arc is the preimage of the stable cutting arc through $q$.
It is not a cutting arc, and it consists of singular arcs of the refined section, as in the proof of \Cref{thm.main.general}.}
\label{fig.refinement}
\end{figure}

\begin{proof}
A closed orbit through a point $x\in \mca{R}$ returns to $x$ in positive and in negative time,
so that $x\notin \ell^{+}\cup \ell^{-}$ by conditions (i) and (ii) of \Cref{dfn.markov}.
This proves the last assertion.
Conditions (i) and (ii) hold for the family $(\mca{R}'_k)_k$, because $\mca{R}$ is unchanged,
and condition (iv) holds, because every transition set of this family is contained in some $\mca{R}_{ij}$.
We verify condition (iii); see \Cref{fig.refinement}.

The set $L\cap \mca{R}$ is finite by \textup{(E0)}, and $f$ permutes its points.
For each $i$, let $S_i^{-}\subset A_i$ and $S_i^{+}\subset B_i$ denote the sets of the first and of the second coordinates of the points of $L\cap \mca{R}_i$.
The rectangles $\mca{R}'_k$ contained in $\mca{R}_i$ are the products $A'\times B'$,
where $A'$ is one of the intervals into which $S_i^{-}$ cuts $A_i$ and $B'$ is one of the intervals into which $S_i^{+}$ cuts $B_i$.
Fix $i$ and $j$ with $\mca{R}_{ij}\neq \emptyset$.
By \Cref{lem.arcs} and by the continuity of $f$ on $\mca{R}_{ij}$,
there are injective continuous maps $\alpha\colon \intr(A_i)\to A_j$ and $\beta\colon B_{ij}\to B_j$ with $f(a,b)=(\alpha(a),\beta(b))$ for every $(a,b)\in \mca{R}_{ij}$,
and they satisfy $\alpha(\intr(A_i))=A_{ij}$ and $\beta(B_{ij})=\intr(B_j)$ by condition (iii) for the given section.
Both maps are therefore monotone.

Let $p$ be a point of $L\cap \mca{R}_i$.
Since $L$ is a union of closed orbits with $L\cap \partial \mca{R}=\emptyset$,
the points $p$ and $f(p)$ lie in the interiors of rectangles, so that $p\in \mca{R}_{ik}$ holds for exactly one $k$.
The intervals $B_{ik}$ with $i$ fixed are mutually disjoint, because the sets $\mca{R}_{ik}$ are.
Hence, if the second coordinate $b$ of $p$ lies in $B_{ij}$,
then $k=j$ and $\beta(b)$ is the second coordinate of the point $f(p)$ of $L\cap \mca{R}_j$.
This proves the inclusion
\[\beta(S_i^{+}\cap B_{ij})\subset S_j^{+}.\]
Let $q$ be a point of $L\cap \mca{R}_j$.
In the same way, $q\in f(\mca{R}_{kj})$ holds for exactly one $k$,
and the intervals $A_{kj}$ with $j$ fixed are mutually disjoint, because the sets $f(\mca{R}_{kj})=A_{kj}\times \intr(B_j)$ are.
Hence, if the first coordinate of $q$ lies in $A_{ij}$,
then $k=i$ and this coordinate is the image under $\alpha$ of the first coordinate of the point $f^{-1}(q)$ of $L\cap \mca{R}_i$.
This proves the inclusion
\[S_j^{-}\cap A_{ij}\subset \alpha(S_i^{-}).\]

Let $A'\times B'\subset \mca{R}_i$ and $A''\times B''\subset \mca{R}_j$ be two of the rectangles $\mca{R}'_k$.
A point $(a,b)$ of $\intr(A')\times \intr(B')$ is mapped by $f$ into $\intr(A'')\times \intr(B'')$
if and only if $b\in B_{ij}$, $\alpha(a)\in \intr(A'')$, and $\beta(b)\in \intr(B'')$ hold.
By the second inclusion and the injectivity of $\alpha$, the interval $\alpha(\intr(A'))$ contains no point of $S_j^{-}$,
so it is contained in $\intr(A'')$ or disjoint from it.
By the first inclusion, each endpoint of the interval $\beta(\intr(B')\cap B_{ij})$ is a point of $S_j^{+}$ or an endpoint of $B_j$,
so this interval contains $\intr(B'')$ or is disjoint from it.
Hence the transition set from $A'\times B'$ to $A''\times B''$ is empty or equal to $\intr(A')\times \bigl(\intr(B')\cap \beta^{-1}(\intr(B''))\bigr)$,
and in the second case its image is $\alpha(\intr(A'))\times \intr(B'')$.
This is condition (iii) for the family $(\mca{R}'_k)_k$.
\end{proof}

\begin{lem} \label{lem.refined.nonarithmetic}
In the situation of \Cref{thm.main.general}, let $L$ be a finite union of closed orbits meeting $\mca{R}$,
and let $\Sigma(\mca{G}',h')$ denote the suspension attached to a refinement along $L$ given by the assumption of that theorem, with the roof $h'$ pulled back from $h$.
If $a\in\R$ and $e^{ia\ell_{h'}(\tau')}=1$ holds for all but finitely many closed orbits $\tau'$ of $\Sigma(\mca{G}',h')$, then $a=0$.
In particular, $\Sigma(\mca{G}',h')$ has no circle factor.
\end{lem}

\begin{proof}
This is shown in the proof of \Cref{thm.main.general}, in which the union $L_n$ may be replaced by $L$.
\end{proof}

\begin{lem} \label{lem.infinitely.many}
Let $\varphi$ and $\mca{R}$ satisfy the assumptions of \Cref{thm.main.general} on the section and on its refinements,
let $L$ be a finite union of closed orbits meeting $\mca{R}$, let $G$ be a finite group,
and let $\rho\colon \pi_1(M-L)\to G$ be a surjective homomorphism.
Then for every $g\in G$ there are infinitely many closed orbits $\gamma$ of $\varphi$,
not contained in $L$, such that $g\in \rho([\gamma])$.
\end{lem}

\begin{proof}
Let $\mca{G}'$ denote the graph of a refinement along $L$ given by the assumption, and let $v_0$ be a vertex.
By \Cref{lem.surj} and by \Cref{lem.dense} \textup{(1)},
the holonomies of the directed loops based at $v_0$ form the whole group $G$;
let $w$ be such a loop with holonomy $g$.
The shift $\Sigma(\mca{G}')$ has infinitely many periodic orbits, so that $\mca{G}'$ is not a cycle,
and there are two distinct first-return loops $u_1$ and $u_2$ based at $v_0$.
Put $N=\#G$ and $w_k=w\,u_1^{kN}u_2^{N}$ for $k\in \Z_{\geq 1}$; the holonomy of $w_k$ is $g$.
A directed loop based at $v_0$ is a product of first-return loops in a unique way.
If $k$ is large, then the cyclic word of $w_k$ in the first-return loops contains exactly one maximal run of $u_1$ of length at least $kN$,
so that $w_k$ is not a proper power, and the periodic orbits of the $w_k$ are pairwise distinct because their periods are.
All but finitely many of them are realized by closed orbits of $\varphi$ that are not contained in $L$ and that avoid the edges of the refined rectangles:
all but finitely many periodic orbits are realized, by the realization assumption;
only finitely many periodic orbits are attached to $L$; and only finitely many closed orbits meet the edges,
by \Cref{lem.orbits.bijection}, each of them realizing finitely many periodic orbits.
The Frobenius class of such a closed orbit is the class of $g$ by \Cref{lem.homotopic},
and distinct periodic orbits are realized by distinct closed orbits, by \Cref{lem.orbits.bijection}.
\end{proof}

\section{The modular flow} \label{sec.modular}

The cross-section of this section and its codings are classical, and \Cref{rem.modular.sources} gives the sources;
we include the short proofs in order to fix the notation.
What is new is that the section is a Markov section with an escape locus that satisfies the conditions of \Cref{subsec.surj} (\Cref{prop.modular} \textup{(2)} and \textup{(3)}),
the cocycle on the induced system (\Cref{prop.induced} \textup{(4)}), and the refinement of \Cref{subsec.refine}.

\subsection{The section from the Farey tessellation}
\label{subsec.farey}

Let $\bb{H}^2$ denote the upper half plane, let $\Gamma=\PSL_2\Z$,
and let $M=T^1(\Gamma\backslash \bb{H}^2)=\Gamma\backslash \PSL_2\R$,
which is homeomorphic to the complement of the trefoil $K_{2,3}$ in $S^3$.
Let $\varphi$ denote the geodesic flow.
For $\wt v\in T^1\bb{H}^2$ we write $\wt v^{-}$ and $\wt v^{+}$ for the backward and the forward endpoint of the geodesic through $\wt v$ in $\partial \bb{H}^2=\bb{P}^1\R$.

Let $\mca{F}$ denote the Farey tessellation, that is,
the ideal triangulation of $\bb{H}^2$ whose vertex set is $\bb{P}^1\Q$ and whose edges join $p/q$ and $r/s$ with $|ps-qr|=1$.
The group $\Gamma$ acts on $\mca{F}$ transitively on the oriented edges and transitively on the pairs consisting of a triangle and one of its vertices.
Let $e_0$ denote the edge joining $0$ and $\infty$,
that is, the imaginary axis, and let $\Delta_0$ denote the ideal triangle with vertices $0$, $1$, and $\infty$.
Let $\wt{\mca{R}}\subset T^1\bb{H}^2$ denote the set of unit tangent vectors whose base point lies on $e_0$ and which point into the side $\{\mathrm{Re}\,z>0\}$ of $e_0$,
and let $\mca{R}\subset M$ denote its image; see \Cref{fig.farey}.
By \Cref{prop.modular} \textup{(1)} below, the projection $\wt{\mca{R}}\to \mca{R}$ is bijective,
and for $v\in \mca{R}$ we write $v^{\pm}=\wt v^{\pm}$, where $\wt v\in \wt{\mca{R}}$ is the lift of $v$.

\begin{figure}[t]
\centering
\begin{tikzpicture}[scale=2.2,>=Stealth]
  \def\yt{2.0}
  \fill[green!7] (0,0) -- (0,\yt) -- (1,\yt) -- (1,0) arc (0:180:0.5) -- cycle;
  \draw[thick] (-2.3,0) -- (2.5,0);
  \draw[very thick] (0,0) -- (0,\yt) node[above] {$e_0$};
  \draw[thick] (1,0) -- (1,\yt);
  \draw[thick] (1,0) arc (0:180:0.5);
  \draw[gray] (0.5,0) arc (0:180:0.25);
  \draw[gray] (1,0) arc (0:180:0.25);
  \draw[gray] (2,0) arc (0:180:0.5);
  \draw[gray] (2,0) -- (2,\yt);
  \draw[gray] (0,0) arc (0:180:0.5);
  \draw[gray] (-1,0) -- (-1,\yt);
  \draw[gray] (-1,0) arc (0:180:0.5);
  \draw[gray] (-2,0) -- (-2,\yt);
  \foreach \x/\l in {-2/{-2},-1/{-1},0/{0},1/{1},2/{2}} { \node[below] at (\x,0) {$\l$}; }
  \node[below] at (0.5,0) {\small $\tfrac12$};
  \node at (0.5,1.75) {$\Delta_0$};
  \node at (1.5,1.45) {$P\Delta_0$};
  \draw[->,thick,gray!60!black] (0.68,1.6) -- (1.32,1.6);
  \node[gray!60!black] at (0.82,1.7) {\small $P$};
  \draw[blue,thick,decoration={markings,mark=at position 0.62 with {\arrow{>}}},postaction={decorate}]
     (-2,0) arc (180:0:1.3);
  \fill[blue] (0,{sqrt(1.3^2-0.7^2)}) circle (0.025);
  \node[blue,anchor=east] at (-1.72,1.15) {$v^{+}\in(0,1)$};
  \draw[red,thick,decoration={markings,mark=at position 0.5 with {\arrow{>}}},postaction={decorate}]
     (-1.5,0) arc (180:8:2);
  \fill[red] (0,{sqrt(4-0.25)}) circle (0.025);
  \node[red,anchor=west] at (1.35,1.9) {$v^{+}\in(1,\infty)$};
  \foreach \y in {0.3,0.6,0.95,1.35} {
    \pgfmathsetmacro{\c}{(1-\y*\y)/2}
    \pgfmathsetmacro{\r}{1-\c}
    \pgfmathsetmacro{\ang}{atan2(\y,-\c)}
    \draw[green!45!black,densely dashed,decoration={markings,mark=at position 0.5 with {\arrow{>}}},postaction={decorate}]
       (0,\y) arc (\ang:0:\r);
  }
  \node[green!45!black,anchor=east] at (-0.05,0.62) {$\ell^{+}$};
  \node[green!45!black,anchor=west] at (0.22,0.32) {\small $D^{+}$};
\end{tikzpicture}
\caption{The Farey tessellation near $e_0$. The section $\mca{R}$ consists of
the vectors crossing $e_0$ from left to right. A vector with $v^{+}\in(0,1)$
leaves $\Delta_0$ through the edge $(0,1)$, a vector with $v^{+}\in(1,\infty)$
leaves it through the edge $(1,\infty)$, and a vector with $v^{+}=1$ lies in
$\ell^{+}$ and converges to the vertex $1$. The vectors in $\ol{\Delta_0}$ pointing
to $1$ form $D^{+}$, and $P$ glues the two edges of $\ol{\Delta_0}$ adjacent to
$1$.}
\label{fig.farey}
\end{figure}

\begin{prop} \label{prop.modular}
The following four assertions hold.
\begin{enumerate}
\item The stabilizer of $e_0$ in $\Gamma$ is of order two, generated by
$S(z)=-1/z$, and $S$ exchanges the two sides of $e_0$.
Hence the projection $\wt{\mca{R}}\to \mca{R}$ is bijective,
$v\mapsto (v^{-},v^{+})$ is a homeomorphism from $\mca{R}$ onto $I^{-}\times I^{+}=(-\infty,0)\times (0,\infty)$,
and $\mca{R}$ is a rectangle whose stable arcs are the sets $I^{-}\times\{b\}$ with $b\in I^{+}$.
\item Put $\ell^{+}=\{v\in \mca{R}\mid v^{+}=1\}$ and
$\ell^{-}=\{v\in \mca{R}\mid v^{-}=-1\}$, and let $\mca{R}_1=I^{-}\times (0,1]$
and $\mca{R}_2=I^{-}\times(1,\infty)$ be the two rectangles obtained by cutting $\mca{R}$ along $\ell^{+}$,
the cutting arc $\ell^{+}$ being placed in $\mca{R}_1$ as an edge.
Then $(\mca{R}_1,\mca{R}_2)$ is a Markov section with escape locus $(\ell^{+},\ell^{-})$,
the graph $\mca{G}$ has two vertices and four edges,
and the first return map has the combinatorics of the fake horseshoe of Bonatti and Pinsky \cite[Section 3]{BonattiPinsky2021},
with parabolic instead of hyperbolic corners.
\item Let $P=\spmx{2&-1\\1&0}\in \Gamma$, which is parabolic with fixed
point $1$.
Then $D^{+}$ is the image of $\{(z,\overrightarrow{1})\mid z\in \ol{\Delta_0}\}$,
where $(z,\overrightarrow{1})$ is the unit tangent vector at $z$ pointing to the vertex $1$,
and the only identification is the one produced by $P$, which maps the edge $(0,1)$ to the edge $(1,\infty)$.
Thus $D^{+}$ is a properly embedded surface whose boundary is $\ell^{+}$, and the orbit $e_0$ lies in its interior.
The analogous description holds for $D^{-}$, and $D^{+}\cap D^{-}$ is the orbit $e_0$,
which is a single orbit, since $S$ reverses the orientation of $e_0$.
Moreover \textup{(E0)}, \textup{(E1)}, \textup{(E2)}, \textup{(E3)}, and \textup{(E4$'$)} hold,
the faces being the sheets $\mca{W}(\{v^{+}=c\})$ with $c\in\{1/2,2\}$,
the forward face $\varphi_{[0,\infty)}(\ell^{+})$,
and the backward face $\varphi_{(-\infty,0]}(\ell^{-})$, and $N=e_0$.
\item Recall the matrices $L=\spmx{1&0\\1&1}$ and $R=\spmx{1&1\\0&1}$ of
\Cref{subsec.bg.rademacher}.
The closed orbits of $\varphi$ correspond bijectively to the periodic points of $\Sigma(\mca{G})$ other than the two fixed points of the shift, taken up to the shift,
and a periodic point with itinerary $(x_0,\dots,x_{n-1})$ corresponds to the conjugacy class of the word obtained by substituting $L$ for the symbol $1$ and $R$ for the symbol $2$.
Under that correspondence, the sum of the return times over one period equals the translation length $\ell(\gamma)=2\log \xi_\gamma$,
where $\xi_\gamma>1$ is determined by $\xi_\gamma+\xi_\gamma^{-1}=|\tr\gamma|$,
and the closed orbit is primitive if and only if the cyclic word $(x_0,\dots,x_{n-1})$ is not a proper power.
The two fixed points of the shift correspond to the parabolic classes of $L$ and $R$, and to no closed orbit.
\end{enumerate}
\end{prop}

We prove (1) and (2) here, and (3) and (4) in \Cref{subsec.farey.orbits}.

\begin{proof}[Proof of \Cref{prop.modular}, parts \textup{(1)} and
\textup{(2)}]
(1) A vector of $\wt{\mca{R}}$ determines a geodesic with backward endpoint in $(-\infty,0)$ and forward endpoint in $(0,\infty)$,
and conversely such a geodesic meets $e_0$ exactly once, from the side $\{\mathrm{Re}\,z<0\}$ to the side $\{\mathrm{Re}\,z>0\}$.
Hence $\wt v\mapsto (\wt v^{-},\wt v^{+})$ is a homeomorphism from $\wt{\mca{R}}$ onto $I^{-}\times I^{+}$.
If $\gamma\in \Gamma$ maps a vector of $\wt{\mca{R}}$ to a vector of $\wt{\mca{R}}$,
then $\gamma e_0$ is an edge of $\mca{F}$ that meets $e_0$ at a point of $\bb{H}^{2}$, so that $\gamma e_0=e_0$.
An element of $\Gamma$ preserving $e_0$ preserves $\{0,\infty\}$,
and the only such elements of $\PSL_2\Z$ are the identity and $S$.
Since $S$ exchanges the two sides of $e_0$, we get $\gamma=1$.
Hence the projection $\wt{\mca{R}}\to \mca{R}$ is bijective, and $\mca{R}$ is parametrized by $(v^{-},v^{+})$ as stated.
The stable manifold of $v$ consists of the vectors with the same forward endpoint, which identifies the stable arcs.

(2) The triangle of $\mca{F}$ on the side into which the vectors of $\mca{R}$ point is $\Delta_0$,
whose other two edges join $0$ with $1$ and $1$ with $\infty$.
The geodesic through $v$ ends at $v^{+}$, which lies on the boundary arc cut off by one of those two edges,
so the geodesic leaves $\Delta_0$ through that edge.
Hence the next edge crossed is the edge $(0,1)$ if $v^{+}\in (0,1)$ and the edge $(1,\infty)$ if $v^{+}\in (1,\infty)$,
while if $v^{+}=1$, then the geodesic converges to the ideal vertex $1$ inside $\Delta_0$ and crosses no further edge.
The elements
\[T_1=\spmx{1&0\\-1&1},\qquad T_2=\spmx{1&-1\\0&1}\]
of $\Gamma$, that is, $T_1(z)=z/(1-z)$ and $T_2(z)=z-1$,
map the edge $(0,1)$ and the edge $(1,\infty)$ respectively to $e_0$.
They preserve the crossing direction, since they send $I^{-}\times(0,1)$
and $I^{-}\times(1,\infty)$ into $(-1,0)\times(0,\infty)$ and into $(-\infty,-1)\times(0,\infty)$ respectively,
and a pair with a negative first coordinate and a positive second one is a crossing in the chosen direction.
Since $\Gamma$ is transitive on oriented edges, $f$ is defined on $\mca{R}-\ell^{+}$ and has exactly two branches,
and the same argument in negative time gives condition (ii) of \Cref{dfn.markov} with $\ell^{-}$.

On each branch the return time is the hyperbolic length of the segment of the geodesic through $v$ between $e_0$ and the next edge, which depends continuously on $v$; this is condition (iv).
The branch is determined by $v^{+}$ alone, so $\mca{R}_{ij}=I^{-}\times B_{ij}$, which is the first half of condition (iii).
For the second half we compute the images of the two branches.
We have
\[T_1\bigl((0,1)\bigr)=T_2\bigl((1,\infty)\bigr)=(0,\infty)=I^{+},\]
\[T_1\bigl((-\infty,0)\bigr)=(-1,0),\qquad
T_2\bigl((-\infty,0)\bigr)=(-\infty,-1).\]
Hence $f(\mca{R}_1)=(-1,0)\times I^{+}$ and $f(\mca{R}_2)=(-\infty,-1)\times I^{+}$ are disjoint subrectangles that are full in the expanding direction
and proper in the contracting direction, which is condition (iii).
The gap between the two images is exactly $\ell^{-}$,
so that the backward escape locus is recovered from the same computation.
All four sets $\mca{R}_{ij}$ are nonempty; see \Cref{fig.horseshoe}.
\end{proof}

\begin{figure}[t]
\centering
\begin{tikzpicture}[scale=1.15,>=Stealth]
\begin{scope}
  \def\W{2.4}\def\H{2.4}
  \draw[thick] (0,0) rectangle (\W,\H);
  \draw[green!45!black,very thick] (0,1.2) -- (\W,1.2);
  \node[right,green!45!black] at (\W,1.2) {\small $\ell^{+}$};
  \draw[dashed] (0,0.6) -- (\W,0.6);
  \draw[dashed] (0,1.8) -- (\W,1.8);
  \draw[dotted,thick] (1.2,0) -- (1.2,\H);
  \node[above,] at (1.2,\H) {\small $\ell^{-}$};
  \node[left] at (0,0) {\small $0$};
  \node[left] at (0,0.6) {\small $\tfrac12$};
  \node[left] at (0,1.2) {\small $1$};
  \node[left] at (0,1.8) {\small $2$};
  \node[left] at (0,\H) {\small $\infty$};
  \node[left] at (-0.4,1.2) {\small $v^{+}$};
  \node[below] at (0,0) {\small $-\infty$};
  \node[below] at (1.2,0) {\small $-1$};
  \node[below] at (\W,0) {\small $0$};
  \node[below] at (1.2,-0.3) {\small $v^{-}$};
  \node at (0.6,0.3) {\small $\mca{R}_{11}$};
  \node at (0.6,0.9) {\small $\mca{R}_{12}$};
  \node at (0.6,1.5) {\small $\mca{R}_{21}$};
  \node at (0.6,2.1) {\small $\mca{R}_{22}$};
  \node[right] at (\W,0.4) {\small $\mca{R}_1$};
  \node[right] at (\W,2.0) {\small $\mca{R}_2$};
\end{scope}
\draw[->,thick] (3.6,1.2) -- (4.4,1.2) node[midway,above] {$f$};
\begin{scope}[xshift=5.6cm]
  \def\W{2.4}\def\H{2.4}
  \fill[blue!12] (1.2,0) rectangle (\W,\H);
  \fill[red!12] (0,0) rectangle (1.2,\H);
  \draw[thick] (0,0) rectangle (\W,\H);
  \draw[dotted,thick] (1.2,0) -- (1.2,\H);
  \node[above] at (1.2,\H) {\small $\ell^{-}$};
  \node[blue!70!black] at (1.8,1.2) {\small $f(\mca{R}_1)$};
  \node[red!70!black] at (0.6,1.2) {\small $f(\mca{R}_2)$};
  \node[below] at (0,0) {\small $-\infty$};
  \node[below] at (1.2,0) {\small $-1$};
  \node[below] at (\W,0) {\small $0$};
  \node[left] at (0,0) {\small $0$};
  \node[left] at (0,\H) {\small $\infty$};
\end{scope}
\end{tikzpicture}
\caption{The rectangle $\mca{R}=I^{-}\times I^{+}$ in the coordinates
$(v^{-},v^{+})$, cut along $\ell^{+}=\{v^{+}=1\}$ into $\mca{R}_1$ and
$\mca{R}_2$, and the transition sets $\mca{R}_{ij}$ (left). The images
$f(\mca{R}_1)=(-1,0)\times I^{+}$ and $f(\mca{R}_2)=(-\infty,-1)\times I^{+}$
are full in the expanding direction, and the gap between them is
$\ell^{-}=\{v^{-}=-1\}$ (right).}
\label{fig.horseshoe}
\end{figure}

\subsection{The escape locus and the closed orbits}
\label{subsec.farey.orbits}

\begin{proof}[Proof of \Cref{prop.modular}, parts \textup{(3)} and \textup{(4)}]
(3) A vector $w$ lies in $D^{+}$ if and only if the geodesic through $w$ crosses no edge of $\mca{F}$ in positive time,
which holds if and only if $w$ points to a vertex of the triangle containing its base point,
or the base point lies on an edge and $w$ points to the opposite vertex or along the edge.
By the transitivity on pairs of a triangle and a vertex,
the set of such vectors is the image of $\{(z,\overrightarrow{1})\mid z\in \ol{\Delta_0}\}$.
The stabilizer of the pair $(\Delta_0,1)$ in $\Gamma$ is trivial,
since the stabilizer of $\Delta_0$ is of order three and permutes the vertices cyclically,
so the only identification comes from the two triangles adjacent to an edge through the vertex $1$.
That identification is realized by $P$, which satisfies $P(0)=\infty$, $P(1)=1$,
and $P(\infty)=2$, and which therefore maps $\Delta_0$ to the triangle $(1,2,\infty)$ and the edge $(0,1)$ to the edge $(1,\infty)$.
Since $P$ fixes $1$, it maps $(z,\overrightarrow{1})$ to $(Pz,\overrightarrow{1})$,
so that the two edges of $\ol{\Delta_0}$ adjacent to the vertex $1$ are glued to each other and become interior in $D^{+}$.
The remaining edge consists of the vectors on $e_0$ pointing to $1$,
which form $\ell^{+}$, and this is the whole boundary of $D^{+}$.
The map is proper: if a sequence leaves every compact subset of the glued triangle,
then its base points converge to an ideal vertex,
hence leave every compact subset of the orbifold, and the vectors leave every compact subset of $M$.

A vector tangent to an edge crosses no edge in either direction, so it lies in $D^{+}\cap D^{-}$.
A geodesic crossing no edge joins two vertices of a single triangle and is therefore an edge,
which gives $D^{+}\cap D^{-}=N=e_0$, that is (E2).
The remaining faces are the sheets $\mca{W}(\{v^{+}=c\})$ with $c=1/2$ and $c=2$,
which lie over the singular arcs $f^{-1}(\ell^{+})$,
and the same analysis as for $D^{+}$ shows that their closures are properly embedded surfaces;
the closures of the faces over $\ell^{+}$ and over $\ell^{-}$ are the closures of $D^{+}$ and of $D^{-}$,
because $D^{+}=\varphi_{[0,\infty)}(\ell^{+})\cup e_0$ holds and $e_0$ lies in the closure of the face, and likewise for $D^{-}$.
This is (E1).

Condition (E0) holds because a crossing of $\mca{R}$ is a crossing of an edge of $\mca{F}$
and a compact geodesic segment crosses finitely many edges of the locally finite tessellation $\mca{F}$.
For (E3), let $w\in M-(\mca{W}\cup N)$, so that the orbit of $w$ meets $\mca{R}$.
If some crossing occurs at a time $t\leq 0$, then a largest such time $t^{-}$ exists by (E0);
the crossing $x=\varphi_{t^{-}}(w)$ satisfies $0\leq -t^{-}\leq r(x)$ by maximality,
and $x$ does not lie on the three singular arcs $\{v^{+}=c\}$ with $c\in\{1/2,1,2\}$, since otherwise $w$ would lie on a face.
As $B_1-(B_{11}\cup B_{12})=\{1/2,1\}$ and $B_2-(B_{21}\cup B_{22})=\{2\}$,
the point $x$ lies in some $\mca{R}_{ij}$, and hence $w\in E_{ij}\subset \mca{U}$.
If every crossing occurs at a positive time, then the first crossing,
which exists by (E0), lies in $\ell^{-}$ by \Cref{dfn.markov} (ii),
and $w$ lies on the backward face, contrary to the assumption.
This gives $M-\mca{U}\subset \mca{W}\cup N$, which is (E3).
Condition (E4$'$) holds by the ergodicity of the geodesic flow on a hyperbolic orbifold of finite area, for which compactness is not needed.

(4) We work with the lift $\wt{\mca{R}}\subset T^1\bb{H}^2$ of $\mca{R}$ of \Cref{subsec.farey}.
By the proof of (2), a vector $\wt{v}\in \wt{\mca{R}}$ with itinerary starting with $x_0$ satisfies $T_{x_0}(\varphi_{r(\wt v)}(\wt v))\in \wt{\mca{R}}$, so that the first return map lifts to
\[\wt{f}(\wt v)=T_{x_0}\bigl(\varphi_{r(\wt v)}(\wt v)\bigr).\]
Iterating, for a vector with itinerary $(x_0,\dots,x_{n-1})$ we obtain
\[\wt{f}^{\,n}(\wt v)
=T_{x_{n-1}}\cdots T_{x_0}\bigl(\varphi_{r_n(\wt v)}(\wt v)\bigr),
\qquad r_n=\sum_{i=0}^{n-1}r\circ \wt f^{\,i}.\]
Hence $\wt v$ is a periodic point of period $n$ if and only if
\[\varphi_{r_n(\wt v)}(\wt v)=\gamma\, \wt v,
\qquad \gamma:=\bigl(T_{x_{n-1}}\cdots T_{x_0}\bigr)^{-1}
=T_{x_0}^{-1}\cdots T_{x_{n-1}}^{-1},\]
which is the word in $L$ and $R$ attached to the itinerary $(x_0,\dots,x_{n-1})$, because $T_1^{-1}=L$ and $T_2^{-1}=R$.
The displayed condition says that the geodesic through $\wt v$ is invariant under $\gamma$
and that $\gamma$ translates it by $r_n(\wt v)$ in the positive direction.
Therefore $\gamma$ is hyperbolic, the geodesic is its axis, and $r_n(\wt v)=\ell(\gamma)$.

Conversely, let $\gamma\in \Gamma$ be hyperbolic.
Its fixed points are quadratic irrationals, hence are not vertices of $\mca{F}$,
so the axis of $\gamma$ crosses infinitely many edges of $\mca{F}$.
Replacing $\gamma$ by a conjugate, we may assume that the axis crosses $e_0$ in the chosen direction,
and the computation above then produces a periodic itinerary.
Two periodic points on the same $\sigma$-orbit give cyclic permutations of the word, hence conjugate elements, and conversely.

It remains to identify the words that occur.
If a single symbol occurs, then $\gamma=L^{\,n}$ or $\gamma=R^{\,n}$, which is parabolic,
and no closed orbit arises; these are the two fixed points of the shift.
Suppose that both symbols occur.
Then $\gamma$ is a product of the non-negative matrices $L$ and $R$ in which both of them occur.
Already $LR=\spmx{1&1\\1&2}$ has trace $3$, and multiplying a non-negative matrix of trace at least $3$ by $L$ or by $R$ does not decrease the trace, so that $\tr\gamma\geq 3$.
Hence $\gamma$ is hyperbolic.
Finally a closed orbit is primitive if and only if the corresponding cyclic word is not a proper power.
The resulting bijection between primitive hyperbolic conjugacy classes and cyclic words in $L$
and $R$ containing both letters is the one used by Ghys \cite[Section 3.5]{Ghys2007ICM}.
\end{proof}

\begin{lem} \label{lem.tails}
A sequence $x\in \Sigma(\mca{G})$ is the itinerary of an orbit of the modular flow if and only if $x$ has no constant tail,
that is, if and only if both symbols occur infinitely often in each of the two directions.
\end{lem}

\begin{proof}
By the proof of \Cref{prop.modular} \textup{(2)},
the symbol of a vector $v\in \mca{R}$ is $1$ if $v^{+}\in(0,1]$ and $2$ if $v^{+}\in(1,\infty)$,
the map $f$ acts on both coordinates by the M\"obius transformation $T_{x_0}$ attached to the symbol,
the forward branch is determined by $v^{+}$ alone, and the backward branch is determined by $v^{-}$ alone.

Suppose that $x$ is the itinerary of an orbit and that the forward tail of $x$ is constant.
If the tail consists of the symbol $1$, then some point of the orbit has expanding coordinate $\eta_0$ with $T_1^{k}(\eta_0)=\eta_0/(1-k\eta_0)\in(0,1]$ for every $k\geq 0$;
the condition for $k$ gives $\eta_0\leq 1/(k+1)$,
and letting $k\to\infty$ gives $\eta_0\leq 0$, a contradiction.
If the tail consists of the symbol $2$, then some point has expanding coordinate $\eta_0$ with $\eta_0-k\in(1,\infty)$ for every $k\geq 0$, which is impossible.
A constant backward tail is excluded likewise:
the backward itinerary $(x_{-1},x_{-2},\dots)$ constrains the contracting coordinate by $\xi\in T_{x_{-1}}\cdots T_{x_{-k}}(I^{-})$ for every $k\geq 1$,
and $T_1^{k}(I^{-})=(-1/k,0)$ and $T_2^{k}(I^{-})=(-\infty,-k)$ have empty intersections over $k$.

Conversely, suppose that $x$ has no constant tail.
In forward time, both symbols occur infinitely often,
so the set of $\eta\in I^{+}$ with $T_{x_{k-1}}\cdots T_{x_0}(\eta)\in J_{x_k}$ for $0\leq k\leq m$,
where $J_1=(0,1]$ and $J_2=(1,\infty)$, is a nested family of intervals in $m$ whose closures shrink to a single point $\eta\in I^{+}$,
by the convergence of infinite continued fractions, and $\eta$ is irrational.
In backward time, the nested images $T_{x_{-1}}\cdots T_{x_{-k}}(I^{-})$ are eventually bounded,
and their closures likewise shrink to a single irrational point $\xi\in I^{-}$.
Let $v\in \mca{R}$ denote the vector with $(v^{-},v^{+})=(\xi,\eta)$.
The forward symbols of $v$ are those of $x$ by the construction of $\eta$,
and the backward symbols are those of $x$ by the construction of $\xi$, and the cross conditions hold:
$T_1^{-1}=L$ maps $I^{+}$ into $(0,1)$ and $T_2^{-1}=R$ maps $I^{+}$ into $(1,\infty)$,
while $T_1$ and $T_2$ map $I^{-}$ into itself.
Since $\eta$ and $\xi$ are irrational, the orbit of $v$ never meets $\ell^{+}\cup\ell^{-}$, and the itinerary of $v$ is $x$.
\end{proof}

\subsection{The four rectangles} \label{subsec.four}

\begin{prop} \label{prop.modular.four}
Cut $\mca{R}$ along $\ell^{+}$ and along $\ell^{-}$, that is, consider the four rectangles
\[\mca{R}_{u,i}=A_u\times B_i,\qquad
A_{\mathrm{I}}=[-1,0),\ A_{\mathrm{II}}=(-\infty,-1),\qquad
B_1=(0,1],\ B_2=(1,\infty),\]
each cutting arc $\{v^{+}=c\}$ being placed in the adjacent rectangle with $v^{+}\leq c$
and each cutting arc $\{v^{-}=d\}$ in the adjacent rectangle with $v^{-}\geq d$.
Then the family $(\mca{R}_{u,i})$ is again a Markov section with escape locus $(\ell^{+},\ell^{-})$,
and its graph is the edge shift of the graph of \Cref{prop.modular} \textup{(2)}.
Moreover the rectangle $Y=f(\mca{R}_2)\cap \mca{R}_1$,
on which \Cref{subsec.induced} induces the first return map, is $\mca{R}_{\mathrm{II},1}$.
\end{prop}

\begin{proof}
By the proof of \Cref{prop.modular} (2) we have $T_1(I^{-})=\intr(A_{\mathrm{I}})$ and $T_2(I^{-})=A_{\mathrm{II}}$,
so each branch maps the whole stable interval into a single one of the two new stable intervals.
Hence the rectangle of the image depends only on the branch and not on $v^{-}$,
and conditions (iii) and (iv) of \Cref{dfn.markov} continue to hold,
with $\mca{R}_{(u,i),(u',j)}=\intr(A_u)\times \bigl(\intr(B_i)\cap T_i^{-1}\intr(B_j)\bigr)$ when $u'$ is the interval determined by $i$, and empty otherwise.
Since $T_1(-\infty)=-1$ and $T_2(0)=-1$, no point of $\intr \mca{R}$ is mapped onto $\ell^{-}$,
and $\ell^{+}$ has no forward image, so the sets $D^{\pm}$ are unchanged,
in accordance with the convention of \Cref{subsec.boxes} that $D^{\pm}$ are attached to $\mca{R}$ and not to a decomposition of it.
The first coordinate $u$ of the state determines the previous branch, so the graph is the edge shift.
Finally $Y=f(\mca{R}_2)\cap \mca{R}_1=(-\infty,-1)\times(0,1]=\mca{R}_{\mathrm{II},1}$.
\end{proof}

In particular $Y$ is a single rectangle of the section,
so that every return word to $Y$ is a loop of the graph and the cocycle $\alpha$ needs no choice of base paths,
in accordance with \Cref{prop.induced} \textup{(4)}.

\subsection{The induced system} \label{subsec.induced}

For \Cref{sec.arith} we shall need the ordering by the hyperbolic length,
for which the roof function is the return time $r$ itself.
That function is not bounded away from zero on $\mca{R}$, so we induce on a subrectangle.
The result is the classical continued fraction model \cite{Series1985JLMS},
and we give the computation because the cocycle has to be tracked through the induction.

Keep the notation of \Cref{prop.modular}, and recall from \Cref{prop.modular.four} the rectangle
\[Y=f(\mca{R}_2)\cap \mca{R}_1=(-\infty,-1)\times (0,1]\subset \mca{R}.\]
A point of $Y$ is a point of $\mca{R}$ whose itinerary satisfies $x_{-1}=2$ and $x_0=1$,
that is, a point at which the symbol switches from $2$ to $1$.
Hence the first return map $f_Y$ of $f$ to $Y$ has branches indexed by the pairs $(k,l)\in \Z_{\geq 1}\times \Z_{\geq 1}$,
the branch $(k,l)$ corresponding to the word $1^{k}2^{l}$; see \Cref{fig.induced}.

\begin{figure}[t]
\centering
\begin{tikzpicture}[scale=1.0,>=Stealth]
\def\W{4.2}\def\H{4.2}
\begin{scope}
  \foreach \k in {1,...,12} {
    \foreach \l in {1,...,12} {
      \pgfmathsetmacro{\ylo}{\H/(\k+1/\l)}
      \pgfmathsetmacro{\yhi}{\H/(\k+1/(\l+1))}
      \pgfmathsetmacro{\shade}{min(12+6*\l,48)}
      \fill[blue!\shade] (0,\ylo) rectangle (\W,\yhi);
      \ifnum\l<5 \draw[gray!60,thin] (0,\ylo) -- (\W,\ylo); \fi
    }
    \pgfmathsetmacro{\ytop}{\H/\k}
    \pgfmathsetmacro{\ybot}{\H/(\k+1)}
    \pgfmathsetmacro{\ytail}{\H/(\k+1/13)}
    \fill[blue!48] (0,\ytail) rectangle (\W,\ytop);
    \ifnum\k<6 \draw[thick] (0,\ytop) -- (\W,\ytop); \fi
    \ifnum\k<4 \node[left] at (0,\ytop) {\small $\ifnum\k=1 1\else\tfrac1\k\fi$}; \fi
    \pgfmathsetmacro{\ymid}{(\ytop+\ybot)/2}
    \ifnum\k<3 \node[right] at (\W,\ymid) {\small $k=\k$}; \fi
  }
  \fill[blue!20] (0,0) rectangle (\W,\H/13);
  \draw[thick] (0,0) rectangle (\W,\H);
  \node[left] at (0,0) {\small $0$};
  \node[below] at (0,0) {\small $-\infty$};
  \node[below] at (\W,0) {\small $-1$};
  \node[below] at (\W/2,-0.3) {\small $v^{-}$};
  \node[left] at (-0.45,\H/2) {\small $v^{+}$};
  \node at (\W/2,{\H*(1/2+2/3)/2}) {\small $(1,1)$};
  \node at (\W/2,{\H*(2/3+3/4)/2}) {\small $(1,2)$};
  \node at (\W/2,{\H*(1/3+2/5)/2}) {\small $(2,1)$};
  \node[above] at (\W/2,\H) {$Y=(-\infty,-1)\times(0,1]$};
\end{scope}
\draw[->,thick] (5.0,\H/2) -- (5.8,\H/2) node[midway,above] {$f_Y$};
\begin{scope}[xshift=6.6cm]
  \foreach \l in {1,...,12} {
    \foreach \k in {1,...,12} {
      \pgfmathsetmacro{\xlo}{\W/(\l+1/\k)}
      \pgfmathsetmacro{\xhi}{\W/(\l+1/(\k+1))}
      \pgfmathsetmacro{\shade}{min(12+6*\k,48)}
      \fill[red!\shade] (\xlo,0) rectangle (\xhi,\H);
      \ifnum\k<5 \draw[gray!60,thin] (\xlo,0) -- (\xlo,\H); \fi
    }
    \pgfmathsetmacro{\xr}{\W/\l}
    \pgfmathsetmacro{\xl}{\W/(\l+1)}
    \pgfmathsetmacro{\xtail}{\W/(\l+1/13)}
    \fill[red!48] (\xtail,0) rectangle (\xr,\H);
    \ifnum\l<6 \draw[thick] (\xr,0) -- (\xr,\H); \fi
    \ifnum\l<4 \node[below] at (\xr,0) {\small $-\l$}; \fi
    \pgfmathsetmacro{\xmid}{(\xr+\xl)/2}
    \ifnum\l<3 \node[above] at (\xmid,\H) {\small $l=\l$}; \fi
  }
  \fill[red!20] (0,0) rectangle (\W/13,\H);
  \draw[thick] (0,0) rectangle (\W,\H);
  \node[below] at (0,0) {\small $-\infty$};
  \node[left] at (0,0) {\small $0$};
  \node[left] at (0,\H) {\small $1$};
  \node[below] at (\W/2,-0.3) {\small $v^{-}$};
  \node[rotate=90] at ({\W*(1/2+2/3)/2},\H/2) {\small $(1,1)$};
  \node[rotate=90] at ({\W*(2/3+3/4)/2},\H/2) {\small $(2,1)$};
  \node[rotate=90] at ({\W*(1/3+2/5)/2},\H/2) {\small $(1,2)$};
\end{scope}
\end{tikzpicture}
\caption{The induced system on $Y$.
The branch $(k,l)$ of $f_Y$ is the horizontal strip of $Y$ on which $\lfloor 1/v^{+}\rfloor=k$ and $\lfloor 1/\mathrm{G}(v^{+})\rfloor=l$ (left),
and $f_Y$ maps it onto the vertical strip $(-l-1/k,\,-l-1/(k+1))\times(0,1]$ (right).
The strips are full in the expanding direction $v^{+}$ and proper in the contracting direction $v^{-}$,
and the two pictures are exchanged by $(k,l)\mapsto(l,k)$.
The axes are drawn in the coordinates $v^{+}$ and $-1/v^{-}$,
so that the strips with large $k$ or $l$ accumulate at $v^{+}=0$ and at $v^{-}=-\infty$, which represent the cusp.}
\label{fig.induced}
\end{figure}

\begin{prop} \label{prop.induced}
Write $y=v^{+}$ for the expanding coordinate on $Y$,
and let $\mathrm{G}(y)=1/y-\lfloor 1/y\rfloor$ be the Gauss map.
Then the following four assertions hold.
\begin{enumerate}
\item The branch of $f_Y$ at $v$ is $(k,l)=(\lfloor 1/y\rfloor,
\lfloor 1/\mathrm{G}(y)\rfloor)$, that is, $k$ and $l$ are the first two digits of the continued fraction expansion of $y$, and $f_Y$ acts on the expanding coordinate by $y\mapsto \mathrm{G}^{2}(y)$.
\item The branch $(k,l)$ is normalized by
\[\gamma_{k,l}=T_2^{\,l}T_1^{\,k}=\spmx{1+kl&-l\\-k&1},
\qquad \gamma_{k,l}^{-1}=L^{\,k}R^{\,l}=\spmx{1&l\\k&kl+1},\]
and a closed orbit of $f_Y$ with itinerary $\bigl((k_1,l_1),\dots,(k_j,l_j)\bigr)$ corresponds to the conjugacy class of $\gamma=L^{k_1}R^{l_1}\cdots L^{k_j}R^{l_j}$.
\item Put $\Theta(\xi,\eta)=\frac12\log(-\xi/\eta)-\log|\eta-\xi|$ on $\mca{R}$.
Then the return time of $\varphi$ to $\mca{R}$ satisfies
\[r=\log|F'(\eta)|+\Theta\circ f-\Theta,\]
where $F$ denotes the action of $f$ on the expanding coordinate.
Consequently the induced return time satisfies
\[r_Y=-2\log\bigl(y\,\mathrm{G}(y)\bigr)+\Theta\circ f_Y-\Theta
=\log\bigl|(\mathrm{G}^{2})'(y)\bigr|+\Theta\circ f_Y-\Theta,\]
so that the two sides have the same sums over the closed orbits,
and $-2\log(y\mathrm{G}(y))> 2\log 2>0$, while it is unbounded above.
\item If $\alpha$ denotes the cocycle attached to a homomorphism
$\rho\circ \iota_*:\pi_1(\mca{G})\to G$ and $h_{ij}=\rho\bigl(\iota_*(\theta_{v_i}\,e_{ij}\,\theta_{v_j}^{-1})\bigr)$ denotes its value on the edge $e_{ij}$, in the notation of \Cref{sec.flat}, then
\[\alpha(k,l)=h_{11}^{\,k-1}h_{12}h_{22}^{\,l-1}h_{21},\]
so that $\alpha$ depends only on the letter $(k,l)$ of the induced alphabet.
\end{enumerate}
\end{prop}

\begin{proof}
(1) By the proof of \Cref{prop.modular} (2), the branch $1$ acts on the expanding coordinate by $y\mapsto y/(1-y)$ and the branch $2$ by $y\mapsto y-1$.
Iterating the first branch gives $y\mapsto y/(1-ky)$, which lies in $(0,1)$ if and only if $y<1/(k+1)$,
so the block of ones has length $k=\lfloor 1/y\rfloor$ and ends at $y/(1-ky)=1/(1/y-k)=1/\mathrm{G}(y)\in [1,\infty)$.
The block of twos then subtracts $1$ exactly $l=\lfloor 1/\mathrm{G}(y)\rfloor$ times,
and leaves $1/\mathrm{G}(y)-l=\mathrm{G}^{2}(y)\in (0,1)$.

(2) The first assertion is the matrix product,
and the second follows from $T_1^{-1}=L$ and $T_2^{-1}=R$, as in the proof of \Cref{prop.modular} (4).

(3) We compute the return time of $\varphi$ to $\mca{R}$ directly, rather than invoking a Livsic type theorem.
Let $v\in \mca{R}$ have endpoints $(\xi,\eta)$,
and let $T\in \Gamma$ denote the element that normalizes the branch, as in the proof of \Cref{prop.modular} \textup{(2)}.
Write $T=\spmx{\ast&\ast\\c&d}$, only the bottom row being used below.
We first express the flow time along the geodesic through a M\"obius transformation $h$,
then compare $h$ with the corresponding map $h'$ after applying $T$,
and finally evaluate both maps at the two crossing points.

The map $h(z)=(z-\xi)/(\eta-z)$ belongs to $\PSL_2\R$ and maps the geodesic with endpoints $(\xi,\eta)$ to the imaginary axis,
on which the geodesic flow is $y\mapsto e^{t}y$;
hence the flow time from $z$ to $z'$ along that geodesic is $\log\bigl(\mathrm{Im}\,h(z')/\mathrm{Im}\,h(z)\bigr)$.
The crossing point with $e_0$ is $z_0=i\sqrt{-\xi\eta}$, and a direct computation gives
\[h(z_0)=i\,\frac{\sqrt{-\xi\eta}\,(\eta-\xi)}{\eta^{2}-\xi\eta},
\qquad \mathrm{Im}\,h(z_0)=\sqrt{-\xi/\eta}.\]

Let $h'$ denote the corresponding map for the endpoints $(T\xi,T\eta)$.
Writing $H$ and $H'$ for the matrices $\spmx{1&-\xi\\-1&\eta}$ and $\spmx{1&-T\xi\\-1&T\eta}$, one computes
\[H'T=\begin{pmatrix} (1,-\xi)/(c\xi+d)\\ (-1,\eta)/(c\eta+d)\end{pmatrix},\]
so that
\[h(z)=\frac{c\xi+d}{c\eta+d}\,h'(Tz)\qquad\text{for every } z\in \bb{H}^{2}.\]

Since the next crossing point is $T^{-1}\bigl(i\sqrt{-T\xi\cdot T\eta}\bigr)$, we obtain
\[r=\log\frac{c\xi+d}{c\eta+d}
+\tfrac12\log\frac{-T\xi}{T\eta}-\tfrac12\log\frac{-\xi}{\eta}.\]
Finally $(c\xi+d)(c\eta+d)=(\eta-\xi)/(T\eta-T\xi)$ and $F'(\eta)=(c\eta+d)^{-2}$, so that
\[\log\frac{c\xi+d}{c\eta+d}=\log|F'(\eta)|+\log(\eta-\xi)-\log(T\eta-T\xi),\]
which is the asserted identity with $\Theta$ as in the statement.

Summing over a return to $Y$ gives the formula for $r_Y$, because $|(\mathrm{G}^{2})'(y)|=|y\,\mathrm{G}(y)|^{-2}$.
On the branch $(k,l)$ we have $y\mathrm{G}(y)=1-ky<1/(k+1)\leq 1/2$, so that $-2\log(y\mathrm{G}(y))>2\log 2$.
As $y\to 0$ the digit $k$ tends to infinity and the roof tends to infinity.

(4) The cocycle of the two-symbol coding is the product of the elements $h_{ij}$ along the itinerary,
and the word of the branch $(k,l)$ is $1^{k}2^{l}$ followed by the letter $1$.
\end{proof}

The transfer operator attached to this system and to a unitary representation $\varrho$ of $G$ is
\[(\mca{L}_{s,\varrho}w)(z)=\sum_{k,l\geq 1}
\bigl(kz+kl+1\bigr)^{-2s}\,
w\!\left(\frac{z+l}{kz+kl+1}\right)\varrho\bigl(\alpha(k,l)\bigr),\]
where the inverse branches are the M\"obius transformations $\gamma_{k,l}^{-1}$ and the weights are their derivatives to the power $s$.
If $\varrho$ is trivial, then $\mca{L}_{s,\varrho}$ is the square of the operator
\[(\mca{L}_sw)(z)=\sum_{n\geq 1}(z+n)^{-2s}\,w\Bigl(\frac{1}{z+n}\Bigr)\]
of Mayer \cite[Proposition 1]{Mayer1990CMP}.

\begin{rem} \label{rem.modular.sources}
The section $\mca{R}$, with its two pieces and its two branches, is the cross-section of Adler and Flatto \cite[Section 3]{AdlerFlatto1984ETDS}.
It appears also in \cite[Section 3]{GuivarchLeJan1993ASENS} and in \cite[Section 3.3]{Heersink2019CMP},
and it is the case of $\PSL_2\Z$ of the construction of Pohl \cite[Section 3]{Pohl2014DCDS},
which gives cross-sections with finitely many pieces for the Fuchsian groups with cusps that satisfy condition (A) of \cite{Pohl2014DCDS}.
The itinerary of a vector of $\mca{R}$ is the cutting sequence of Series \cite{Series1985JLMS}, and the rectangle $Y$ is one half of her cross-section,
so that \Cref{prop.modular} \textup{(1)} and \textup{(4)}, \Cref{lem.tails}, and \Cref{prop.induced} \textup{(1)} and \textup{(2)} are contained in \cite[Theorems A, B, and C]{Series1985JLMS}.
For the relation of \Cref{prop.induced} \textup{(3)}, see \cite[Section 3.2]{Series1985JLMS}, \cite[Theorem 6.1]{KatokUgarcovici2007BAMS}, and \cite[(3.5)]{Heersink2019CMP}.
\end{rem}

\subsection{Removing finitely many modular knots} \label{subsec.refine}

In order to treat a homomorphism defined on $\pi_1(S^{3}-L_n)$ for a general $L_n$,
we refine the Markov section so that $L_n$ lies in the boundary of the rectangles.
The refinement uses the forward invariance of the cutting set.

Let $L'=L_n-K_{2,3}$, a finite union of closed geodesics in $M$, and let
\[\mca{S}^{+}=\{v^{+}\mid v\in L'\cap \mca{R}\},\qquad
\mca{S}^{-}=\{v^{-}\mid v\in L'\cap \mca{R}\},\]
which are finite sets, since each closed geodesic crosses $\mca{R}$ finitely many times.
Let $\mca{P}$ denote the partition of $\mca{R}$ into the rectangles cut out by the lines $\{v^{+}=c\}$ with $c\in \mca{S}^{+}\cup\{1\}$
and the lines $\{v^{-}=d\}$ with $d\in \mca{S}^{-}\cup\{-1\}$,
each cutting arc $\{v^{+}=c\}$ being placed in the adjacent rectangle with $v^{+}\leq c$
and each cutting arc $\{v^{-}=d\}$ in the adjacent rectangle with $v^{-}\geq d$,
as in \Cref{prop.modular} and \Cref{prop.modular.four}.

\begin{prop} \label{prop.refine}
The following three assertions hold.
\begin{enumerate}
\item The family $\mca{P}$ is a Markov section with escape locus
$(\ell^{+},\ell^{-})$, and conditions \textup{(E0)}, \textup{(E1)}, \textup{(E2)}, \textup{(E3)}, and \textup{(E4$'$)} hold for it.
\item For every $n\in \Z_{\geq 1}$, the number of $\mca{P}$-itineraries of
length $n$ that are compatible with a given itinerary of length $n$ for the partition $\{\mca{R}_1,\mca{R}_2\}$ is at most $(\#\mca{S}^{+}+1)(\#\mca{S}^{-}+2)$.
\item The system induced by $\mca{P}$ on $Y$ has the same M\"obius branches
$\gamma_{k,l}^{-1}$ as in \Cref{prop.induced} \textup{(2)},
each of them occurring at most $(\#\mca{S}^{+}+1)(\#\mca{S}^{-}+2)$ times,
and the cocycle $\alpha$ is a function of the induced letter alone.
\end{enumerate}
\end{prop}

\begin{proof}
(1) The set $L'$ is a union of closed orbits,
so the first return map of the section permutes the finite set $L'\cap \mca{R}$.
Its action $F$ on the coordinate $v^{+}$ therefore satisfies $F(\mca{S}^{+})\subset \mca{S}^{+}$,
and symmetrically the backward map maps $\mca{S}^{-}$ into itself.
The image under a branch of an interval with endpoints in $\mca{S}^{+}\cup\{0,1,\infty\}$ therefore has its endpoints in that set again,
and it is consequently a union of intervals of the partition;
the symmetric statement in the contracting direction holds because the two branches map $I^{-}$ into disjoint ranges, as verified below in the argument for (E3).
That is condition (iii) of \Cref{dfn.markov}.
Conditions (i), (ii), and (iv) are unchanged.
The new cuts add finitely many faces, each of them swept out by a stable or by an unstable arc, and $N$ is unchanged.
The closures of the new faces are properly embedded by the same analysis as for the sheets in the proof of \Cref{prop.modular} (3).
Hence (E1) and (E2) hold, and (E4$'$) holds because every arc of $\mca{P}$ is contained in an arc of the given section.
Condition (E0) holds because the surface $\mca{R}$ is unchanged.
For (E3), the argument of \Cref{prop.modular} (3) applies.
The heights outside the transition intervals form the finite set consisting of $\mca{S}^{+}\cup\{1\}$ and of its preimages under the two branches,
and the corresponding stable arcs are singular for $\mca{P}$.
In the contracting direction, the two branches map $I^{-}$ into the disjoint ranges $(-1,0)$ and $(-\infty,-1)$,
so a point of $\intr \mca{R}$ whose image lies on a cut $\{v^{-}=d\}$ with $d\in \mca{S}^{-}$ has contracting coordinate $T_i^{-1}(d)=v^{-}(f^{-1}(p))\in \mca{S}^{-}$,
where $p$ is the point of $L'\cap \mca{R}$ on that cut,
and hence lies on a cut $\{v^{-}=d'\}$ with $d'\in \mca{S}^{-}$ itself,
while the cut $\{v^{-}=-1\}$ has no preimage in $\intr \mca{R}$.
Therefore a last crossing at a time $t\leq 0$ lies in a transition cell or on a singular arc,
and a first crossing in positive time lies in $\ell^{-}$, as before.

(2) The $\mca{P}$-itinerary of a point is determined by the positions of $F^{k}(v^{+})$ relative to $\mca{S}^{+}$
and of the contracting coordinates of $f^{k}(v)$ relative to $\mca{S}^{-}\cup\{-1\}$, for $0\leq k<n$.
The first kind of information is the position of $v^{+}$ relative to $\bigcup_{k<n}F^{-k}(\mca{S}^{+})$.
Since $F(\mca{S}^{+})\subset \mca{S}^{+}$ by (1), the family $F^{-k}(\mca{S}^{+})$ increases with $k$, so that
\[\bigcup_{k<n}F^{-k}(\mca{S}^{+})=F^{-(n-1)}(\mca{S}^{+}).\]
Both branches of $F$ map their domain onto the whole of $I^{+}$ by \Cref{prop.modular} (2),
so $F^{n-1}$ maps a cylinder of length $n$ bijectively onto a cylinder of length one.
Hence that cylinder contains at most $\#\mca{S}^{+}$ points of $F^{-(n-1)}(\mca{S}^{+})$,
and it is cut into at most $\#\mca{S}^{+}+1$ pieces.
The second kind of information is determined by the given itinerary and by the position of $v^{-}$ relative to $\mca{S}^{-}\cup\{-1\}$.
Indeed, the contracting coordinate of $f^{k}(v)$ is the image of $v^{-}$ under the composition $T$ of the $k$ branches named by the given itinerary, and $T$ is monotone.
If a point $d\in \mca{S}^{-}$ lies in $T(I^{-})$, then $T^{-1}(d)\in \mca{S}^{-}$ by (1), and the point $-1$ does not lie in $T(I^{-})$ for $k\geq 1$.
This gives the factor $\#\mca{S}^{-}+2$.

(3) The set $Y$ is a union of rectangles of $\mca{P}$,
and the first return time to $Y$ is not changed by the refinement,
so the branches are the M\"obius transformations $\gamma_{k,l}^{-1}$ of \Cref{prop.induced} \textup{(2)}.
A letter of the induced system is a return word in the alphabet $\mca{P}$,
and by (2) the number of such words lying over a given pair $(k,l)$ is at most $(\#\mca{S}^{+}+1)(\#\mca{S}^{-}+2)$.
The cocycle $\alpha$ is the product of the elements $h_{ij}$ along the return word,
and every return word begins in $Y$ and ends in $Y$, so that $\alpha$ is determined by the word.
\end{proof}

\begin{lem} \label{lem.refine.realized}
All but finitely many periodic points of the shift of the refined section $\mca{P}$ are realized.
\end{lem}

\begin{proof}
Let $x'$ be a periodic point of the shift of $\mca{P}$.
The cells of a finite refined cylinder are nonempty, by condition (iii) of \Cref{dfn.markov},
so the closures in $[0,\infty]$ of the cells of $x'$ in the expanding coordinate form a nested family of nonempty compact intervals,
and their intersection contains a point $y^{*}$ fixed by the composition of the closed branches along one period.
If $y^{*}$ lies in the interior of every cell of $x'$,
and if the symmetric statement holds in the contracting coordinate, then $x'$ is realized.
Otherwise some iterate of $y^{*}$ under the closed branches lies in $\mca{S}^{+}\cup\{0,1,\infty\}$,
or the symmetric statement holds in the contracting coordinate,
and we treat the first case, the second case being symmetric.
The set $\mca{S}^{+}$ is finite and forward invariant,
so a periodic $y^{*}$ lying in $\mca{S}^{+}$ is the coordinate $v^{+}$ of a component of $L'$,
and there are finitely many periodic points $x'$ whose $y^{*}$ is one of those coordinates, because the branches are increasing,
so that the side of the cutting arc at one crossing determines the sides at all the other crossings.
The periodic points of the closed branches in the grand orbit of $\{0,1,\infty\}$ are $0$ and $\infty$,
which are the fixed coordinates of the two constant base words;
near each of them the cutting levels are avoided, only one cell occurs,
and hence there are finitely many periodic points $x'$ whose $y^{*}$ is $0$ or $\infty$.
This proves the lemma.
\end{proof}

\subsection{The generic ordering} \label{subsec.modular.generic}

\begin{thm}[{Precise form of \Cref{thmA.generic}}] \label{thm.modular}
Let $h$ be a positive H\"older roof function on $\Sigma(\mca{G})$,
and let $\ell_h$ be the attached length function.
Assume that $\ell_h(\gamma_1)/\ell_h(\gamma_2)\notin \Q$ holds for some primitive hyperbolic $\gamma_1,\gamma_2\in \PSL_2\Z$.
The function $\ell_h$ satisfies $h_{\min}\,\ell_{\rm c}\leq \ell_h\leq h_{\max}\,\ell_{\rm c}$ and is therefore comparable to the combinatorial length.
Then the modular knots ordered by $\ell_h$, with the missing trefoil inserted at an arbitrary fixed position
and with an arbitrary order among the knots of equal length, obey the Chebotarev law in $S^3$.
In particular the set of modular knots together with the missing trefoil is a Chebotarev link in $S^{3}$.
\end{thm}

\begin{eg} \label{eg.explicit.ordering}
Let $h$ take the value $\alpha>0$ on the sequences with first letter $1$ and the value $\beta>0$ on those with first letter $2$.
Then $\ell_h=\alpha\,n_L+\beta\,n_R$ holds in the notation of \Cref{subsec.ord.roof},
and the closed orbits of the words $LR$ and $L^{2}R$ have the lengths $\alpha+\beta$ and $2\alpha+\beta$.
Hence the assumption of \Cref{thm.modular} holds if and only if $\alpha/\beta\notin \Q$.
For $\alpha=1-\varepsilon$ and $\beta=1+\varepsilon$ we obtain the ordering by $\ell_{\rm c}+\varepsilon\Psi$,
which obeys the Chebotarev law for every irrational $\varepsilon$ with $|\varepsilon|<1$,
whereas for $\varepsilon=0$ the law fails whenever $K_{2,3}$ belongs to the link, by \Cref{thm.combinatorial}.
\end{eg}

\begin{proof}
We apply \Cref{thm.main.general} \textup{(1)} to $\varphi$ on $M=S^3-K_{2,3}$ and to the section with four rectangles of \Cref{prop.modular.four}.
The map that sends a sequence of states $(u_n,i_n)_{n\in\Z}$ to the sequence $(i_n)_{n\in\Z}$ identifies the subshift of that section with $\Sigma(\mca{G})$,
and $h$ is regarded as a roof function on that subshift.
For a finite union $L'$ of modular knots, the partition $\mca{P}$ of \Cref{subsec.refine} is a refinement of that section along $L'$ by \Cref{prop.refine} \textup{(1)},
and the section itself is the case $L'=\emptyset$.
Hence the assumptions on the given section and on the refinements hold by \Cref{prop.refine} \textup{(1)} and by \Cref{lem.refine.realized}.
The realizing point is unique, for the given section and for the refined ones,
because the itinerary determines the unrefined itinerary,
which determines the pair $(v^{-},v^{+})$ and hence the point,
by the nested intervals in the proof of \Cref{lem.tails}.
Note that \Cref{dfn.Chebotarev} allows the ambient manifold $S^3$ and the extra knot $K_{2,3}$,
since $S^3-L_n=M-(L_n-K_{2,3})$ holds whenever $L_n$ contains $K_{2,3}$.
If $L_n$ does not contain $K_{2,3}$, then filling in $K_{2,3}$ gives a surjection $\pi_1(M-L_n)\surj \pi_1(S^3-L_n)$,
through which we pull $\rho$ back without changing the Frobenius classes of the modular knots.
The assumption of \Cref{thm.main.general} on the lengths is the assumption on $\gamma_1$ and $\gamma_2$,
because the periodic orbits attached to the primitive hyperbolic classes are realized, by \Cref{prop.modular} \textup{(4)}.
\end{proof}

\subsection{The cusp as an obstruction} \label{subsec.cusp}

The theorem of McMullen \cite[Theorem 1.2]{McMullen2013CM} does not apply to the modular flow directly, and this subsection gives three reasons.
None of the statements of this subsection is used in the proofs of the main theorems.
The first statement concerns the symbolic model.

\begin{prop} \label{prop.no.finite.markov}
The geodesic flow on $M=T^1(\PSL_2\Z\backslash \bb{H}^2)$ admits no Markov section whose rectangles are compact.
\end{prop}

\begin{proof}
Suppose that the flow admits such a family $(\mca{R}_i)_{i=1}^{n}$,
and put $\mca{R}=\bigcup_{i}\mca{R}_i$, which is compact.
A Markov section has empty escape locus, so condition (i) of \Cref{dfn.markov} applies to every point of $\mca{R}$,
and iterating the first return map shows that the forward orbit of every $x\in \mca{R}$ meets $\mca{R}$ infinitely many times.
Since $\mca{R}$ is a compact surface transverse to the flow,
an orbit segment of finite length meets $\mca{R}$ only finitely many times,
and hence the forward orbit of every $x\in \mca{R}$ meets $\mca{R}$ at arbitrarily large times.
Now choose an unstable arc of $\mca{R}_1$ and a lift of $\mca{R}_1$ to $T^{1}\bb{H}^{2}$.
Along the lifted arc the forward endpoint runs injectively over an interval of $\partial \bb{H}^{2}$,
since the stable arcs are the level sets of the forward endpoint,
and that interval contains a point of $\bb{Q}\cup\{\infty\}$.
The corresponding vector $v\in \mca{R}_1$ generates a geodesic that converges to the cusp,
so that $\varphi_t(v)$ leaves every compact subset of $M$ as $t\to\infty$,
and in particular the forward orbit of $v$ meets $\mca{R}$ only finitely many times.
This contradiction proves the proposition.
\end{proof}

For the geodesic flows of the closed orientable $2$-orbifolds of negative Euler characteristic,
Tsang \cite[Construction 1.3]{Tsang2023NY} constructed explicit Markov partitions through veering triangulations.
For a hyperbolic surface with cusps, the construction of \cite[Remark 6.7]{Tsang2023NY} passes to the double of the surface along its cusps,
and it is stated there that the cycles of the resulting finite graph that avoid the boundary correspond to the closed orbits.
By \Cref{prop.no.finite.markov}, the rectangles of a Markov section of the modular flow itself cannot be compact,
and the section of \Cref{prop.modular} keeps the orbits that escape to the cusp.

The obstruction appears concretely in the section of \Cref{prop.modular}.
By \Cref{prop.modular} (2), the two branches of the return map act on the expanding coordinate by $x\mapsto x/(1-x)$ on $(0,1)$ and by $x\mapsto x-1$ on $(1,\infty)$;
the quotient of this map by the involution $x\mapsto 1/x$ is the Farey map.
This map has neutral fixed points at $0$ and at $\infty$,
and these correspond to the two fixed points of the shift, both of which represent the cusp.
The closed orbit of $\gamma=L^{q}R$ has $\ell(\gamma)<2\log(q+2)$ and crosses $\mca{R}$ exactly $q+1$ times by \Cref{prop.modular} \textup{(4)},
so its shortest return time is at most $2\log(q+2)/(q+1)$.
Hence the return time $r$ satisfies $\inf r=0$,
the infimum being over the points of $\mca{R}$ at which $r$ is defined,
and \Cref{thm.PP} does not apply with $h=r$ even though the alphabet is finite.

The second statement concerns the way in which McMullen obtains his examples in $S^{3}$.
The complement of the figure-eight knot is the complement of a closed orbit of an Anosov flow on a closed 3-manifold,
and the Chebotarev law is transported to $S^{3}$ by a Dehn surgery along that closed orbit \cite[Section 1]{McMullen2013CM}.
The linking number of a closed orbit with the figure-eight knot is the period of the corresponding periodic point of the monodromy, and hence it is positive.
In contrast, $\Psi(L^{q}R)=1-q$ and $\Psi(R^{q}L)=q-1$ hold:
the closed orbits that go deep into the cusp wind around it in both directions.
The next proposition shows that, for this reason,
no Dehn filling of $M=S^{3}-K_{2,3}$ makes the modular flow the restriction of a flow with the core of the filling as a closed orbit.

\begin{prop} \label{prop.not.closed.orbit}
Let $\varphi'$ be a continuous flow on a closed 3-manifold $M'$,
and let $K'\subset M'$ be a periodic orbit of $\varphi'$ that admits a tubular neighborhood.
Then the restriction of $\varphi'$ to $M'-K'$ is not orbit equivalent to the modular flow on $M$.
\end{prop}

\begin{proof}
Let $\tau_0>0$ denote the period of $K'$.
Suppose that there is a homeomorphism $\Xi\colon M\to M'-K'$ that maps every orbit of the modular flow onto an orbit of $\varphi'$ and preserves the orientation of the orbits.

\emph{Step 1.} Let $u\colon K'\to \R/\Z$ denote the coordinate with $u(\varphi'_t(x))=u(x)+t/\tau_0$,
let $U$ be a tubular neighborhood of $K'$ with a retraction $\pi\colon U\to K'$, and put $\bar u=u\circ \pi$.
For a path $\sigma'$ in $U$, let $\Delta(\sigma')\in \R$ denote the increment of a continuous lift of $\bar u$ along $\sigma'$.
For an orbit segment of duration $t$ in $K'$, the increment equals $t/\tau_0$.
Uniform continuity gives a neighborhood $U_0\subset U$ of $K'$
such that every $x\in U_0$ satisfies $\varphi'_{[0,\tau_0]}(x)\subset U$ and $|\Delta(\varphi'_{[0,t]}(x))-t/\tau_0|<1/4$ for every $t\in [0,\tau_0]$.
If an orbit segment $\sigma'$ of duration $t$ is contained in $U_0$,
then we divide it into segments of duration at most $\tau_0$, and we obtain
\begin{equation} \label{eq.one.direction}
\Delta(\sigma')\geq -\frac14
\qquad\text{and}\qquad
\Delta(\sigma')\geq \frac{3t}{4\tau_0}-1.
\end{equation}

\emph{Step 2.} Since $M'$ is closed, the set $\Xi^{-1}(M'-\intr U_0)$ is compact.
Hence there is $Y_1>1$ such that the cusp region $E\subset M$,
which consists of the unit vectors based at the points of height $y\geq Y_1$, satisfies $\Xi(E)\subset U_0-K'$.
We use the coordinates $(x \bmod 1,\,y,\,\theta)$ on $E$, where $\theta$ is the angle of the vector.
The group $H^{1}(E;\Z)$ is generated by the classes of $dx$ and of $d\theta/2\pi$,
so that there are $a,b\in \Z$ and a continuous function $g\colon E\to \R$ with $\bar u\circ \Xi=a\,x+b\,\theta/2\pi+g$ modulo $1$ on $E$.
Put $c_0=2\sup\{|g(w)|\mid w\in E,\ y(w)=Y_1\}$, which is finite.

\emph{Step 3.} For $R>Y_1$, let $\sigma_R^{+}$ and $\sigma_R^{-}$ denote the orbit segments in $E$ along a half circle of radius $R$ centered at a real point,
oriented to the right and to the left respectively.
A direct computation gives
\[\Delta x(\sigma_R^{\pm})=\pm 2\sqrt{R^{2}-Y_1^{2}},\qquad
\Delta\theta(\sigma_R^{\pm})=\mp 2\arccos(Y_1/R).\]
The endpoints of these segments have height $Y_1$, and therefore
\[\Bigl|\Delta\bigl(\Xi(\sigma_R^{\pm})\bigr)\mp
\Bigl(2a\sqrt{R^{2}-Y_1^{2}}-\frac{b}{\pi}\arccos\frac{Y_1}{R}\Bigr)\Bigr|
\leq c_0.\]
Both $\Xi(\sigma_R^{+})$ and $\Xi(\sigma_R^{-})$ are orbit segments of $\varphi'$ in $U_0$,
so that the first inequality of \eqref{eq.one.direction} gives $\bigl|2a\sqrt{R^{2}-Y_1^{2}}-(b/\pi)\arccos(Y_1/R)\bigr|\leq c_0+1/4$ for every $R>Y_1$.
Hence $a=0$, and $\Delta(\Xi(\sigma_R^{+}))$ is bounded independently of $R$.

\emph{Step 4.} We claim that the duration of $\Xi(\sigma_R^{+})$ tends to infinity as $R\to\infty$.
If not, then the durations are at most some $D>0$ along a sequence $R_k\to \infty$.
The initial points $p_k$ and the final points $q_k$ of $\sigma_{R_k}^{+}$ lie in the compact torus $\{y=Y_1\}$;
after we pass to a subsequence, they converge to a vector $p_\infty$ that points upward and to a vector $q_\infty$ that points downward.
By the continuity of $\varphi'$, there is $t_\infty\in [0,D]$ with $\Xi(q_\infty)=\varphi'_{t_\infty}(\Xi(p_\infty))$,
so that $q_\infty$ lies on the forward orbit of $p_\infty$.
That is impossible, because the forward orbit of $p_\infty$ goes up the cusp and never returns to the height $Y_1$.
Therefore the second inequality of \eqref{eq.one.direction} shows that $\Delta(\Xi(\sigma_R^{+}))$ tends to infinity, which contradicts Step 3.
\end{proof}

The construction of McMullen needs less than an orbit equivalence,
namely a homeomorphism that matches the closed orbits, but that is not possible either.

\begin{cor} \label{cor.not.closed.orbit}
Let $\varphi'$ be a flow without fixed points on a closed 3-manifold $M'$,
and let $K'\subset M'$ be a periodic orbit that admits a tubular neighborhood.
Assume that every orbit segment of $\varphi'$ in $M'-K'$ is a limit,
in the Hausdorff topology, of segments of periodic orbits; this holds for a transitive pseudo-Anosov flow.
Then there is no homeomorphism $\Xi\colon M\to M'-K'$ that maps the modular knots bijectively onto the periodic orbits of $\varphi'$ other than $K'$.
\end{cor}

\begin{proof}
Suppose that such a homeomorphism $\Xi$ exists.
Every orbit segment $s$ of the modular flow is a limit of segments $s_k$ of closed geodesics, by the closing lemma.
Let $x\in s$, and let $B$ be a flow box of $\varphi'$ around $\Xi(x)$.
If $s'\subset s$ is a short segment around $x$,
then the corresponding segments $s'_k$ satisfy $\Xi(s'_k)\subset B$ for large $k$.
A connected subset of an orbit that is contained in a flow box is contained in a single plaque,
and a limit of subsets of plaques is contained in a plaque.
Hence $\Xi(s')$ is contained in a plaque, and $\Xi$ maps every orbit of the modular flow into an orbit of $\varphi'$.
The same argument applies to $\Xi^{-1}$, by the assumption on $\varphi'$;
the approximating periodic orbits are different from $K'$ for large $k$,
because $K'$ is closed and the segment being approximated lies in $M'-K'$.
Therefore $\Xi$ is an orbit equivalence.
It preserves the orientation of all orbits or reverses the orientation of all orbits, because $M$ is connected,
and \Cref{prop.not.closed.orbit}, applied to $\varphi'$ or to the flow $\varphi'$ with reversed time, gives a contradiction.
\end{proof}
\internal{[Verified 2026-09-22. The weakest form of the route of
McMullen, a bijection of closed orbits that preserves the free homotopy
classes in the complements of all finite sublinks, is not excluded.]}

The cusp appears in the symbolic model in the same way:
the two fixed points of the shift represent the parabolic class of $R$ and the class of $L$,
which is conjugate to $R^{-1}$, so that the cusp occurs with both orientations.
If the cusp is opened to a funnel, then there are two closed orbits with opposite orientations on the boundary torus \cite[Section 3]{BonattiPinsky2021}.

The third statement concerns the completions of the modular flow in which the missing trefoil does not become a closed orbit.

\begin{q} \label{q.time.change}
Is $\ell_h$ the length function of a time change of the modular flow?
That is, is there a positive continuous function $a$ on $S^{3}-K_{2,3}$
such that $\ell_h(C_\gamma)=\int_{C_\gamma}a$ holds for every primitive hyperbolic $\gamma$?
\end{q}

If such a function $a$ exists, then it is unbounded.
Indeed, if $a$ were bounded, then $\ell_h(C_\gamma)\leq \|a\|_\infty\,\ell(\gamma)$ would hold for every $\gamma$,
whereas $\ell_h$ is bounded below by a positive multiple of $\ell_{\rm c}$
and the ratio $\ell_{\rm c}/\ell$ is unbounded by \Cref{prop.incommensurable} \textup{(2)}.

An example of such a completion is obtained from the modular flow by a time change that tends to zero at the cusp fast enough;
the vector field is then extended by zero to $K_{2,3}$.
The next proposition shows that the hyperbolic length is not the period function of any completion of that kind by a Lipschitz flow on $S^{3}$,
and that the ordering by the periods differs from the trace ordering.

\begin{prop} \label{prop.no.completion}
Let $\varphi$ be the flow of a Lipschitz vector field on $S^{3}$,
and let $\Xi\colon S^{3}\to S^{3}$ be a homeomorphism that maps $K_{2,3}$ onto a smooth knot invariant under $\varphi$
and maps every orbit of the modular flow onto an orbit of $\varphi$, preserving the orientation of the orbits.
For a primitive hyperbolic $\gamma\in \PSL_2\Z$,
let $T(\gamma)$ denote the period of the closed orbit $\Xi(C_\gamma)$ of $\varphi$.
Then there is a constant $c>0$ such that $|\Psi(\gamma)|\leq c\,T(\gamma)$ holds for every primitive hyperbolic $\gamma\in \PSL_2\Z$.
In particular $T$ is not bounded above by any multiple of the hyperbolic length $\ell$, and $T\neq \ell$.
Moreover, the ordering of the modular knots by $T$ differs from the ordering by $\ell$:
there are infinitely many pairs $(\gamma,\gamma')$ with $\ell(\gamma)<\ell(\gamma')$ and $T(\gamma)>T(\gamma')$.
\end{prop}

\begin{proof}
Put $K'=\Xi(K_{2,3})$, which is a trefoil knot and hence a fibered knot.
Let $f\colon S^{3}-K'\to \R/\Z$ be a fibration that is an open book with binding $K'$:
there is a tubular neighborhood $U\cong K'\times D^{2}$ of $K'$,
with coordinates $(s,re^{2\pi i\theta})$, on which $f(s,re^{2\pi i\theta})=\theta$.
Put $\omega=f^{*}d\theta$, a closed $1$-form on $S^{3}-K'$ whose integral over a meridian of $K'$ is $1$,
so that $\operatorname{lk}(C,K')=\int_C\omega$ holds for every closed curve $C$ in $S^{3}-K'$.

Let $v$ denote the vector field of $\varphi$.
Since $v$ is Lipschitz and the coordinates on $U$ are smooth,
there is $b_1>0$ with $|v(x)-v(y)|\leq b_1|x-y|$ for all $x,y\in U$,
the norm being the Euclidean norm in the coordinates.
For $x=(s,re^{2\pi i\theta})\in U-K'$, put $\pi(x)=(s,0)\in K'$.
Since $K'$ is invariant under $\varphi$, the vector $v(\pi(x))$ is tangent to $K'$,
so the component of $v(x)-v(\pi(x))$ in the direction of $\partial_\theta$ is that of $v(x)$,
and its Euclidean length is $2\pi r\,|d\theta(v(x))|$.
Hence $2\pi r\,|d\theta(v(x))|\leq |v(x)-v(\pi(x))|\leq b_1r$, that is, $|\omega(v(x))|\leq b_1/2\pi$ on $U-K'$.
On the compact set $S^{3}-\intr U$ the function $|\omega(v)|$ is bounded as well, say by $b_2$.
Therefore every closed orbit $C$ of $\varphi$ in $S^{3}-K'$ satisfies
\[|\operatorname{lk}(C,K')|=\Bigl|\int_C\omega\Bigr|\leq c\,T(C),\qquad
c=\max(b_1/2\pi,b_2),\]
where $T(C)$ is the period of $C$.

Now let $\gamma\in \PSL_2\Z$ be primitive hyperbolic.
The homeomorphism $\Xi$ maps $C_\gamma$ onto $\Xi(C_\gamma)$ and $K_{2,3}$ onto $K'$,
and a homeomorphism of $S^{3}$ multiplies all linking numbers by a common sign,
so $|\operatorname{lk}(\Xi(C_\gamma),K')|=|\operatorname{lk}(C_\gamma,K_{2,3})| =|\Psi(\gamma)|$ by \Cref{thm.ghys}.
This proves the inequality.
For $\gamma=L^{q}R$ with $q\in \Z_{\geq 1}$ we have $|\Psi(\gamma)|=q-1$ and $\ell(\gamma)<2\log(q+2)$,
by the proof of \Cref{prop.incommensurable} \textup{(2)},
so an inequality $T\leq c'\ell$ would give $q-1\leq 2cc'\log(q+2)$ for every $q$, which is false.

For the last assertion, let $\Lambda\subset M$ denote the set of the unit vectors whose geodesic has a cutting sequence in which every block of equal letters has length at most two.
It is a compact subset of $M$, invariant under the modular flow.
The homeomorphism $\Xi$ maps orbits onto orbits, and $\varphi$ has no fixed point in $\Xi(\Lambda)$.
Hence the time that $\varphi$ takes from $\Xi(w)$ to the image under $\Xi$ of the time-one map of the modular flow at $w$ is a continuous function of $w\in \Lambda$, and it is bounded by a constant $b_3>0$.
It follows that $T(\gamma')\leq b_3(\ell(\gamma')+1)$ holds for every $\gamma'$ with $C_{\gamma'}\subset \Lambda$.
The classes $\gamma'_m=(LR)^{m}L^{2}R$ are primitive and satisfy $C_{\gamma'_m}\subset \Lambda$,
and $\ell(\gamma'_m)=4m\log\kappa+O(1)$ holds with $\kappa$ as in \Cref{prop.incommensurable}.
Hence there is $b_4>0$ such that for every large $q$ some $m$ satisfies $\ell(L^{q}R)<\ell(\gamma'_m)\leq \ell(L^{q}R)+b_4$.
Then $T(\gamma'_m)\leq b_3(2\log(q+2)+b_4+1)$, whereas $T(L^{q}R)\geq (q-1)/c$,
so that $T(L^{q}R)>T(\gamma'_m)$ holds for every large $q$.
\end{proof}

The completion of Bonatti and Pinsky of \Cref{subsec.bg.lorenz} is not of the kind considered in \Cref{prop.no.completion}:
it completes the geodesic flow of a funnel deformation,
and it has wandering points in the complement of the invariant trefoil,
whereas every point of $M$ is non-wandering for the modular flow.
The periods of that completion are comparable to $\ell_{\rm c}$,
because its closed orbits lie in a compact hyperbolic set,
on which the return time of a section is bounded above and bounded away from zero.

\section{The combinatorial ordering} \label{sec.combinatorial}

The combinatorial length is the length attached to the constant roof function,
whose suspension has a circle factor, so that \Cref{thm.main.general} does not apply.
This section proves \Cref{thmA.combinatorial} by the Perron--Frobenius theory of the finite matrices attached to the refined sections of \Cref{subsec.refine}.
The failure of the law in \Cref{rem.exact.length} is the only one.
For the ordering by combinatorial length we count the modular knots by the threshold $\ell_{\rm c}\leq N$,
as in \Cref{thmA.combinatorial}, and any ordering among the knots of the same length is allowed;
see \Cref{rem.combinatorial.ties}.

We first compute the linking number with a fixed modular knot along the two families of words $R^{a}L$ and $RL^{b}$,
whose closed geodesics go deep into the cusp.
Let $w$ be a cyclic word in $L$ and $R$, and let $P$ be a word in $L$ and $R$.
Then $\operatorname{occ}_P(w)$ denotes the number of cyclic occurrences of $P$ in $w$,
that is, the number of $i\in\{1,\dots,|w|\}$ such that $P$ is a prefix of the periodic sequence obtained from the $i$-th cyclic permutation of $w$.
Simon \cite[Proposition 7.1]{Simon2025GT} proved that,
for primitive hyperbolic $\gamma,\sigma\in \PSL_2\Z$ that are not conjugate,
\begin{equation} \label{eq.simon}
\operatorname{lk}(C_\gamma,C_\sigma)
=\frac12\sum_{u}\Bigl(\operatorname{occ}_{RuL}(\gamma)\operatorname{occ}_{LuR}(\sigma)
+\operatorname{occ}_{LuR}(\gamma)\operatorname{occ}_{RuL}(\sigma)\Bigr),
\end{equation}
where $u$ runs through all words in $L$ and $R$, the empty word included,
and where $\gamma$ and $\sigma$ are identified with their cyclic words.

\begin{lem} \label{lem.link.cusp}
Let $\sigma$ be a primitive hyperbolic conjugacy class of $\PSL_2\Z$,
and let $n_R(\sigma)$ and $n_L(\sigma)$ be the numbers of letters $R$ and $L$ in the cyclic word of $\sigma$,
and let $m_R(\sigma)$ and $m_L(\sigma)$ be the largest exponents of $R$ and of $L$ that occur in that word.
If $a\geq m_R(\sigma)$ is an integer with $R^{a}L\neq \sigma$,
then $\operatorname{lk}(C_{R^{a}L},C_\sigma)=n_R(\sigma)$.
If $b\geq m_L(\sigma)$ is an integer with $RL^{b}\neq\sigma$,
then $\operatorname{lk}(C_{RL^{b}},C_\sigma)=n_L(\sigma)$.
\end{lem}

\begin{proof}
We evaluate \eqref{eq.simon} for $\gamma=R^{a}L$ with $a\geq m_R(\sigma)$.
A word $RuL$ occurs cyclically in $R^{a}L$ if and only if $RuL=R^{k+1}L(R^{a}L)^{t}$ for some $k\in\{0,\dots,a-1\}$ and $t\in \Z_{\geq 0}$, and then it occurs exactly once.
If $t\geq 1$, then $LuR=LR^{k}L(R^{a}L)^{t-1}R^{a+1}$ contains $R^{a+1}$ and does not occur in $\sigma$, because $a+1>m_R(\sigma)$.
If $t=0$, then $LuR=LR^{k+1}$, which occurs in $\sigma$ exactly as many times as there are maximal blocks of $R$ in $\sigma$ of length at least $k+1$;
the count is zero for $k+1>m_R(\sigma)$.
Hence the first sum in \eqref{eq.simon} equals $\sum_{k\geq 0}\#\{\text{maximal blocks of }R\text{ in }\sigma\text{ of length at least }k+1\}$, which is $n_R(\sigma)$.
The second sum is treated in the same way.
A word $LuR$ occurs in $R^{a}L$ if and only if $LuR=L(R^{a}L)^{t}R^{k+1}$ with $t\geq 0$ and $0\leq k\leq a-1$.
If $t\geq 1$, then $RuL=R(R^{a}L)^{t}R^{k}L$ contains $R^{a+1}$ and contributes nothing.
If $t=0$, then $RuL=R^{k+1}L$, which again counts the maximal blocks of $R$ in $\sigma$ of length at least $k+1$.
The two sums are equal, and the factor $\frac12$ gives the assertion.
The family $RL^{b}$ is treated by exchanging $L$ and $R$.
The same count appears in \cite[Remark 7.3]{Simon2025GT}.
\end{proof}

We next state a lemma of linear algebra.
It is a case of the twisted Perron--Frobenius theorem of Adachi and Sunada \cite{AdachiSunada1987JFA},
and we include the short argument because we need the equality case.
Let $\mca{G}$ be an irreducible finite directed graph with transition matrix $A=(A_{vw})_{v,w\in V(\mca{G})}$,
so that $A_{vw}=1$ if $(v,w)\in E(\mca{G})$ and $A_{vw}=0$ otherwise, and let $r>0$ be its Perron root.
Let $G$ be a finite group, let $\alpha:E(\mca{G})\to G$ be a map,
and for a unitary representation $\varrho:G\to \mathrm{U}(V_\varrho)$ let $L_\varrho$ be the operator on $\bigoplus_{v\in V(\mca{G})}V_\varrho$ given by
\[(L_\varrho f)(v)=\sum_{w:\,(v,w)\in E(\mca{G})}\varrho(\alpha(v,w))\,f(w).\]
For $x\in \Sigma(\mca{G})$ with $\sigma^{n}(x)=x$,
the product $\alpha(x_0,x_1)\alpha(x_1,x_2)\cdots\alpha(x_{n-1},x_0)$ is the holonomy of $x$,
whose conjugacy class we denote by $[x]$, and
\begin{equation} \label{eq.trace}
\operatorname{Tr}(L_\varrho^{n})=\sum_{x:\,\sigma^{n}(x)=x}\chi_\varrho([x]),
\end{equation}
where $\chi_\varrho$ is the character of $\varrho$.

\begin{lem} \label{lem.twisted}
In the situation above, let $\varrho$ be an irreducible unitary representation of $G$,
and assume that for some vertex $v_0$ the holonomies of the directed loops based at $v_0$ generate $G$.
Then the spectral radius of $L_\varrho$ is at most $r$,
and it equals $r$ if and only if $\varrho$ is one-dimensional, say $\varrho=\chi$,
and there is $\lambda\in S^{1}$ such that $\chi([x])=\lambda^{n}$ holds for every $x\in \Sigma(\mca{G})$ and every $n\in \Z_{\geq 1}$ with $\sigma^{n}(x)=x$.
\end{lem}

\begin{proof}
Put $|f|(v)=\|f(v)\|$ for $f\in \bigoplus_v V_\varrho$.
Since the values of $\varrho$ are unitary, the triangle inequality gives $|L_\varrho f|\leq A|f|$ entrywise,
hence $|L_\varrho^{n}f|\leq A^{n}|f|$, and the spectral radius of $L_\varrho$ is at most $r$.

Suppose that $L_\varrho f=\mu f$ with $f\neq 0$ and $|\mu|=r$.
Then $r|f|\leq A|f|$.
Let $\psi$ be a left Perron vector of $A$, which has positive entries since $A$ is irreducible.
From $\psi(A|f|-r|f|)=0$ and $A|f|-r|f|\geq 0$ we get $A|f|=r|f|$,
so $|f|$ is a Perron vector and has positive entries.
Equality in the triangle inequality at a vertex $v$ means that the vectors $\varrho(\alpha(v,w))f(w)$,
for the edges $(v,w)$, are nonnegative multiples of one another, and they sum to $\mu f(v)$.
Comparing norms, we obtain $\varrho(\alpha(v,w))f(w)=\bigl(|f|(w)/(r|f|(v))\bigr)\mu f(v)$.
Put $u(v)=f(v)/|f|(v)$ and $\lambda=\mu/r$.
Then $u(v)$ is a unit vector and
\begin{equation} \label{eq.section}
\varrho(\alpha(v,w))\,u(w)=\lambda\,u(v)\qquad\text{for every edge }(v,w).
\end{equation}
Along a directed loop $(v_0,\dots,v_n=v_0)$, iterating \eqref{eq.section} gives $\varrho(g)u(v_0)=\lambda^{n}u(v_0)$, where $g$ is the holonomy of the loop.
The line spanned by $u(v_0)$ is therefore invariant under $\varrho(g)$ for every holonomy $g$ of a directed loop at $v_0$, hence under $\varrho(G)$ by the assumption.
Since $\varrho$ is irreducible, $V_\varrho$ is that line, so that $\varrho$ is a character $\chi$,
and $\chi([x])=\lambda^{n}$ for every periodic point as stated,
since $[x]$ is the holonomy of a directed loop at $x_0$.

Conversely, let $\varrho=\chi$ and $\lambda$ be as in the statement.
Then $\chi(\alpha(v,w))\lambda^{-1}$ has trivial holonomy along every directed loop.
We construct a map $u:V(\mca{G})\to S^{1}$ with $\chi(\alpha(v,w))u(w)=\lambda u(v)$ for every edge $(v,w)$.
Fix $u(v_0)=1$.
For a vertex $v$, choose a directed path $p$ from $v_0$ to $v$, of length $n$,
which exists since $\mca{G}$ is irreducible, and put $u(v)=\lambda^{n}\chi(\alpha(p))^{-1}$.
The value does not depend on $p$: two such paths, completed by a common directed path back to $v_0$,
give two directed loops, along both of which $\chi(\alpha(\,\cdot\,))\lambda^{-\mathrm{length}}$ takes the value $1$.
Then $f(v)=u(v)\phi(v)$, with $\phi$ a Perron vector of $A$, satisfies $L_\chi f=r\lambda f$.
\end{proof}

\begin{thm}[{\Cref{thmA.combinatorial}}] \label{thm.combinatorial}
Let $L$ be the union of finitely many modular knots, together with $K_{2,3}$ or not,
let $G$ be a finite group, and let $\rho:\pi_1(S^{3}-L)\surj G$ be a surjective homomorphism.
Let $C\subset G$ be a conjugacy class and let $n\in \Z_{\geq 1}$.
Put
\[P_n(C)=\#\{\gamma\mid \ell_{\rm c}(\gamma)=n,\ C_\gamma\not\subset L,\
\rho([C_\gamma])=C\},\qquad P_n=\#\{\gamma\mid \ell_{\rm c}(\gamma)=n\},\]
where $\gamma$ runs through the primitive hyperbolic conjugacy classes of $\PSL_2\Z$.
Then the following three conditions are equivalent.
\begin{enumerate}
\item For every conjugacy class $C\subset G$,
$\lim_{n\to\infty}P_n(C)/P_n=\#C/\#G$.
\item For every conjugacy class $C\subset G$,
$\lim_{N\to\infty}\sum_{n\leq N}P_n(C)/\sum_{n\leq N}P_n=\#C/\#G$.
\item There is no homomorphism $\chi:G\to\{\pm1\}$ such that
$\chi(\rho([C_\gamma]))=(-1)^{\operatorname{lk}(C_\gamma,K_{2,3})}$ holds for all but finitely many $\gamma$.
\end{enumerate}
Condition \textup{(3)} can fail only if $L$ contains $K_{2,3}$,
and then $(-1)^{\operatorname{lk}(\,\cdot\,,K_{2,3})}$ is the character of $\pi_1(S^{3}-L)$ that sends the meridian of $K_{2,3}$ to $-1$ and the other meridians to $1$.
If \textup{(3)} fails, then the homomorphism $\chi$ is unique, and
\[\frac{P_n(C)}{P_n}=\frac{\#C}{\#G}\bigl(1+(-1)^{n}\chi(C)\bigr)+o(1)
\qquad (n\to\infty)\]
holds for every conjugacy class $C\subset G$;
consequently the proportion of the class $C$ among the modular knots with $\ell_{\rm c}\leq N$ has the lower limit $\frac23\cdot\#C/\#G$ and the upper limit $\frac43\cdot\#C/\#G$.
\end{thm}

\begin{proof}
\emph{Step 1: the symbolic model.} Put $L'=L-K_{2,3}$ and $M=S^{3}-K_{2,3}$.
If $L$ does not contain $K_{2,3}$, then pull $\rho$ back to $\pi_1(M-L')$ through the surjection $\pi_1(M-L')\surj \pi_1(S^{3}-L)$,
as in the proof of \Cref{thm.modular}; the Frobenius classes of the modular knots are unchanged.
Refine the section of \Cref{sec.modular} along $L'$ as in \Cref{prop.refine},
and let $\mca{G}'$ denote the graph of the refined section,
which is irreducible by \Cref{lem.flow.irreducible}.
By \Cref{lem.refine.realized}, all but finitely many periodic points of $\Sigma(\mca{G}')$ are realized,
and by \Cref{prop.modular} \textup{(4)} a periodic point of least period $n$ is realized by a closed orbit of combinatorial length $n$,
since a modular knot crosses $\mca{R}$ exactly $\ell_{\rm c}$ times.
Conversely, all but finitely many modular knots are realized by a unique periodic point up to the shift,
by the uniqueness of the realizing point established in the proof of \Cref{thm.modular}.
The exceptions are among the finitely many modular knots that meet the cutting arcs, and each of them is realized by at most finitely many periodic points, because it crosses $\mca{R}$ finitely many times.
Since there are $2^{n}$ words of length $n$ in $L$ and $R$, the number of $x\in\Sigma(\mca{G}')$ with $\sigma^{n}(x)=x$ is $2^{n}+O(1)$,
so the Perron root of the transition matrix $A'$ of $\mca{G}'$ is $2$, and
\begin{equation} \label{eq.primitive.count}
P_n=\frac{2^{n}}{n}+O\bigl(2^{n/2}\bigr).
\end{equation}
Put $\rho'=\rho\circ\iota_*:\pi_1(\mca{G}')\to G$, which is surjective by \Cref{lem.surj},
and let $\alpha=\tilde\rho'$ denote the cocycle of \Cref{sec.flat}.
By \Cref{lem.homotopic} and by the formula for the Frobenius class in \Cref{sec.flat},
the holonomy class $[x]$ of a realized periodic point $x$ is $\rho([C_\gamma])$ for the modular knot $C_\gamma$ realizing $x$.
By \Cref{lem.dense} \textup{(1)},
the holonomies of the directed loops at any vertex form the whole group $G$,
so \Cref{lem.twisted} applies to $\mca{G}'$, $\alpha$,
and every irreducible unitary representation $\varrho$ of $G$.

\emph{Step 2: counting.} Fix a conjugacy class $C\subset G$.
For every $x$ with $\sigma^{n}(x)=x$, the orthogonality of the columns of the character table gives $\mathbf 1_{C}([x])=\frac{\#C}{\#G}\sum_{\varrho}\overline{\chi_\varrho(C)}\, \chi_\varrho([x])$,
where the sum is over the irreducible representations of $G$.
Summing over the periodic points of period $n$ and using \eqref{eq.trace}, we obtain
\begin{equation} \label{eq.count.twisted}
\#\{x\mid \sigma^{n}(x)=x,\ [x]\in C\}
=\frac{\#C}{\#G}\Bigl(2^{n}+O(1)+\sum_{\varrho\neq 1}\overline{\chi_\varrho(C)}
\operatorname{Tr}(L_\varrho^{n})\Bigr).
\end{equation}
The number of periodic points of period $n$ whose least period is smaller than $n$ is at most
\[\sum_{d\mid n,\ d<n}\bigl(2^{d}+O(1)\bigr)=O\bigl(2^{n/2}\bigr),\]
and the periodic points of least period $n$ are $n$ times as many as the closed orbits of least period $n$.
Together with the finitely many periodic points that are not realized, this gives
\begin{equation} \label{eq.Pn}
P_n(C)=\frac{1}{n}\,\#\{x\mid \sigma^{n}(x)=x,\ [x]\in C\}+O\bigl(2^{n/2}\bigr).
\end{equation}

\emph{Step 3: (3) implies (1).} Assume (3), and suppose that for some $\varrho\neq 1$ the spectral radius of $L_\varrho$ equals $2$.
By \Cref{lem.twisted}, $\varrho$ is a character $\chi$ and there is $\lambda\in S^{1}$ with $\chi([x])=\lambda^{n}$ for all periodic points, so that
\begin{equation} \label{eq.special}
\chi(\rho([C_\gamma]))=\lambda^{\ell_{\rm c}(\gamma)}
\end{equation}
holds for all but finitely many $\gamma$.
The character $\chi\circ\rho$ of $\pi_1(S^{3}-L)$ factors through $H_1(S^{3}-L)$,
which is free abelian on the meridians of the components of $L$,
so there are $\zeta_K\in S^{1}$, one for each component $K$ of $L$, with
\begin{equation} \label{eq.character}
\chi(\rho([C_\gamma]))=\prod_{K\subset L}\zeta_K^{\operatorname{lk}(C_\gamma,K)}.
\end{equation}

Consider the words $\gamma_a=R^{a}L$ with $a\geq 1$.
Then $\ell_{\rm c}(\gamma_a)=a+1$ and $\operatorname{lk}(C_{\gamma_a},K_{2,3})=\Psi(\gamma_a)=a-1$ by \Cref{thm.ghys},
and for each modular knot $K\subset L'$ the linking number $\operatorname{lk}(C_{\gamma_a},K)$ is independent of $a$ for $a$ large, by \Cref{lem.link.cusp}.
Comparing \eqref{eq.special} and \eqref{eq.character} for $a$ and $a+1$ large,
we get $\zeta_{K_{2,3}}=\lambda$ if $L$ contains $K_{2,3}$, and $\lambda=1$ otherwise.
The words $RL^{b}$ give in the same way $\zeta_{K_{2,3}}^{-1}=\lambda$, or $\lambda=1$.
Hence $\lambda\in\{\pm1\}$, and $\zeta_{K_{2,3}}=\lambda$ when $K_{2,3}\subset L$.
Since $\Psi(\gamma)\equiv\ell_{\rm c}(\gamma)\pmod 2$,
the factor $\zeta_{K_{2,3}}^{\Psi(\gamma)}$ in \eqref{eq.character},
when present, equals $\lambda^{\ell_{\rm c}(\gamma)}$, and \eqref{eq.special} becomes
\begin{equation} \label{eq.remaining}
\prod_{K\subset L'}\zeta_K^{\operatorname{lk}(C_\gamma,K)}=1
\qquad\text{for all but finitely many }\gamma.
\end{equation}
Each $\zeta_K$ is a root of unity, since $\chi$ has finite image;
let $m$ be a common multiple of their orders.
Let $\rho_m:\pi_1(S^{3}-L)\surj (\Z/m\Z)^{L'}$ send the class of a loop to its linking numbers with the components of $L'$ modulo $m$;
it is surjective, since the meridians generate $H_1(S^{3}-L)$.
By \Cref{lem.infinitely.many} applied to $\rho_m$,
every element of $(\Z/m\Z)^{L'}$ is the class $\rho_m([C_\gamma])$ of infinitely many $\gamma$.
By \eqref{eq.remaining}, the character $(k_K)_K\mapsto\prod_K\zeta_K^{k_K}$ of $(\Z/m\Z)^{L'}$ is trivial, so $\zeta_K=1$ for every $K\subset L'$.

If $L$ does not contain $K_{2,3}$, then $\chi\circ\rho$ is trivial on the classes of all but finitely many modular knots, hence on $G$,
because every element of $G$ is the class $\rho([C_\gamma])$ of infinitely many $\gamma$ by \Cref{lem.infinitely.many};
this contradicts $\chi\neq 1$.
If $L$ contains $K_{2,3}$, then $\chi(\rho([C_\gamma]))=\lambda^{\Psi(\gamma)}$ with $\lambda=\pm1$,
and $\lambda=1$ is excluded for the same reason,
so $\chi(\rho([C_\gamma]))=(-1)^{\operatorname{lk}(C_\gamma,K_{2,3})}$ for all but finitely many $\gamma$, which contradicts (3).
Therefore every $L_\varrho$ with $\varrho\neq 1$ has spectral radius smaller than $2$,
so $\operatorname{Tr}(L_\varrho^{n})=O(\theta^{n})$ for some $\theta<2$,
and \eqref{eq.count.twisted}, \eqref{eq.Pn}, and \eqref{eq.primitive.count} give $P_n(C)/P_n\to \#C/\#G$.

\emph{Step 4: (1) implies (2).} This follows from \eqref{eq.primitive.count},
since the classes of combinatorial length at most $N-m$ form a proportion $O(2^{-m})$ of the classes of combinatorial length at most $N$, as in the proof of \Cref{thm.two.densities}.

\emph{Step 5: (2) implies (3).} Suppose that $\chi$ is as in (3).
Then $\chi$ is nontrivial, because $\Psi(\gamma)$ is odd for infinitely many $\gamma$.
For all but finitely many $\gamma$ we have $\chi(\rho([C_\gamma]))=(-1)^{\Psi(\gamma)}=(-1)^{\ell_{\rm c}(\gamma)}$,
so $\rho([C_\gamma])$ lies in $\chi^{-1}(1)$ if $\ell_{\rm c}(\gamma)$ is even and in $\chi^{-1}(-1)$ if it is odd.
By \eqref{eq.primitive.count}, the proportion of the classes with $\ell_{\rm c}\leq N$ that have even combinatorial length has upper density $2/3$ and lower density $1/3$, as in \Cref{rem.exact.length}.
If (2) held for every conjugacy class, then the classes with $\rho([C_\gamma])\in\chi^{-1}(1)$ would have density $\#\chi^{-1}(1)/\#G=1/2$, which is a contradiction.

\emph{Step 6: the case in which \textup{(3)} fails.} Let $\chi$ be as in (3).
The argument of Step 3 shows that a representation $\varrho\neq 1$ for
which $L_\varrho$ has spectral radius $2$ is a character with $\varrho(\rho([C_\gamma]))=(-1)^{\operatorname{lk}(C_\gamma,K_{2,3})}$ for all but finitely many $\gamma$;
by \Cref{lem.infinitely.many} such a character is unique, so that $\varrho=\chi$.
The spectral radius of $L_\chi$ is $2$:
otherwise every $L_\varrho$ with $\varrho\neq 1$ would have spectral radius smaller than $2$, so that \textup{(1)} would hold as at the end of Step 3, and \textup{(2)} by Step 4, which contradicts Step 5.
Hence, by \Cref{lem.twisted}, there is $\lambda\in S^{1}$ with $\chi([x])=\lambda^{n}$ for every periodic point $x$ of period $n$,
and $\lambda=-1$, because $\chi([x])=(-1)^{n}$ holds for all but finitely many periodic orbits.
By \eqref{eq.section} with $\lambda=-1$, the matrix $L_\chi$ is conjugate to $-A'$ by a diagonal matrix,
and hence $\operatorname{Tr}(L_\chi^{n})=(-1)^{n}\operatorname{Tr}(A'^{n}) =(-1)^{n}(2^{n}+O(1))$.
Every other $\varrho\neq 1$ satisfies $\operatorname{Tr}(L_\varrho^{n})=O(\theta^{n})$ with some $\theta<2$.
We insert these values into the character sum of Step 2 and obtain the asymptotic for $P_n(C)/P_n$.
The statement about the thresholds follows, because $P_n\sim 2^{n}/n$ by \eqref{eq.primitive.count},
so that the knots with $\ell_{\rm c}\equiv N \pmod 2$ are asymptotically two thirds of the knots with $\ell_{\rm c}\leq N$.
\end{proof}

\begin{cor} \label{cor.combinatorial}
The modular knots, without the missing trefoil,
obey the Chebotarev law in $S^{3}$ for the ordering by combinatorial length,
in the sense that for every finite union $L$ of modular knots,
every surjection $\rho:\pi_1(S^{3}-L)\surj G$ onto a finite group,
and every conjugacy class $C\subset G$, condition \textup{(2)} of \Cref{thm.combinatorial} holds.
With the missing trefoil inserted, the same holds for exactly those $\rho$ for
which $(-1)^{\operatorname{lk}(\,\cdot\,,K_{2,3})}$ is not of the form $\chi\circ\rho$ with $\chi:G\to\{\pm1\}$.
\end{cor}

\begin{proof}
If $L$ does not contain $K_{2,3}$, then condition (3) of \Cref{thm.combinatorial} holds, by the last sentence of the theorem.
\end{proof}

\begin{rem} \label{rem.combinatorial.ties}
The threshold form of \Cref{thm.combinatorial} cannot be replaced by the form of \Cref{dfn.Chebotarev},
in which the knots of the same combinatorial length are put in some order:
those knots are about half of all the knots of combinatorial length at most $N$,
and an order that lists the knots of one class first changes the proportions at the intermediate positions.
For the orderings by $\ell_h$ and by $|\tr\gamma|$,
the knots of a given length are negligible in number, and the two forms agree.
\end{rem}

The ordering by the hyperbolic length $\ell(\gamma)=2\log\xi_\gamma$,
equivalently by $|\tr\gamma|$, is the subject of the next section.
The roof function of that ordering is the return time, which is unbounded on the induced system,
and the natural symbolic model is the coding by continued fractions, whose alphabet is countable.
\Cref{sec.arith} therefore applies the scheme of \Cref{subsec.PP} to that model,
with the transfer operator of Mayer in place of the operator of Parry and Pollicott \cite{ParryPollicott1990}.

\section{The trace ordering} \label{sec.arith}

\subsection{The transfer operator} \label{subsec.transfer}

The method of this section is that of Pollicott \cite{Pollicott1986BSMF},
with the transfer operator of Mayer \cite{Mayer1976BSMF} twisted by a representation as in the work of Fried \cite{Fried1996Invent}.
The new statement is the non-vanishing on the line $\mathrm{Re}\,s=1$ (\Cref{lem.non-vanishing}); see \Cref{rem.arith.sources} for the relationship with the earlier work.

We keep the notation of \Cref{subsec.induced},
and we write $f=\mathrm{G}^{2}$ for the induced map on $(0,1)$,
reusing the letter of the first return map of \Cref{sec.markov},
which does not occur below, $A=\Z_{\geq 1}\times \Z_{\geq 1}$ for its alphabet,
and $r_{k,l}(z)=2\log(kz+kl+1)$ for $(k,l)\in A$, taken with the principal branch of the logarithm,
which is possible because $\mathrm{Re}(kz+kl+1)>0$ holds on the disk $\Omega$ below.
The inverse branch of $f$ attached to the letter $(k,l)$ is the M\"obius transformation $\gamma_{k,l}^{-1}$ of \Cref{prop.induced} \textup{(2)},
and on $(0,1)$ we have $r_{k,l}=-\log\bigl|(\gamma_{k,l}^{-1})'\bigr|$.
We write $S_{v}=\gamma_{k,l}$ for the letter $v=(k,l)$.
For a word $w=v_1v_2\cdots v_n$ of letters $v_i\in A$, put $S_{w}=S_{v_n}\cdots S_{v_2}S_{v_1}$,
so that $S_{w}^{-1}=S_{v_1}^{-1}\circ S_{v_2}^{-1}\circ\cdots\circ S_{v_n}^{-1}$ is the inverse branch of $f^{n}$ attached to $w$.
Put also $\alpha(w)=\alpha(v_1)\,\alpha(v_2)\cdots \alpha(v_n)$,
the values being multiplied in the order of the crossings,
and $r_{w}(z)=\sum_{i=1}^{n}r_{v_i}\bigl(S_{v_{i+1}}^{-1}\circ\cdots\circ S_{v_n}^{-1}(z)\bigr)$, the composition being empty for $i=n$.
The fixed point of $S_{w}^{-1}$ in $(0,1)$ has itinerary $w$,
and we write $\alpha^{(n)}(x)=\alpha(w)$ for the length $n$ itinerary $w$ of a point $x\in (0,1)$.
Two features simplify the setting.
First, we adopt as the roof function the logarithm of the Jacobian,
\[r=-2\log\bigl(y\mathrm{G}(y)\bigr)=\log |f'|,\]
which is cohomologous to the return time $r_Y$ by \Cref{prop.induced} (3) and has the same sums over the closed orbits,
so that the transfer operator at $s=1$ preserves Lebesgue measure.
We write again $r$ for this roof; the return time of \Cref{sec.markov} does not occur below.
Second, the cocycle $\alpha$ depends only on the current letter by \Cref{prop.induced} (4).

For a finite dimensional unitary representation $\varrho$ of $G$ on $\C^{d}$, put
\[(\mca{L}_{s,\varrho}W)(z)=\sum_{(k,l)\in A}e^{-s\,r_{k,l}(z)}\,
W(\gamma_{k,l}^{-1}z)\,\varrho(\alpha(k,l)),\]
acting on the space of bounded holomorphic functions on the disk $\Omega=\{|z-1|<3/2\}$ with values in $\C^{d}$,
which contains $[0,1]$, the matrices $\varrho(\alpha(k,l))$ acting on the right.
With that convention the iterates multiply the values of $\alpha$ in the order of the crossings,
in agreement with the Frobenius classes of \Cref{subsec.frobenius}.
\begin{lem} \label{lem.nuclear}
The following three assertions hold.
\begin{enumerate}
\item Every inverse branch of $f$ maps $\ol{\Omega}$ into the disk
$\Omega_1=\{|z-1|\leq 1\}$, which is contained in the interior of $\Omega$.
\item For $\sigma=\mathrm{Re}\,s>1/2$ we have
$\sup_{z\in \ol{\Omega}}\sum_{(k,l)\in A}\bigl|e^{-s r_{k,l}(z)}\bigr|\leq 4^{\sigma}\zeta(2\sigma)^{2}\,e^{\pi|\mathrm{Im}\,s|}<\infty$.
\item For $\mathrm{Re}\,s>1/2$ the operator $\mca{L}_{s,\varrho}$ is nuclear of
order zero on the space of bounded holomorphic functions on $\Omega$ with values in $\C^{d}$,
and $s\mapsto \mca{L}_{s,\varrho}$ is holomorphic there.
\end{enumerate}
\end{lem}

\begin{proof}
(1) It suffices to treat the maps $z\mapsto 1/(c+z)$ with $c\in \Z_{\geq 1}$,
which are the inverse branches of the Gauss map,
since every inverse branch of $f=\mathrm{G}^{2}$ is a composition of two of them.
The set $c+\ol{\Omega}$ is the disk with real diameter $[c-\frac12,c+\frac52]$, which avoids the origin,
and hence the image of $\ol{\Omega}$ under $z\mapsto 1/(c+z)$ is the disk with real diameter $\bigl[\frac{1}{c+5/2},\frac{1}{c-1/2}\bigr]$.
For $c=1$ that disk has center $8/7$ and radius $6/7$, so it lies in $\Omega_1$,
and for $c\geq 2$ its real diameter is contained in $(0,\frac23]$, so it lies in $\Omega_1$ as well.

(2) For $z\in \ol{\Omega}$ we have $\mathrm{Re}\,z\geq -\frac12$, and hence for every $(k,l)\in A$
\[|kz+kl+1|=k\Bigl|z+l+\frac1k\Bigr|\geq k\Bigl(l-\frac12\Bigr)\geq
\frac{kl}{2}.\]
By the definition of $r_{k,l}$ we have $e^{-sr_{k,l}(z)}=(kz+kl+1)^{-2s}$, and
\[\bigl|(kz+kl+1)^{-2s}\bigr|=|kz+kl+1|^{-2\sigma}\,
e^{2\,\mathrm{Im}\,s\cdot\arg(kz+kl+1)}.\]
The real part of $kz+kl+1$ is positive on $\ol{\Omega}$, so that $|\arg(kz+kl+1)|<\pi/2$,
and the sum is bounded by $e^{\pi|\mathrm{Im}\,s|}\sum_{(k,l)\in A}(kl/2)^{-2\sigma} =4^{\sigma}\zeta(2\sigma)^2e^{\pi|\mathrm{Im}\,s|}$.

(3) Let $\mca{H}(\Omega)$ and $\mca{H}(\Omega_1)$ denote the spaces of bounded holomorphic functions with values in $\C^{d}$ on $\Omega$ and on $\Omega_1$,
and let $\mathrm{res}:\mca{H}(\Omega)\to \mca{H}(\Omega_1)$ be the restriction.
For $W\in \mca{H}(\Omega)$ with $\|W\|\leq 1$,
the Taylor expansion $W(z)=\sum_{k\geq 0}c_k(z-1)^{k}$ at the center of $\Omega$ satisfies $\|c_k\|\leq (2/3)^{k}$ by the Cauchy estimates, since $\Omega$ has radius $3/2$.
The truncation $T_nW=\sum_{k<n}c_k(z-1)^{k}$ has rank at most $nd$
and satisfies $\|\mathrm{res}\,W-T_nW\|_{\mca{H}(\Omega_1)}\leq \sum_{k\geq n}(2/3)^{k} =3\,(2/3)^{n}$, because $\Omega_1$ has radius one.
Hence the approximation numbers of $\mathrm{res}$ satisfy $a_{nd}(\mathrm{res})\leq 3(2/3)^{n}$.

By (1) we may factor $\mca{L}_{s,\varrho}=J\circ \mathrm{res}$,
where $J:\mca{H}(\Omega_1)\to \mca{H}(\Omega)$ is given by $JW=\sum_{(k,l)\in A}e^{-sr_{k,l}}\,(W\circ \gamma_{k,l}^{-1})\,\varrho(\alpha(k,l))$;
the map $J$ is well defined and bounded, with $\|J\|\leq \sum_{(k,l)\in A}\|e^{-sr_{k,l}}\|_{\infty}$,
which is finite by (2) since $\varrho$ is unitary and $\|\varrho(\alpha(k,l))\|=1$.
Therefore $a_{nd}(\mca{L}_{s,\varrho})\leq \|J\|\,3(2/3)^{n}$,
so the approximation numbers decay geometrically and $\mca{L}_{s,\varrho}$ is nuclear of order zero.
Holomorphy in $s$ follows from (2), since the series defining $J$ converges in operator norm, locally uniformly in $s$.
\end{proof}

The twist plays no role in this lemma; only $\|\varrho(\alpha(k,l))\|=1$ is used.

\subsection{The $L$-function as a Fredholm determinant}
\label{subsec.fredholm}

\begin{prop} \label{prop.fredholm}
For $\mathrm{Re}\,s>1$ we have
\[\det\bigl(I-\mca{L}_{s,\varrho}\bigr)=\prod_{k\geq 0}L(s+k,\varrho)^{-1},
\qquad
L(s,\varrho)=\prod_{\gamma}\det\bigl(I-\varrho([\gamma])e^{-s\ell(\gamma)}\bigr)^{-1},\]
the product being over the primitive closed geodesics.
In particular $L(s,\varrho)=\det(I-\mca{L}_{s+1,\varrho})/\det(I-\mca{L}_{s,\varrho})$ extends meromorphically to the half plane $\mathrm{Re}\,s>1/2$, which is all that we shall use.
\end{prop}

\begin{proof}
Expanding the $n$-th power, $\mca{L}_{s,\varrho}^{\,n}$ is the sum, over the directed paths $w$ of length $n$,
of the operators $W\mapsto e^{-s r_w}\,(W\circ S_w^{-1})\,\varrho(\alpha(w))$,
and $\sum_{|w|=n}\|e^{-sr_w}\|_{\infty}<\infty$ by \Cref{lem.nuclear} (2).
Each $S_w^{-1}$ maps $\ol{\Omega}$ into $\Omega_1$,
hence is a contraction of $\Omega$ for the hyperbolic metric and has a unique fixed point $z_w$ in $\Omega$,
which lies in $(0,1)$ and satisfies $(S_w^{-1})'(z_w)=\bigl((f^{n})'(z_w)\bigr)^{-1}$.
Conjugating so that $z_w=0$ and writing $\lambda=(S_w^{-1})'(0)$,
the operator $W\mapsto e^{-sr_w}\,(W\circ S_w^{-1})\,\varrho(\alpha(w))$ sends $z^{k}$ to a series whose terms have degree at least $k$,
with the coefficient of $z^{k}$ equal to $e^{-sr_w(z_w)}\lambda^{k}$ times the right multiplication by $\varrho(\alpha(w))$.
Its eigenvalues are therefore the numbers $e^{-sr_w(z_w)}\lambda^{k}$ with the multiplicities of the eigenvalues of $\varrho(\alpha(w))$,
and since a nuclear operator of order zero has a trace equal to the sum of its eigenvalues, that operator has trace
\[\frac{\mathrm{tr}\,\varrho(\alpha(w))\,e^{-s\,r_w(z_w)}}{1-(S_w^{-1})'(z_w)}.\]
The series converges in the nuclear norm, by the factorization in the proof of \Cref{lem.nuclear} (3) applied to each term, so the traces may be summed, and we obtain
\[\mathrm{tr}\,\mca{L}_{s,\varrho}^{\,n}
=\sum_{f^{n}x=x}\frac{\mathrm{tr}\,\varrho(\alpha^{(n)}(x))\,e^{-s\,r^{(n)}(x)}}
{1-\bigl((f^{n})'(x)\bigr)^{-1}}.\]

At a periodic point corresponding to $\gamma$ we have $(f^{n})'(x)=\xi_{\gamma}^{2}=e^{\ell(\gamma)}$ by \Cref{prop.induced} \textup{(2)},
so that $\bigl(1-e^{-\ell(\gamma)}\bigr)^{-1}=\sum_{k\geq 0}e^{-k\ell(\gamma)}$.
Substituting gives
\[\mathrm{tr}\,\mca{L}_{s,\varrho}^{\,n}
=\sum_{k\geq 0}\ \sum_{f^{n}x=x}\mathrm{tr}\,\varrho(\alpha^{(n)}(x))\,
e^{-(s+k)r^{(n)}(x)}.\]
A point with $f^{n}x=x$ lies on a periodic orbit whose least period $m$ divides $n$,
and each such orbit contains exactly $m$ such points, so that
\[\sum_{n\geq 1}\frac{1}{n}\sum_{f^{n}x=x}(\cdots)
=\sum_{\gamma}\sum_{j\geq 1}\frac{1}{j}(\cdots),\]
the outer sum on the right being over the primitive closed geodesics and the inner one over their powers.
Since $\sum_{j\geq 1}\frac1j\mathrm{tr}\,\varrho([\gamma])^{j}e^{-j\sigma\ell(\gamma)} =-\log\det\bigl(I-\varrho([\gamma])e^{-\sigma\ell(\gamma)}\bigr)$,
summing the series in $n$ gives the first identity,
and the second follows by taking the quotient of two consecutive factors.
\end{proof}

Every periodic point of the induced system corresponds to a closed geodesic, with no exception.
Indeed the two periodic points of \Cref{prop.modular} \textup{(4)} that are not realized are the two constant words,
and a constant word never contains the transition that defines $Y$,
so that it does not occur in the induced system.
The Euler product above is therefore exactly the product over the closed geodesics.

Only the dynamical $L$-function occurs in the Tauberian argument,
and \Cref{prop.fredholm} provides it on a half plane containing $\mathrm{Re}\,s\geq 1$.
We therefore do not need the identity of Mayer \cite[Theorem 2]{Mayer1991BAMS} between $\det(I-\mca{L}_{s})$ and the Selberg zeta function,
and we do not need the continuation of $\mca{L}_{s,\varrho}$ beyond the line $\mathrm{Re}\,s=1/2$.

\subsection{Non-vanishing on the line} \label{subsec.non-vanishing}

The induced systems that occur in \Cref{subsec.refine} are Markov shifts with finitely many states, and not full shifts.
We therefore fix the notation in that generality.
A full shift is the case of a single state.
Let $F$ denote the finite set of states, let $B_u\subset (0,1)$ denote the interval attached to $u\in F$, and let $X=\bigsqcup_{u\in F}B_u$.
The letters of the alphabet are the edges of the graph of the system,
so we use for them the vocabulary of \Cref{subsec.graph}.
A letter $v$ starts at one or more states $u$, and it arrives at a state $u_v$ that depends on $v$ alone;
its branch $S_v^{-1}$ maps $B_u$ into $B_{u_v}$.
Thus the arrows of this graph run backward in time:
a word $w=v_1\cdots v_n$, read in the order of the forward itinerary as in \Cref{subsec.transfer},
corresponds to the path $v_n,\dots,v_1$, along which the branches compose to $S_w^{-1}$.
The Markov property says that for every $u'\in F$ the sets $S_v^{-1}(B_u)$,
taken over the pairs $(u,v)$ with $v$ starting at $u$ and $u_v=u'$,
are mutually disjoint and cover $B_{u'}$ up to a set of measure zero.

Let $\Sigma_G$ denote the extension shift, whose states are the pairs $(u,g)\in F\times G$,
and in which a letter $v$ starting at $u$ leads from $(u,g)$ to $(u_v,\alpha(v)g)$.
Hence the path of a word $w=v_1\cdots v_n$ leads from $(u,g)$ to a state with second coordinate $\alpha(w)g$,
where $\alpha(w)=\alpha(v_1)\cdots\alpha(v_n)$ is the product in the order of the forward itinerary.
Let $\mca{L}_{s}^{\rm reg}$ denote the operator attached to the left regular representation of $G$, which acts on functions on $X\times G$ by
\[(\mca{L}^{\rm reg}_{s}W)(z,g)
=\sum_{v}e^{-s\,r_{v}(z)}\,W\bigl(S_{v}^{-1}z,\,\alpha(v)g\bigr),\]
the sum being over the letters $v$ that start at the state of $z$.
On a function of the form $W(z,g)=w(z)\varrho(g)\xi$,
with $w$ a row vector and $\xi$ a constant column vector,
the operator $\mca{L}^{\rm reg}_{s}$ acts as $\mca{L}_{s,\varrho}$ on $w$.
Since the regular representation contains every irreducible representation,
\[\mca{L}^{\rm reg}_{s}\cong \bigoplus_{\varrho\in \widehat{G}}
\mca{L}_{s,\varrho}^{\oplus \dim \varrho},\]
and it suffices to work with the scalar operator $\mca{L}^{\rm reg}_{s}$.

\begin{lem} \label{lem.non-vanishing}
Assume that the following two conditions hold.
\begin{itemize}
\item[\textup{(N1)}] The shift $\Sigma_G$ is irreducible.
\item[\textup{(N2)}] There are two closed geodesics $\gamma_1$ and
$\gamma_2$ with $\ell(\gamma_1)/\ell(\gamma_2)\notin \Q$.
\end{itemize}
Let $s$ satisfy $\mathrm{Re}\,s=1$, let $\varrho$ be an irreducible unitary representation of $G$,
and suppose that $1$ is an eigenvalue of $\mca{L}_{s,\varrho}$.
Then $s=1$, the representation $\varrho$ is trivial,
and the eigenfunction is a constant multiple of the positive fixed point $m$ of $\mca{L}_{1,\mathbf 1}$.
\end{lem}

\begin{proof}
By the decomposition above it is enough to treat the regular representation.
So let $s=1+it$ and let $W$ be a nonzero solution of $\mca{L}^{\rm reg}_{s}W=W$.

Since $|e^{-itr_{v}}|=1$, the triangle inequality gives
\begin{equation} \label{eq.domination}
|W|=\bigl|\mca{L}^{\rm reg}_{s}W\bigr|
\leq \mca{L}^{\rm reg}_{1}|W| \qquad \text{on } X\times G.
\end{equation}
Let $\mu$ denote the product of Lebesgue measure on $X$ with the counting measure on $G$.
Because $r=\log|f'|$, the change of variables formula gives $\int_{B_u}e^{-r_v(z)}V(S_v^{-1} z)\,dz=\int_{S_v^{-1}(B_u)}V(y)\,dy$ for every state $u$ and every letter $v$ starting at $u$.
By the Markov property, the images $S_v^{-1}(B_u)$ tile $X$ when the pair $(u,v)$ runs over the pairs with $v$ starting at $u$,
and the substitution $g\mapsto \alpha(v)g$ permutes $G$.
Hence $\int \mca{L}^{\rm reg}_{1}V\,d\mu=\int V\,d\mu$ for every $V\in L^{1}(\mu)$.
Applying this to \eqref{eq.domination} yields $\int |W|\,d\mu\leq \int |W|\,d\mu$,
so that equality holds in \eqref{eq.domination} at $\mu$-almost every point, hence everywhere by continuity.

Next, $|W|$ vanishes nowhere.
Indeed, suppose that $|W|(z_0,g_0)=0$.
Equality in \eqref{eq.domination} and the positivity of the weights force $|W|$ to vanish at $(S_{v}^{-1}z_0,\alpha(v)g_0)$ for every letter $v$ that starts at the state of $z_0$,
and hence, by iteration, at $(S_{w}^{-1}z_0,\alpha(w)g_0)$ for every word $w$ whose last letter starts at the state of $z_0$.
Let $g\in G$, and let $[v]$ be a cylinder.
By (N1) there is a word $w$ which begins with $v$,
whose last letter starts at the state of $z_0$, and which satisfies $\alpha(w)=gg_0^{-1}$.
The point $S_{w}^{-1}(z_0)$ then lies in $[v]$.
Hence the zero set of $|W|$ is dense in $X$ for each fixed $g$, and continuity gives $W\equiv 0$.
The same argument applies to every non-negative solution of $V=\mca{L}^{\rm reg}_{1}V$ that is continuous on $\ol{X}\times G$:
the equation extends to the closed cells, since $S_{v}^{-1}(\ol{B_u})\subset \ol{B_{u_v}}$ holds
and both sides are continuous, and a zero again propagates to a dense set.

Write $W=|W|\,\omega$ with $|\omega|=1$.
Equality in \eqref{eq.domination} means that, for the letters $v$ starting at the state of $z$,
the values $W(S_{v}^{-1}z,\,\alpha(v)g)$ multiplied by the weights $e^{-s\,r_{v}(z)}$ are non-negative multiples of a common unimodular number.
After division by $|W|(z,g)$ that becomes
\begin{equation} \label{eq.cohom}
e^{-it\,r_{v}(z)}\,\omega\bigl(S_{v}^{-1}z,\,\alpha(v)g\bigr)=\omega(z,g),
\end{equation}
for every $z\in X$, every $g\in G$, and every letter $v$ starting at the state of $z$.

Suppose first that $t\neq 0$.
Let $\gamma$ be a closed geodesic, let $w$ be its word, let $n$ denote the length of $w$,
let $z_w$ denote the fixed point of $S_{w}^{-1}$, and let $k$ be the order of $\alpha(w)$ in $G$.
Iterating \eqref{eq.cohom} along $w$ gives $e^{-it\ell(\gamma)}\omega(z_w,\alpha(w)g)=\omega(z_w,g)$,
and iterating that identity $k$ times gives $e^{-itk\ell(\gamma)}\omega(z_w,g)=\omega(z_w,g)$.
Since $\omega$ does not vanish, we obtain $k\,\ell(\gamma)\in (2\pi/t)\Z$.
Applying that to the two geodesics of (N2), with orders $k_1$ and $k_2$,
we get $k_i\ell(\gamma_i)=2\pi n_i/t$ with $n_i\in \Z$,
and hence $\ell(\gamma_1)/\ell(\gamma_2)=k_2n_1/(k_1n_2)\in \Q$, which contradicts (N2).

Suppose now that $t=0$.
Then \eqref{eq.cohom} states that $\omega(S_{v}^{-1}z,\alpha(v)g)=\omega(z,g)$ for all $v$, $z$,
and $g$, and the points $(S_{w}^{-1}z,\alpha(w)g)$ are dense in $X\times G$,
as in the argument on the vanishing of $|W|$; hence $\omega$ is constant by continuity,
and $|W|$ is a positive continuous fixed point of $\mca{L}^{\rm reg}_{1}$.
Such a fixed point is unique up to a positive scalar.
Indeed, let $h_1$ and $h_2$ be two of them.
They are continuous on the compact set $\ol{X}\times G$,
because $\ol{X}\subset [0,1]\subset \Omega$ and the functions in question are continuous on $\Omega\times G$,
and they are positive there, because a zero would propagate and force the function to vanish identically.
Put $\lambda=\min_{\ol{X}\times G}h_1/h_2$.
Then $h_1-\lambda h_2$ is a non-negative continuous fixed point of $\mca{L}^{\rm reg}_{1}$ that vanishes at a minimizer,
so it vanishes identically by the same propagation, and $h_1=\lambda h_2$.
For $h\in G$, the translate $(z,g)\mapsto |W|(z,gh)$ is again a positive fixed point,
since the right translations commute with $\mca{L}^{\rm reg}_{1}$,
and hence it equals $c(h)\,|W|$ with $c(h)>0$;
the map $c:G\to \R_{>0}$ is a homomorphism from a finite group,
hence trivial, and $|W|$ is a function of $z$ alone.
Therefore $W$ is a constant multiple of a positive function of $z$ alone.
For a solution $w$ of $\mca{L}_{s,\varrho}w=w$ and a column vector $\xi$, this applies to $W(z,g)=w(z)\varrho(g)\xi$ and shows that $w(z)\varrho(g)=w(z)$ holds for every $g\in G$;
since $w\neq 0$ and $\varrho$ is irreducible, the representation $\varrho$ is trivial,
and the eigenfunction descends to the positive fixed point $m$ of $\mca{L}_{1,\mathbf 1}$,
which is unique up to a positive scalar by the same argument.
\end{proof}

Both assumptions hold in our situation.
For (N1), the base is the full shift on $A$, whose graph has a single state $u$,
and a letter $v\in A$ leads from $(u,g)$ to $(u,\alpha(v)g)$.
Hence the set of elements of $G$ reachable from $(u,g)$ in the second coordinate is the semigroup generated by $\alpha(A)$ times $g$.
That semigroup is $G$, because $\alpha(A)$ generates $G$ as a group by \Cref{lem.surj} and \Cref{lem.dense},
and a generating set of a finite group generates it as a semigroup.
For (N2), take $\gamma_1=LR$ and $\gamma_2=LR^{\,2}$, so that
\[\ell(\gamma_1)=2\log \frac{3+\sqrt5}{2},\qquad
\ell(\gamma_2)=2\log \bigl(2+\sqrt3\bigr).\]
The first number is a unit of $\Q(\sqrt5)$ and the second is a unit of $\Q(\sqrt3)$,
so they are multiplicatively independent and the ratio of the two lengths is irrational.

\begin{rem} \label{rem.numerics}
The conclusion of \Cref{lem.non-vanishing} can be inspected numerically.
For $G=\Z/m\Z$ and the character $\chi(\Psi)=e^{2\pi i\theta \Psi}$ with $\theta=k/m$,
the leading eigenvalue of $\mca{L}_{1,\chi}$ is $1$ for $\theta=0$ and is bounded by $0.14$ for $m=5$ and $k\neq 0$.
Moreover the solution $s(\theta)$ of $\lambda(\mca{L}_{s,\chi})=1$ agrees with $1-6\theta$ to five decimal places for $0<\theta\leq 1/12$.
On the spectral side, the corresponding value of $s$ is precisely $1-6\theta$:
the lowest eigenvalue of the Laplacian of weight $12\theta$ with the multiplier system attached to $\chi$
is $6\theta(1-6\theta)$ for $0\leq \theta\leq 1/12$ \cite[Section 2]{Sarnak2010CMA},
and it equals $s(1-s)$ for $s=1-6\theta$.
The value $6$ equals $2\pi$ times the scale $3/\pi$ of the Cauchy law of Sarnak \cite{Sarnak2010CMA} and Mozzochi \cite{Mozzochi2013Israel},
because the characteristic function of the Cauchy law of scale $c$ is $e^{-c|u|}$, and $u=2\pi\theta$.
The dependence on $\theta$ is linear because of the cusp, whereas a compact quotient would produce a quadratic dependence.
\end{rem}

\subsection{The Tauberian step} \label{subsec.tauberian}

We first determine where the $L$-functions are holomorphic and nonzero.

\begin{lem} \label{lem.hol}
Let $\varrho$ be an irreducible unitary representation of $G$.
Then $\det(I-\mca{L}_{s,\varrho})$ has no zero in $\{\mathrm{Re}\,s\geq 1\}$,
except for a simple zero at $s=1$ when $\varrho$ is trivial.
Consequently $L(s,\varrho)$ is holomorphic and nonzero on $\{\mathrm{Re}\,s\geq 1\}$ when $\varrho$ is nontrivial,
and $L(s,\mathbf{1})$ is holomorphic and nonzero there except for a simple pole at $s=1$.
\end{lem}

\begin{proof}
We first treat $\sigma=\mathrm{Re}\,s>1$.
Every branch of the induced map satisfies $\bigl|(\gamma_{k,l}^{-1})'(z)\bigr|=(kz+kl+1)^{-2}\leq (kl+1)^{-2}\leq 1/4$ for every $z\in [0,1]$ and every $(k,l)\in A$,
and every branch $S_w^{-1}$ of the $n$-th iterate is a M\"obius map $z\mapsto (pz+q)/(p'z+q')$ with non-negative integral entries, $p\leq q$, and $p'\leq q'$.
The last two inequalities hold for a single branch, by \Cref{prop.induced} \textup{(2)},
and they are preserved under multiplication, because the entries are non-negative.
Hence
\[\frac{\sup_{[0,1]}\bigl|(S_w^{-1})'\bigr|}{\inf_{[0,1]}\bigl|(S_w^{-1})'\bigr|}
=\Bigl(\frac{p'+q'}{q'}\Bigr)^{2}\leq 4 .\]
Let $\delta_0$ denote the smallest length of a cell, and let $u$ be a state.
The sets $S_w^{-1}(B_u)$, for the directed paths $w$ of length $n$ starting at $u$, are mutually disjoint subsets of $X$,
and the distortion bound gives $\sup_{B_u}\bigl|(S_w^{-1})'\bigr|\leq 4\,|S_w^{-1}(B_u)|/|B_u|$.
Summing over those words yields $\sum_{w}\sup_{B_u}\bigl|(S_w^{-1})'\bigr|\leq 4|X|/\delta_0$,
where $|X|$ denotes the total length of $X$, and therefore
\[\bigl(\mca{L}_{\sigma,\mathbf 1}^{n}\mathbf 1\bigr)(z)
=\sum_{w}\bigl|(S_w^{-1})'(z)\bigr|^{\sigma}
\leq 4^{-(\sigma-1)n}\sum_{w}\bigl|(S_w^{-1})'(z)\bigr|
\leq \frac{4|X|}{\delta_0}\,4^{-(\sigma-1)n}.\]
If $\mca{L}_{s,\varrho}W=\nu W$ with $W$ holomorphic on $\Omega$,
then restricting $W$ to $X$ and using $|\nu|^{n}|W|\leq \mca{L}^{n}_{\sigma,\mathbf 1}|W|$ gives $|\nu|\leq 4^{-(\sigma-1)}<1$.
Hence $\det(I-\mca{L}_{s,\varrho})$ has no zero in $\mathrm{Re}\,s>1$.

On the line $\mathrm{Re}\,s=1$, \Cref{lem.non-vanishing} shows that $1$ is not an eigenvalue,
except at $s=1$ for the trivial representation,
where the eigenvalue $1$ is geometrically simple with eigenfunction $m$.
That eigenvalue is also algebraically simple.
Indeed, if $\mca{L}_{1,\mathbf 1}W=W+m$, then integrating over $X$ against Lebesgue measure,
which is invariant by $\mca{L}_{1,\mathbf 1}$ as in the proof of \Cref{lem.non-vanishing},
gives $\int m=0$, and that contradicts $m>0$.
The eigenvalue is therefore isolated and simple,
so $\mca{L}_{s,\mathbf 1}$ has an eigenvalue $\lambda(s)$ that depends analytically on $s$ near $s=1$ and satisfies $\lambda(1)=1$.
Analytic perturbation gives
\[\lambda'(1)=\frac{\int \partial_s\mca{L}_{s,\mathbf 1}m\big|_{s=1}\,dz}
{\int m\,dz}=-\frac{\int r\,m\,dz}{\int m\,dz}<0,\]
because $r>2\log 2$ holds by \Cref{prop.induced} \textup{(3)} and $m$ is positive.
For the full shift, the probability measure proportional to $m\,dz$ is the Gauss measure $\mu$, and $\int r\,d\mu=2\pi^{2}/(6\log 2)$,
by $\int_0^{1}\frac{-\log y}{1+y}\,dy=\pi^{2}/12$ and the invariance of $\mu$ under the Gauss map.
Hence $\det(I-\mca{L}_{s,\mathbf 1})=\prod_j\bigl(1-\lambda_j(s)\bigr)$ has a simple zero at $s=1$.
The statement about $L(s,\varrho)$ follows from \Cref{prop.fredholm},
since $\det(I-\mca{L}_{s+1,\varrho})$ has no zero in $\mathrm{Re}\,s\geq 1$.
\end{proof}

\begin{thm} \label{thm.arith}
Let $G$ be a finite group and let $\rho:\pi_1(S^3-K_{2,3})\surj G$ be a surjective homomorphism.
Then for every conjugacy class $C\subset G$,
\[\lim_{\nu\to \infty}
\frac{\#\{\gamma \mid |\tr \gamma|\leq \nu,\ \rho([C_\gamma])=C\}}
{\#\{\gamma\mid |\tr\gamma|\leq \nu\}}=\frac{\#C}{\#G},\]
where $\gamma$ runs through the primitive hyperbolic conjugacy classes of $\PSL_2\Z$.
\end{thm}

\begin{proof}
For a conjugacy class $C\subset G$ put
\[\vartheta_C(T)=\sum_{\substack{\gamma,\ k\geq 1\\ k\ell(\gamma)\leq T,\
\rho([C_\gamma])^{k}\in C}}\ell(\gamma),\]
which is non-decreasing in $T$.
We use two identities.
The orthogonality of the columns of the character table is the equality $\sum_{\varrho}\overline{\chi_\varrho(C)}\,\chi_\varrho(g) =\bigl(\#G/\#C\bigr)\mathbf 1_{C}(g)$,
the sum being over the irreducible representations.
The logarithmic derivative of the $L$-function is $-L'/L(s,\varrho)=\sum_{\gamma}\sum_{k\geq 1}\ell(\gamma)\, \chi_\varrho\bigl(\rho([C_\gamma])^{k}\bigr)e^{-sk\ell(\gamma)}$.
Combining them, we obtain
\[\int_0^{\infty}e^{-sT}\,d\vartheta_C(T)
=\frac{\#C}{\#G}\sum_{\varrho}\overline{\chi_\varrho(C)}\;
\Bigl(-\frac{L'}{L}(s,\varrho)\Bigr),\]
where the Frobenius classes are computed by \Cref{lem.homotopic} and \Cref{prop.modular} (4);
the assumption $\gamma\cap \mca{R}\subset \bigcup_{i,j}\mca{R}_{ij}$ of \Cref{lem.homotopic} holds here for every closed orbit,
since a crossing at one of the heights $1/2$, $1$, or $2$ sends the orbit into $\ell^{+}$,
so that a closed orbit avoids the three singular arcs.

By \Cref{lem.hol}, the right hand side is holomorphic on $\{\mathrm{Re}\,s\geq 1\}$ except for a simple pole at $s=1$ with residue $\#C/\#G$.
The Wiener--Ikehara theorem, in the form for non-decreasing functions \cite[Appendix I, (I.8)]{ParryPollicott1990},
therefore gives $\vartheta_C(T)\sim (\#C/\#G)\,e^{T}$.
The terms with $k\geq 2$ contribute at most $\sum_{k\geq 2}\vartheta(T/k)$,
where $\vartheta(T)=\vartheta_{G}(T)\sim e^{T}$;
the number of nonzero terms is $O(T)$ because $\ell(\gamma)\geq 2\log\frac{3+\sqrt5}{2}$,
and each is $O(e^{T/2})$, so that the total is $O(Te^{T/2})=o(e^{T})$.
Writing $\vartheta^{(1)}_C$ for the part with $k=1$,
we therefore have $\vartheta^{(1)}_C(T)\sim (\#C/\#G)e^{T}$, and
\[\#\{\gamma\mid \ell(\gamma)\leq T,\ \rho([C_\gamma])=C\}
=\int_0^{T}\frac{d\vartheta^{(1)}_C(u)}{u}
=\frac{\vartheta^{(1)}_C(T)}{T}+\int_0^{T}\frac{\vartheta^{(1)}_C(u)}{u^{2}}
\,du,\]
in which the second term is $O(e^{T}/T^{2})$, so that
\[\#\{\gamma\mid \ell(\gamma)\leq T,\ \rho([C_\gamma])=C\}
\sim \frac{\#C}{\#G}\cdot \frac{e^{T}}{T}.\]

Finally $\ell(\gamma)=2\log \xi_\gamma$ is an increasing function of $|\tr \gamma|$.
Summing over $C$ gives $\#\{\gamma\mid \ell(\gamma)\leq T\}\sim e^{T}/T$,
so that the classes of a given trace are negligible in number by the argument of \Cref{lem.ties}.
Hence the density statement holds for every ordering by $|\tr \gamma|$,
with an arbitrary order among the classes of equal trace.
\end{proof}

\begin{cor}[{\Cref{corA.linking} for $K=K_{2,3}$ and the trace ordering}] \label{cor.rademacher}
Let $\gamma$ run through the primitive hyperbolic conjugacy classes of $\PSL_2\Z$.
For every $m\in \Z_{\geq 1}$ and every $k\in \Z/m\Z$,
\[\lim_{T\to\infty}
\frac{\#\{\gamma\mid \ell(\gamma)\leq T,\ \Psi(\gamma)\equiv k
\ (\mathrm{mod}\ m)\}}{\#\{\gamma\mid \ell(\gamma)\leq T\}}=\frac1m.\]
\end{cor}

\begin{proof}
Apply \Cref{thm.arith} to the composition of the abelianization $\pi_1(S^3-K_{2,3})\surj \Z$ with the reduction modulo $m$,
and use the theorem of Ghys that ${\rm lk}(C_\gamma,K_{2,3})=\Psi(\gamma)$.
\end{proof}

\subsection{The trace ordering after removing knots} \label{subsec.arith.removing}

\Cref{thm.arith} treats the case $L_n=\{K_{2,3}\}$, in which no modular knot has been removed.
For a general $L_n$ we use the refinement $\mca{P}$ of \Cref{subsec.refine},
and we check that it does not affect the estimates of \Cref{lem.nuclear}.
In the trace ordering, the knots that go deep into the cusp occur early,
and this produces the Cauchy law of \Cref{thm.cauchy}.

One choice in the construction should be made explicit.
A return word to $Y$ now joins two rectangles of $\mca{P}$ that may differ, so that it is not a loop in the graph $\mca{G}'$ of $\mca{P}$.
Hence $\alpha$ is defined only after a choice:
for each rectangle $u$ of $\mca{P}$ contained in $Y$,
fix a path $\theta_u$ in $\mca{G}'$ from a base vertex to $u$,
and put $\alpha(c)=\rho\bigl(\iota_{*}(\theta_u c\,\theta_{u'}^{-1})\bigr)$ for a return word $c$ from $u$ to $u'$.
Different choices change $\alpha$ by a coboundary with values in $G$ and do not change the Frobenius class of a closed orbit, since a closed orbit gives a loop.
With that convention the values of $\alpha$ on the directed loops at the base vertex generate $G$,
by \Cref{lem.surj} and \Cref{lem.dense}, and hence the extension shift is irreducible, which is condition (N1).

The induced system is now a Markov shift over a countable alphabet,
and no longer a full shift, so the transfer operator acts on the finite direct sum
\[\bigoplus_{u}\mca{H}(\Omega),\]
indexed by the rectangles $u$ of $\mca{P}$ contained in $Y$, by
\[(\mca{L}_{s,\varrho}W)_{u}(z)=\sum_{v}e^{-s\,r_{k,l}(z)}\,
W_{u'(v)}(\gamma_{k,l}^{-1}z)\,\varrho(\alpha(v)),\]
the sum being over the letters $v$ that lead from $u$ to $u'(v)$, with $(k,l)$ the pair that lies under $v$.
Here $\mca{H}(\Omega)$ denotes the space of bounded holomorphic functions on $\Omega$ with values in $\C^{d}$, as in the proof of \Cref{lem.nuclear}.
Each entry of this finite matrix of operators satisfies the estimates of \Cref{lem.nuclear}, because the branches are unchanged and,
by \Cref{prop.refine} \textup{(3)}, the multiplicity of each branch is bounded by a constant depending only on $L'$.
Hence $\mca{L}_{s,\varrho}$ is again nuclear of order zero for $\mathrm{Re}\,s>1/2$.
The trace of the $n$-th power is the sum of the traces of the diagonal entries,
and each of those is a sum over the closed paths of length $n$ in the graph of the induced system,
so that the trace formula of \Cref{prop.fredholm} holds with the sum taken over the directed loops.

The closed geodesics not contained in $L'$ have a well defined $\mca{P}$-itinerary,
because a closed orbit meeting a cutting arc would be forward or backward asymptotic to $L'$,
and two distinct closed orbits are never asymptotic.
Conversely, a directed loop of the induced graph is not necessarily realized by a closed geodesic.
By \Cref{lem.refine.realized}, the primitive loops that are not realized are finite in number,
and their contribution to the Euler product is a finite product of factors $\det\bigl(I-\varrho(g)\,e^{-s\ell}\bigr)^{\pm 1}$ with $\ell>0$,
each of which is holomorphic and non-vanishing on $\mathrm{Re}\,s\geq 1$.
Hence the identity of \Cref{prop.fredholm} holds after multiplying one side by finitely many such factors,
and the conclusions of \Cref{lem.hol} are unchanged.

The number $1$ is an eigenvalue of $\mca{L}_{1,\mathbf 1}$ for the refined system as well:
by the change of variables in the proof of \Cref{lem.non-vanishing}, the functional $W\mapsto \sum_{u}\int_{B_u}W_u(z)\,dz$ is fixed by the dual operator, and $\mca{L}_{1,\mathbf 1}$ is compact.
Finally \Cref{lem.non-vanishing} applies, since it is stated for a Markov shift with finitely many states.
Its assumption (N1) holds because the refined graph is irreducible by \Cref{lem.flow.irreducible},
and because the values of $\alpha$ on the directed loops generate $G$, as recalled above.
Its assumption (N2) holds for two closed geodesics that are not contained in $L'$:
the classes $L^{q}R$ with $q\in\Z_{\geq 1}$ have trace $q+2$, infinitely many of the fields $\Q\bigl(\sqrt{(q+2)^{2}-4}\bigr)$ are mutually distinct,
and two units greater than $1$ of two distinct real quadratic fields are multiplicatively independent, as for the two closed geodesics of \Cref{subsec.non-vanishing}.

\begin{thm}[{\Cref{thmA.arith}}] \label{thm.arith.general}
Let $(K_i)_{i\in \Z_{\geq 1}}$ denote the sequence of the modular knots ordered by $|\tr \gamma|$,
with $K_{2,3}$ inserted at an arbitrary fixed position.
Then $(K_i)_i$ obeys the Chebotarev law in $S^{3}$.
\end{thm}

\begin{proof}
Let $n\in \Z_{\geq 1}$, let $\rho:\pi_1(S^3-L_n)\surj G$ be a surjection onto a finite group, and let $C\subset G$ be a conjugacy class.
The ambient manifold is treated as in the proof of \Cref{thm.modular}, whether or not $L_n$ contains $K_{2,3}$.
Refine the Markov section as in \Cref{prop.refine}, and let $\mca{G}'$ denote the graph of the refined section,
so that $L_n\cap \mca{R}\subset \partial \mca{R}$ and $\iota$ maps $\mca{G}'$ into $\mca{U}$, which is disjoint from $L_n$.
By \Cref{lem.surj} the composition $\rho\circ \iota_{*}$ is surjective.
The finitely many closed orbits whose crossings meet $\mca{R}-\bigcup_{i,j}\mca{R}_{ij}$ are removed, as in the proof of \Cref{thm.main.general},
and by \Cref{lem.homotopic} the Frobenius class of a symbolic closed orbit agrees with $\rho([C_\gamma])$.
The argument of \Cref{thm.arith} now applies,
with the operators described above in place of those of \Cref{subsec.transfer},
and with \Cref{lem.nuclear} and \Cref{lem.hol} in their form for a Markov shift with finitely many states.
\end{proof}

\begin{rem} \label{rem.burrin.vonessen}
\Cref{cor.rademacher} is the case $A=k+m\Z$ of a theorem of Burrin and von Essen \cite[Theorem 1.5]{BurrinVonEssen2024IMRN}:
for the modular group,
and more generally for the cofinite Fuchsian groups of genus zero and for the congruence subgroups,
the closed geodesics whose winding number around the cusp lies in a set $A\subset\Z$ have density equal to the natural density of $A$, for the ordering by length.
Their proof uses a prime geodesic theorem twisted by a multiplier system of real weight, obtained from the Selberg trace formula;
the statement can also be extracted from Mozzochi \cite[Theorems 1.2 and 1.5]{Mozzochi2013Israel}.
The estimate for a fixed character $e^{2\pi i\theta\Psi}$ is stated by Sarnak \cite{Sarnak2010CMA} and proved by Mozzochi \cite[Section 5]{Mozzochi2013Israel};
on the dynamical side it corresponds to \Cref{lem.non-vanishing},
and the abscissa $1-6|\theta|$ is the one observed in \Cref{rem.numerics}.
Kelmer \cite{Kelmer2012JMD} used that estimate to prove that the modular knots with a given linking number with $K_{2,3}$ are equidistributed in $T^{1}(\PSL_2\Z\backslash\bb{H}^{2})$,
and asked for a purely dynamical argument \cite[Remark 1.4]{Kelmer2012JMD}.
\Cref{lem.non-vanishing} provides one for the residues of the linking number modulo $m$, and for every finite group $G$,
and the present subsection treats the removal of modular knots, to which those proofs do not apply, at least not directly.
\end{rem}

\begin{rem} \label{rem.arith.sources}
On the dynamical side, the scheme of this section is that of Pollicott \cite{Pollicott1986BSMF} for the trivial representation, on the disks of Mayer \cite{Mayer1976BSMF};
see also \cite{DalboPeigne1996Ast}.
\Cref{lem.nuclear} \textup{(3)} and \Cref{prop.fredholm} are the case of \cite[Theorem 2]{Fried1996Invent} for $(p,q)=(2,3)$ and the subgroup $\PSL_2\Z$ of index two of $\mathrm{PGL}_2\Z$,
with the constant blocks $\varrho(h_{11})$, $\varrho(h_{12})$, $\varrho(h_{22})$, and $\varrho(h_{21})$ in the notation of \Cref{prop.induced} \textup{(4)},
and Fried obtains moreover the meromorphic continuation to $\C$.
The proof of \Cref{lem.non-vanishing} follows the standard pattern for complex transfer operators \cite[Theorem 2]{Pollicott1984ETDS}, \cite[Theorem 4.5]{ParryPollicott1990},
which was carried out over a countable alphabet in \cite[Lemme VII.10]{DalboPeigne1996AIF} and in \cite[Proposition 2.2]{BabillotPeigne2000ASENS}.
For $s=1$, the argument on the product of the interval with a finite set is that of \cite[Proposition 5.3]{GuivarchLeJan1993ASENS}.
Pollicott remarked, at the end of \cite[Section 3]{Pollicott1986BSMF}, that it would be natural to treat also the surfaces attached to the normal subgroups of finite index of the modular group,
and that his approach did not readily adapt to them; \Cref{thm.finiteindex} concerns those surfaces.
\end{rem}

\section{Linking numbers} \label{sec.linking}

We apply \Cref{thm.modular}, \Cref{thm.combinatorial}, and \Cref{thm.arith.general} to cyclic quotients of the groups $\pi_1(S^{3}-L)$, where $L$ is a finite union of modular knots and of $K_{2,3}$.

\begin{cor}[{\Cref{corA.linking}}] \label{cor.linking}
Let $K$ be either $K_{2,3}$ or a modular knot $C_\sigma$, let $m\in \Z_{\geq 1}$,
and let $\lambda$ denote the length function $\ell_h$ of \Cref{thm.modular},
the hyperbolic length $\ell$, or the combinatorial length $\ell_{\rm c}$.
As $\gamma$ runs through the primitive hyperbolic conjugacy classes of $\PSL_2\Z$ with $C_\gamma\neq K$,
the residues of $\operatorname{lk}(C_\gamma,K)$ modulo $m$ are equidistributed for the ordering by $\lambda$,
along the thresholds $\lambda\leq T$, except in the case $\lambda=\ell_{\rm c}$, $K=K_{2,3}$, and $m$ even.
\end{cor}

\begin{proof}
Put $L=K$, so that $H_1(S^{3}-L)\cong\Z$ is generated by the meridian of $K$,
and let $\rho_m:\pi_1(S^{3}-L)\surj \Z/m\Z$ denote the composition of the abelianization with the reduction modulo $m$.
Then $\rho_m([C_\gamma])$ is the residue of $\operatorname{lk}(C_\gamma,K)$.
For $\lambda=\ell_h$ apply \Cref{thm.modular},
and for $\lambda=\ell$ apply \Cref{thm.arith.general}, or \Cref{thm.arith} when $K=K_{2,3}$.
For $\lambda=\ell_{\rm c}$ apply \Cref{thm.combinatorial}.
A homomorphism $\chi:\Z/m\Z\to\{\pm1\}$ with $\chi(\rho_m([C_\gamma]))=(-1)^{\operatorname{lk}(C_\gamma,K_{2,3})}$ for all but finitely many $\gamma$ can exist only if $K=K_{2,3}$,
and then $\chi(k)=(-1)^{k}$ on $\Z/m\Z$, which is a homomorphism if and only if $m$ is even.
\end{proof}

\begin{rem} \label{rem.linking.joint}
The same argument with $L=K_{2,3}\cup C_{\sigma_1}\cup\dots\cup C_{\sigma_n}$ and $G=(\Z/m\Z)^{n+1}$ shows that the vector
\[\bigl(\operatorname{lk}(C_\gamma,K_{2,3}),\
\operatorname{lk}(C_\gamma,C_{\sigma_1}),\dots,
\operatorname{lk}(C_\gamma,C_{\sigma_n})\bigr)\]
is equidistributed in $(\Z/m\Z)^{n+1}$ for the ordering by $\ell_h$ and for the ordering by $\ell$,
and for the ordering by $\ell_{\rm c}$ when $m$ is odd,
or for every $m$ after the first coordinate is dropped.
In particular the linking numbers of $C_\gamma$ with distinct modular knots are independent modulo $m$, in the sense that their joint distribution is the product of the uniform distributions.
By \Cref{thm.ghys} and by \eqref{eq.simon}, the coordinates of the vector are the values at the cyclic word of $\gamma$ of the explicit functions $\Psi,F_{\sigma_1},\dots,F_{\sigma_n}$,
where $F_\sigma(\gamma)=\operatorname{lk}(C_\gamma,C_\sigma)$ is given by \eqref{eq.simon}.
\end{rem}

Two functions of $\gamma$ in the literature are linear combinations of the coordinates in \Cref{rem.linking.joint},
and we record what \Cref{cor.linking} says about them.

\begin{cor} \label{cor.two.symbols}
Let $\sigma$ be a primitive hyperbolic class of $\PSL_2\Z$ with $C_\sigma\neq C_{\sigma^{-1}}$,
let $m\in \Z_{\geq 1}$, and let $\lambda$ be $\ell_h$, $\ell$, or $\ell_{\rm c}$ as in \Cref{cor.linking}.
Let $\gamma$ run through the primitive hyperbolic classes with $C_\gamma\neq C_\sigma$ and $C_\gamma\neq C_{\sigma^{-1}}$.
Then the residues modulo $m$ of each of the following two functions of $\gamma$ are equidistributed for the ordering by $\lambda$, without exception.
\begin{enumerate}
\item
$\operatorname{lk}(C_\gamma,C_\sigma)+\operatorname{lk}(C_\gamma,C_{\sigma^{-1}})$,
which equals $-\frac12\Psi_{C_\sigma}(\gamma)$,
where $\Psi_{C_\sigma}$ is the hyperbolic Rademacher symbol of Duke, Imamo\u{g}lu,
and T\'oth \cite[Theorem 3]{DukeImamogluToth2017Duke};
the sign convention for linking numbers in \cite{DukeImamogluToth2017Duke} is opposite to ours,
see Rickards \cite[Remark 2.5]{Rickards2022linking-arXiv}.
See also Matsusaka \cite{Matsusaka2024MathAnn}.
\item
$\operatorname{lk}(C_\gamma,C_\sigma)-\operatorname{lk}(C_{\gamma^{-1}},C_\sigma)$,
which is the pseudocharacter attached to $C_\sigma$ by Simon \cite[Theorem 7.6]{Simon2025GT}.
\end{enumerate}
Both functions are recovered from the modular integral $F^{\rm lk}_\sigma$ of weight $2$ that Matsusaka \cite[Theorems 1.1 and 1.2]{Matsusaka2026individual-arXiv} attaches to $C_\sigma$:
its homogenized cycle integrals along the closed geodesics of $\gamma$
and of $\gamma^{-1}$ are the linking numbers $\operatorname{lk}(C_\gamma,C_\sigma)$ and $\operatorname{lk}(C_{\gamma^{-1}},C_\sigma)$ themselves,
in the sign convention of \cite{DukeImamogluToth2017Duke};
see Theorem 1.2 there, with the roles of $\gamma$ and $\sigma$ interchanged.
\end{cor}

\begin{proof}
The involution $v\mapsto -v$ of $T^{1}(\PSL_2\Z\backslash\bb{H}^{2})$ is isotopic to the identity through the fiberwise rotations,
which extend to $S^{3}$ as the circle action of the Seifert fibration with $K_{2,3}$ as a regular fiber, and hence it preserves linking numbers.
It maps the closed orbit $C_\gamma$ to $-C_{\gamma^{-1}}$ for every $\gamma$,
so that $\operatorname{lk}(C_{\gamma^{-1}},C_\sigma)=\operatorname{lk}(C_\gamma,C_{\sigma^{-1}})$
and $\operatorname{lk}(C_{\gamma^{-1}},C_\sigma+C_{\sigma^{-1}}) =\operatorname{lk}(C_\gamma,C_\sigma+C_{\sigma^{-1}})$.
By the first identity, the function in \textup{(2)} equals $\operatorname{lk}(C_\gamma,C_\sigma)-\operatorname{lk}(C_\gamma,C_{\sigma^{-1}})$.
Put $L=C_\sigma\cup C_{\sigma^{-1}}$, so that $H_1(S^{3}-L)\cong \Z^{2}$ is generated by the two meridians.
The two functions are the values at $[C_\gamma]$ of the homomorphisms $H_1(S^{3}-L)\to \Z$ that send the meridians to $(1,1)$ and to $(1,-1)$, each of which is surjective,
and their residues are the values of the reductions modulo $m$.
The equidistribution follows as in the proof of \Cref{cor.linking};
the exceptional character of \Cref{thm.combinatorial} does not occur, since $L$ does not contain $K_{2,3}$.
It remains to identify the function in \textup{(1)}.
By \cite[Theorem 3]{DukeImamogluToth2017Duke}, $\Psi_{C_\sigma}(\gamma)$ is the linking number in $S^{3}-K_{2,3}$ of the null-homologous $1$-cycles $C_\gamma+C_{\gamma^{-1}}$ and $C_\sigma+C_{\sigma^{-1}}$,
both oriented by the flow, in the convention of that paper.
By the second identity, that linking number is $2\operatorname{lk}(C_\gamma,C_\sigma+C_{\sigma^{-1}})$ in the convention of \cite{DukeImamogluToth2017Duke}
and $-2\operatorname{lk}(C_\gamma,C_\sigma+C_{\sigma^{-1}})$ in ours.
\end{proof}

\begin{rem} \label{rem.two.symbols}
By \cite[Theorem 6.4]{DukeImamogluToth2017Duke}, and also by \cite[Remark 5.4]{Simon2025GT},
the function in \textup{(1)} is one half of the number of intersections of the closed geodesics of $\sigma$
and of $\gamma$ on the modular surface, counted with multiplicity.
The signed intersection number of two closed curves on the modular surface vanishes,
so the assertion for \textup{(1)} is not a statement about the homology of the surface.
The identification in \cite[Section 6]{DukeImamogluToth2017Duke} uses a Birkhoff section,
that is, a surface in the unit tangent bundle that is bounded by $C_\sigma\cup C_{\sigma^{-1}}$ and is transverse to the flow in its interior.
The function in \textup{(2)} is a pseudocharacter of $\PSL_2\Z$ and not a character.
Moreover, if $m\geq 2$, then neither of the two residues modulo $m$ is a class function of a finite quotient of $\pi_1(S^{3}-K_{2,3})$.
Indeed, put $L'=L\cup K_{2,3}$, let $\rho:\pi_1(S^{3}-L')\to \Z/m\Z$ be either residue,
regarded as a homomorphism of $\pi_1(S^{3}-L')$ through the inclusion $S^{3}-L'\subset S^{3}-L$,
let $\rho':\pi_1(S^{3}-L')\surj G'$ be a surjective homomorphism that factors through $\pi_1(S^{3}-K_{2,3})$,
and suppose that there is a class function $f$ on $G'$ with $\rho([C_\gamma])=f(\rho'([C_\gamma]))$ for all but finitely many $\gamma$.
Then \Cref{thm.modular}, applied to the image $H$ of $(\rho,\rho')$,
shows that every element $(h_1,h_2)$ of $H$ satisfies $h_1=f(h_2)$, and hence $\rho$ vanishes on $\ker\rho'$.
The kernel contains the meridian of $C_\sigma$,
since $\rho'$ factors through $\pi_1(S^{3}-K_{2,3})$, whereas $\rho$ sends that meridian to $1$.
Hence the trace formula twisted by a representation of $\pi_1(S^{3}-K_{2,3})$,
as in \Cref{rem.burrin.vonessen}, does not reach either function,
and the removal of $C_\sigma$ and of $C_{\sigma^{-1}}$ is essential.
\end{rem}

\part{Branched covers and further remarks}

\section{Stably Chebotarev links} \label{sec.stably.chebotarev}

\subsection{Compatible orderings} \label{subsec.branched.general}

Let $M$ be a closed 3-manifold, let $(K_i)_{i\in \Z_{\geq 1}}$ be a sequence of mutually disjoint knots in $M$, and put $\mca{K}=\bigcup_i K_i$.
Let $\varpi\colon M_\varpi\to M$ be a finite cover branched along a finite sublink $L\subset \mca{K}$.
A component $K'$ of $\varpi^{-1}(\mca{K})$ lies over a unique $K_i$,
and the \emph{degree} of $K'$ is the degree $k\in \Z_{\geq 1}$ of the covering $K'\to K_i$.
The \emph{norm} of $K'$ is $\mathrm{N}(K')=i^{k}$, in analogy with the norm of a prime ideal.
Let $(K'_\nu)_{\nu\in \Z_{\geq 1}}$ be an ordering of the components of $\varpi^{-1}(\mca{K})$, that is, a sequence in which every component appears exactly once.
For $\nu\in \Z_{\geq 1}$, let $J(\nu)\in \Z_{\geq 0}\cup\{\infty\}$ denote the supremum of the integers $j\geq 0$ such that every component of degree one over $K_1,\dots,K_j$ is among $K'_1,\dots,K'_\nu$.
The ordering is said to be \emph{compatible} if the following two conditions hold.
\begin{enumerate}
\item[\textup{(C1)}] The number of the components of degree one among $K'_1,\dots,K'_\nu$ that lie over some $K_i$ with $i>J(\nu)$ is $o(\nu)$ as $\nu\to\infty$.
\item[\textup{(C2)}] The number of components of degree at least two among
$K'_1,\dots,K'_\nu$ is $o(\nu)$ as $\nu\to\infty$.
\end{enumerate}
A change of the positions of finitely many components does not affect the compatibility of an ordering.

Let $\lambda$ be a function that assigns a real number $\lambda(K_i)\geq 0$ to every knot $K_i$.
The function $\lambda$ is called a \emph{length function} for $(K_i)_i$ if $\lambda(K_i)\leq \lambda(K_j)$ holds for all $i<j$
and the number $\pi_\lambda(T)=\#\{i\mid \lambda(K_i)\leq T\}$ is finite for every $T\in \R$.
For a length function $\lambda$, the \emph{induced length} of a component $K'$ of degree $k$ over $K_i$ is defined to be $\lambda(K')=k\,\lambda(K_i)$,
and an ordering $(K'_\nu)_\nu$ is said to be \emph{induced by $\lambda$} if $\lambda(K'_\nu)\leq \lambda(K'_{\nu'})$ holds for all $\nu<\nu'$.
An ordering induced by the length function $\lambda(K_i)=\log i$ is called a \emph{norm ordering};
it is an ordering by increasing norm, with an arbitrary order among the components of equal norm.

\begin{thm}[{\cite{UekiSC}}] \label{thm.branched.general}
Assume that $(K_i)_i$ obeys the Chebotarev law in $M$.
\begin{enumerate}
\item Every compatible ordering of the components of $\varpi^{-1}(\mca{K})$
obeys the Chebotarev law in $M_\varpi$.
\item Let $\lambda$ be a length function for $(K_i)_i$, and assume that
$\#\{i\mid \lambda(K_i)=T\}=o(\pi_\lambda(T))$ and $\pi_\lambda(T/2)=o(\pi_\lambda(T))$ hold as $T\to\infty$.
Then every ordering of the components of $\varpi^{-1}(\mca{K})$ induced by $\lambda$ is compatible.
\end{enumerate}
\end{thm}

The two conditions of \textup{(2)} hold for $\lambda(K_i)=\log i$, so that every norm ordering obeys the Chebotarev law in $M_\varpi$.
They also hold if $\pi_\lambda(T)\sim e^{\delta T}/(\delta T)$ holds as $T\to\infty$ for some $\delta>0$, by the argument of \Cref{lem.ties}.

Following \cite[Definition 14]{Ueki7}, a countable link is said to be \emph{stably Chebotarev}
if its preimage under every finite cover branched along a finite sublink is again a Chebotarev link.
By \Cref{thm.branched.general}, every Chebotarev link is stably Chebotarev,
and the preimage obeys the law for every norm ordering.
The two conditions of \textup{(2)} hold also for $\lambda(K_i)=\log p_i$, where $p_i$ denotes the $i$-th prime number, by the prime number theorem.
The orderings induced by that length function are the orderings by the norm $p_i^{k}$,
which is the form that the theorem of Gropper, Ueki, and Wang requires \cite[Definition 2.11 and Section 4]{GropperUekiWang-NU}.
The proof involves no dynamics: it counts the components of degree one by means of the Chebotarev law for a finite quotient of $\pi_1(M-L_m)$ that contains both the cover and the given homomorphism;
see \cite{UekiSC}, where it is also shown that the law can fail if every component takes the position of its image.

\subsection{The modular knots} \label{subsec.branched.modular}

Let $\mca{K}$ denote the union of the modular knots and of $K_{2,3}$, let $L\subset \mca{K}$ be a finite sublink,
and let $\varpi:M_\varpi\to S^{3}$ be a finite cover branched along $L$.
We may assume that $K_{2,3}\subset L$ holds, since $\varpi$ is also a cover branched along $L\cup K_{2,3}$,
with trivial branching along the added component.
If a component $K'$ of $\varpi^{-1}(C_\gamma)$ covers $C_\gamma$ with degree $k\in \Z_{\geq 1}$,
then the class in $\PSL_2\Z$ attached to $K'$ is that of $\gamma^{k}$,
and the period of $K'$ is $k\,\ell(\gamma)=\ell(\gamma^{k})$,
which is an increasing function of $|\tr \gamma^{k}|$.
The ordering by the periods is therefore also the ordering by the traces of the classes attached to the components.
The $\ell_h$-period of $K'$ is the integral of the lifted roof function along $K'$, and it equals $k\,\ell_h(\gamma)$.

\begin{thm} \label{thm.branched}
Let $(K_i)_i$ be the sequence of \Cref{thm.arith.general} or the sequence of \Cref{thm.modular}, with $K_{2,3}$ inserted at a fixed position.
Then the components of $\varpi^{-1}(\mca{K})$ obey the Chebotarev law in $M_\varpi$ for each of the following orderings.
\begin{enumerate}
\item A norm ordering attached to $(K_i)_i$.
\item For the sequence of \Cref{thm.arith.general}, the ordering by
increasing period of the lifted geodesic flow.
\item For the sequence of \Cref{thm.modular}, the ordering by increasing
$\ell_h$-period.
\end{enumerate}
In \textup{(2)} and \textup{(3)}, the order among the components of equal period is arbitrary,
and the finitely many components of $\varpi^{-1}(L)$ are placed at arbitrary fixed positions.
In particular $\varpi^{-1}(\mca{K})$ is a Chebotarev link.
\end{thm}

\begin{proof}
The sequence $(K_i)_i$ obeys the Chebotarev law in $S^{3}$ by \Cref{thm.arith.general} or by \Cref{thm.modular},
so that \Cref{thm.branched.general} for the length function $\log i$ gives the first ordering.
For the second ordering, put $\lambda(C_\gamma)=\ell(\gamma)$ for a modular knot $C_\gamma$,
and let $\lambda(K_{2,3})$ denote the value of $\lambda$ at the next knot in the sequence.
Then $\lambda$ is a length function for the sequence of \Cref{thm.arith.general},
and the two conditions of \Cref{thm.branched.general} \textup{(2)} hold,
because the asymptotic $\#\{\gamma\mid \ell(\gamma)\leq T\}\sim e^{T}/T$ holds by the proof of \Cref{thm.arith}.
The period of a component over a modular knot is its induced length,
so that the second ordering agrees with an ordering induced by $\lambda$ up to the positions of the finitely many components of $\varpi^{-1}(L)$.
Hence it is compatible by \Cref{thm.branched.general} \textup{(2)}, and it obeys the law by \textup{(1)}.
For the third ordering we argue in the same way with $\ell_h$ in place of $\ell$;
the asymptotic $\#\{\gamma\mid \ell_h(\gamma)\leq T\}\sim e^{\delta T}/(\delta T)$, with $\delta$ the topological entropy of the suspension flow of \Cref{thm.modular},
is the prime orbit theorem of Parry and Pollicott \cite[Theorem 6.9]{ParryPollicott1990} for that suspension flow,
whose closed orbits differ from the modular knots by finitely many, by \Cref{prop.modular} \textup{(4)}.
\end{proof}

\Cref{thm.branched} names, besides the norm orderings,
two orderings of the preimage that are defined by the lifted flow,
one lifted from the trace ordering and one lifted from a generic ordering.
The Chebotarev law alone suffices for the idelic class field theory of \Cref{subsec.rem.icft}.

For the two orderings defined by the lifted flow, there is a second argument.
Let $L_m\subset \mca{K}$ be a finite sublink containing $L$,
and let $\varrho$ be a unitary representation of $\pi_1(M_\varpi-\varpi^{-1}(L_m))$ with finite image.
The $L$-function of the lifted flow for $\varrho$ equals the $L$-function of the modular flow for the representation of $\pi_1(S^{3}-L_m)$ induced from $\varrho$,
because a component of degree $k$ has $k$ times the period of its image;
this is the counterpart of the invariance of the Artin $L$-functions under induction.
By Frobenius reciprocity, the trivial representation has the same multiplicity in both representations,
so that \Cref{lem.hol} and the proof of \Cref{thm.arith} apply to the lifted flow.
For the $\ell_h$-periods, the properties recalled after \Cref{thm.PP} play the same role.

\section{Subgroups of finite index} \label{sec.finiteindex}

The law for a subgroup of finite index is an instance of \Cref{thm.branched}.

\begin{thm}[{\Cref{thmA.finiteindex}}] \label{thm.finiteindex}
Let $\Gamma'\subset \PSL_2\Z$ be a subgroup of finite index, let $M'=T^1(\Gamma'\backslash \bb{H}^2)$,
and let $(K_i)$ denote the closed orbits of the geodesic flow on $M'$, ordered by hyperbolic length.
Then $(K_i)$ obeys the Chebotarev law in $M'$.
\end{thm}

\begin{proof}
The manifold $M'=\Gamma'\backslash \PSL_2\R$ is a finite cover of $M=T^{1}(\PSL_2\Z\backslash \bb{H}^{2})=S^{3}-K_{2,3}$.
We fill the ends of $M'$ by solid tori so that the covering map extends,
and we obtain a finite cover $\varpi\colon M_\varpi\to S^{3}$ branched along $K_{2,3}$ with $M'=M_\varpi-\varpi^{-1}(K_{2,3})$.
The closed orbits of the geodesic flow on $M'$ are the components of the preimages of the modular knots,
and the hyperbolic length of such a closed orbit is its period.
By \Cref{thm.branched} \textup{(2)}, these components obey the Chebotarev law in $M_\varpi$ when they are ordered by the period
and the components of $\varpi^{-1}(K_{2,3})$ are placed at the first positions.
Since $\pi_1(M'-L')=\pi_1(M_\varpi-(L'\cup\varpi^{-1}(K_{2,3})))$ holds for every finite union $L'$ of closed orbits, that is the Chebotarev law in $M'$.
The order among the closed orbits of equal length is arbitrary in \Cref{thm.branched} \textup{(2)}.
\end{proof}

For $\Gamma'$ a congruence subgroup this gives an equidistribution statement for the closed geodesics of a modular curve,
in finite covers of the complement of finitely many of them, ordered by length.
For a representation of $\Gamma'$ itself, Fried \cite[Theorem 2]{Fried1996Invent} constructed a transfer operator for the coding by continued fractions,
by inducing the representation to $\mathrm{PGL}_2\Z$,
and in that case the equidistribution follows also from the Selberg trace formula; see Sarnak \cite[Theorem 2.4]{Sarnak1982JNT}.
\Cref{thm.finiteindex} applies to every finite quotient of $\pi_1(M')$,
including the quotients that detect the removed closed geodesics.

\section{Further remarks} \label{sec.remarks}

\subsection{Idelic class field theory} \label{subsec.rem.icft}

The idelic class field theory of Artin, Takagi,
and Chevalley assembles the local theories at the places of a number field into a global one,
and it describes the abelian extensions, the ramified ones included.
An analogue for 3-manifolds has been built in stages by Niibo \cite{Niibo2014}, by Niibo and Ueki \cite{NiiboUeki2019TAMS},
and by Mihara \cite{Mihara2019Canada}, a countable link playing the role of the set of places.
Niibo and Ueki work with a very admissible link,
and Mihara with a stably generic link, over which the ray class groups are available as well.

By \cite[Theorem 3]{Ueki7}, which is stated for an arbitrary 3-manifold,
a Chebotarev link is stably generic in the sense of Mihara,
and hence very admissible in the sense of Niibo and Ueki.
The results of this article therefore make that theory available in the following two situations;
see also \cite[Chapter 7]{Morishita2024}.
\begin{enumerate}
\item The modular knots together with the missing trefoil $K_{2,3}$ in
$S^{3}$, by \Cref{thm.modular} and by \Cref{thm.arith.general}.
\item The closed geodesics of \Cref{thm.finiteindex}, and the preimage of
any of the links above under a finite branched cover, by \Cref{thm.branched}.
\end{enumerate}

\subsection{Supersingular primes} \label{subsec.rem.supersingular}

Let $p$ be a prime number, let $K$ be a knot in $S^{3}$,
and let $(M_n\to S^{3})_n$ be the $\Z/n\Z$-covers branched along $K$.
It is proved by Ueki and Yoshizaki \cite{UekiYoshizaki2025ProcCambridge} that the sizes $|H_1(M_{p^{n}})_{\rm tor}|$ converge in $\Z_p$,
and a knot whose limit is $1$ is an analogue of a supersingular prime of a fixed elliptic curve.
For elliptic curves over $\Q$ the supersingular primes are infinite in number by a theorem of Elkies \cite{Elkies1987Invent},
their density among the primes is zero for a curve without complex multiplication by a theorem of Serre \cite{Serre1981PMIHES},
and the conjecture of Lang and Trotter predicts the precise asymptotics.
We now have a family of knots that is Chebotarev for the trace ordering.
One may therefore ask for the density of those modular knots whose $p$-adic limit is $1$, the knots being ordered by $|\tr\gamma|$.
That question is not of Chebotarev type, since it does not ask that a Frobenius class lie in a fixed conjugacy class of a fixed finite quotient, and we do not know the answer.

\subsection{Reconstruction} \label{subsec.rem.NU}

The theorem of Neukirch and Uchida \cite[Chapter XII]{NSW} reconstructs a number field from its absolute Galois group.
Its analogue for a 3-manifold endowed with a stably Chebotarev link is studied by Gropper, Ueki, and Wang \cite{GropperUekiWang-NU},
where the Hilbert ramification theory and the Chebotarev density theorem play the roles that they play in the arithmetic proof.
By \Cref{thm.branched.general}, applied to the sequences of \Cref{thm.arith.general} and of \Cref{thm.modular}, the union of the modular knots and of $K_{2,3}$ is such a link,
and it therefore provides an example in $S^{3}$.

\subsection{Three questions} \label{subsec.rem.questions}

We end with three questions that the results suggest.

\subsubsection{The bias} \label{subsubsec.q.bias}
The residue classes of the primes modulo $m$ are equidistributed,
and the race between them shows the bias of Chebyshev,
a fluctuation of logarithmic size that is governed, under the generalized Riemann hypothesis,
by the zeros of the Dirichlet $L$-functions \cite{RubinsteinSarnak1994}.

For the modular knots and the trace ordering the bias is larger and deterministic.
Let $N(T)$ denote the number of primitive hyperbolic conjugacy classes $\gamma$ with $\ell(\gamma)\leq T$,
and let $N_k(T)$ denote the number of those with $\Psi(\gamma)\equiv k\pmod m$.
The prime geodesic theorem twisted by $e^{2\pi i\theta\Psi}$,
due to Sarnak \cite{Sarnak2010CMA} and Mozzochi \cite{Mozzochi2013Israel}, gives
\[\sum_{\ell(\gamma)\leq T}e^{2\pi i\theta\Psi(\gamma)}
=\operatorname{Li}\bigl(e^{(1-6|\theta|)T}\bigr)+O(e^{3T/4})
\qquad(0<|\theta|\leq 1/24);\]
the exponent $1-6|\theta|$ is the one of \Cref{rem.numerics}.
The same sum is $O(e^{3T/4})$ for $1/24<|\theta|\leq 1/2$,
and the eigenvalue $6|\theta|(1-6|\theta|)$ of the Laplacian of weight $12\theta$ is the only one below $1/4$,
by a theorem of Bruggeman quoted in \cite{Sarnak2010CMA}.
Summing over the characters of $\Z/m\Z$, we obtain
\[N_k(T)-\frac{N(T)}{m}
=\frac{2}{m}\sum_{1\leq j\leq m/24}\cos\frac{2\pi jk}{m}\,
\operatorname{Li}\bigl(e^{(1-6j/m)T}\bigr)+O(e^{3T/4}).\]
If $m>24$, then the term $j=1$ exceeds the error term,
so the class $\Psi\equiv 0$ is ahead of every other class by a power of $e^{T}$.
The bias comes from that eigenvalue, which plays the role of a Siegel zero.

For the combinatorial ordering the count in each length is given exactly by the binomial formula of \Cref{subsec.ord.laws}:
for odd $m$, the number of words of length $n$ with $\Psi\equiv k\pmod m$ differs from $2^{n}/m$ by
\[\frac1m\sum_{j=1}^{m-1}e^{-2\pi ijk/m}\bigl(2\cos(2\pi j/m)\bigr)^{n},\]
whose largest terms are those with $j=(m\pm1)/2$, of absolute value $(2\cos(\pi/m))^{n}$.
Hence for $m=3$ the bias is bounded, while for odd $m\geq 5$ the class $\Psi\equiv 0$ is ahead of every other class by a power of $2^{n}$ in the even lengths and behind every other class in the odd lengths.

For the generic ordering, and for the linking numbers with a modular knot $C_\sigma$ in all three orderings,
the bias is governed by the eigenvalues of the twisted transfer operators other than the leading one,
about which we know nothing beyond \Cref{lem.non-vanishing} and \Cref{lem.twisted}.
A computation for $\sigma=RL$ and $m=3$ shows counts in each combinatorial length that oscillate with a period of about eight or nine lengths,
which indicates a pair of complex eigenvalues.
What are these eigenvalues, and which classes are ahead?

\subsubsection{Pairs} \label{subsubsec.q.pairs}
Ghys \cite{Ghys2007ICM} asked for an interpretation of the linking number of two modular knots,
and \Cref{cor.linking} describes its statistics when one of the knots is fixed.
Simon \cite[Theorem 7.4]{Simon2025GT} proved that the linking numbers with all the modular knots determine a modular knot.
Two questions arise.
The first concerns the pairs $(\gamma,\sigma)$ with $\ell(\gamma),\ell(\sigma)\leq T$ and $\operatorname{lk}(C_\gamma,C_\sigma)\equiv k\pmod m$,
in which both knots vary: do they have a density?
This requires an error term in \Cref{cor.linking} that is uniform in $\sigma$.
Our argument does not give one, since the refinement of \Cref{subsec.refine} depends on $\sigma$;
compare \cite{IshidaKuramotoZheng2025}, where the analogous uniformity for the primes comes from the large sieve.
The second question concerns the limit law of $\operatorname{lk}(C_\gamma,C_\sigma)$ itself,
for fixed $\sigma$, after subtracting its mean,
which grows linearly since the terms of \eqref{eq.simon} are non-negative,
and dividing by $\sqrt{\ell_{\rm c}(\gamma)}$ for the combinatorial ordering and by $\sqrt{\ell(\gamma)}$ for the trace ordering.
By \eqref{eq.simon} and \Cref{lem.link.cusp},
each block of equal letters of $\gamma$ contributes a bounded amount,
in contrast with the linking number with $K_{2,3}$,
so we expect a Gaussian law for both orderings, and no Cauchy law as in \Cref{thm.cauchy}.
Coles and Sharp \cite[Theorem 6.3 and Corollary 6.4]{ColesSharp2024GD} treated a transitive Anosov flow on a closed 3-manifold
whose roof function cannot be chosen locally constant.
They determined the asymptotic number of the closed orbits whose class in the homology of the complement
of finitely many null-homologous closed orbits is fixed or grows linearly in $T$.
If a closed orbit is null-homologous in the 3-manifold, then that class determines its linking numbers with those orbits.
The power of $T$ in their formula is that of a Gaussian local limit law.

\subsubsection{The $L$-functions} \label{subsubsec.q.lfunctions}
When no modular knot has been removed, the $L$-function $L(s,\varrho)$ of \Cref{prop.fredholm} is a Selberg zeta function twisted by a representation of $\PSL_2\Z$,
and its zeros in $1/2<\mathrm{Re}\,s<1$ correspond to the eigenvalues of the twisted Laplacian.
After the removal of modular knots, $\rho$ does not factor through $\PSL_2\Z$,
and $L(s,\varrho)$ is defined only through the Fredholm determinants of the transfer operators.
By \Cref{lem.non-vanishing}, $L(s,\varrho)$ has neither zero nor pole on the line $\mathrm{Re}\,s=1$ when $\varrho$ is irreducible and nontrivial,
and nothing is known about its zeros in the strip $1/2<\mathrm{Re}\,s<1$.
Is there an operator, in place of the Laplacian, whose spectrum the zeros describe,
and does $L(s,\varrho)$ satisfy a functional equation?

\section*{Acknowledgments}

I would like to express my sincere gratitude to Masayuki Asaoka, Nadav Gropper, Atsushi Katsuda, Hirokazu Maruhashi, Toshiki Matsusaka, Masanori Morishita, Kimihiko Motegi, Yasuharu Nakae, Hirofumi Niibo, Yuta Nozaki, Tali Pinsky, Masato Tsujii, Naoya Umezaki, Yi Wang,
and Hyuga Yoshizaki, from whom I have learned much about low-dimensional topology, hyperbolic dynamics,
and arithmetic topology, and whose questions have shaped this work over the years, including over its earlier versions.

In particular, Toshiki Matsusaka introduced me to the modular knots,
and the first form of this article grew out of our joint work on the Rademacher symbols.
The question of Naoya Umezaki, whether the conclusion survives a change of the ordering, is the origin of \Cref{thm.two.densities}.
Tali Pinsky answered my queries on the Lorenz knots and on the completions of the modular flow.

I am grateful to the anonymous referees and to the editors for their careful reading and helpful comments on earlier versions of this article.
The author has been partially supported by JSPS KAKENHI Grant Numbers JP19K14538 and JP23K12969.

\subsection*{AI Declaration}
The earlier versions of this article, which have been on arXiv since 2021, were written without AI assistants.
The present version was prepared between August and November 2026 in active collaboration with Claude (Anthropic; models Fable 5, 5.1, and 5.2 and Opus 5.5, accessed through claude.ai).
The author worked with Claude on the formulation of statements and of proofs, on the numerical experiments, and on the drafting and checking of the manuscript.
The proposal to treat the ordering by the word length and the ordering by the trace as two separate problems was made by Claude at the beginning of this collaboration.
The criterion of \Cref{thmA.combinatorial} for the combinatorial ordering was first suggested by a numerical experiment that Claude carried out.
Everything so obtained has been verified independently by the author, who takes full responsibility for the contents of this article.

{\small 
\bibliographystyle{amsalpha}
\bibliography{ju_mkc1_arXiv.bbl}
}

\medskip
\noindent Department of Mathematics, Ochanomizu University,
2-1-1 Otsuka, Bunkyo-ku, Tokyo 112-8610, Japan

\noindent\textit{E-mail address}: \texttt{uekijun46@gmail.com}

\internal{\emph{Relation to the earlier versions.}\ 
The earlier versions of this article deduced the law from the theorem of McMullen \cite{McMullen2013CM}.
The first four versions on the arXiv did so by filling the cusp, and they announced the corresponding statement for the modular knots around an arbitrary torus knot.
Two points of that route do not survive a closer examination.
A filling of the cusp of $T^{1}(\PSL_2\Z\backslash \bb{H}^{2})$ does not produce a pseudo-Anosov flow with the missing trefoil as a closed orbit;
by Bonatti and Pinsky \cite{BonattiPinsky2021} the trefoil appears as a union of heteroclinic orbits of an Axiom A flow,
and the corresponding assertion for the triangle groups is not correct either.
Moreover, and independently of that, a completion of the flow is related to the geodesic flow by a topological equivalence, which does not preserve lengths,
so that a density with respect to the length of the completed flow says nothing about the density with respect to $|\tr \gamma|$;
\Cref{thm.two.densities} makes that failure quantitative.

\medskip

The fifth version on the arXiv no longer filled the cusp.
It derived the law for the generic ordering, for the unit tangent bundle of an arbitrary Fuchsian group with cusps,
by inserting into the proof of McMullen the Markov partition that \cite[Remark 6.7]{Tsang2023NY} announces for hyperbolic surfaces with cusps.
That construction concerns the double of the surface;
a Markov section of the flow on the unit tangent bundle itself cannot have compact rectangles, by \Cref{prop.no.finite.markov}.
The statement for an arbitrary Fuchsian group with cusps is not proved.

\medskip

The present version does not go through the theorem of McMullen.
It uses a section built by hand from the Farey tessellation, and it treats the trace ordering by a transfer operator argument.
The modular knots around torus knots are treated in a separate article on the triangle groups,
and the planetary links, together with the generalized pseudo-Anosov flows and the surgeries of Fried, in a separate article on the planetary links.
}

\end{document}